\documentclass[11pt,a4paper]{amsart}
\usepackage[T1]{fontenc}
\usepackage[english]{babel}

\usepackage{amsmath,mathtools}
\usepackage{amssymb}
\usepackage{amsthm}
\usepackage{mathrsfs}
\usepackage{stmaryrd}

\usepackage{tikz-cd}				
\usetikzlibrary {arrows.meta,backgrounds,fit,positioning}
\usepackage{extarrows}

\usepackage{tikz} 				
\usetikzlibrary{decorations.pathmorphing}
\tikzstyle{vertex}=[circle, draw, inner sep=0pt, minimum size=6pt]

\usepackage{fancyhdr}
\usepackage{xcolor}
\usepackage{graphicx}
\usepackage{enumitem}
\setlist[enumerate,1]{label=(\arabic*)}
\usepackage{caption}
\usepackage{bbold}

\usepackage[hidelinks]{hyperref}

\usepackage[capitalise,nosort]{cleveref}

\usepackage[margin=.9in]{geometry}

\numberwithin{equation}{subsection}

\newtheorem{Theorem}[equation]{Theorem}
\newtheorem{TheoremA}{Theorem}

\newtheorem{Corollary}[equation]{Corollary}
\newtheorem{Proposition}[equation]{Proposition}
\newtheorem{Lemma}[equation]{Lemma}

\theoremstyle{definition}
\newtheorem{Definition}[equation]{Definition}
\newtheorem{Notation}[equation]{Notation}
\newtheorem{Example}[equation]{Example}

\newtheorem{Construction}[equation]{Construction}

\theoremstyle{plain}

\theoremstyle{remark}
\newtheorem{Remark}[equation]{Remark}

\crefname{Lemma}{Lemma}{Lemmas}
\crefname{Theorem}{Theorem}{Theorems}
\crefname{TheoremA}{Theorem}{Theorems}
\crefname{Definition}{Definition}{Definitions}
\crefname{Notation}{Notation}{Notations}
\crefname{Proposition}{Proposition}{Propositions}
\crefname{Remark}{Remark}{Remarks}
\crefname{Corollary}{Corollary}{Corollaries}
\crefname{Example}{Example}{Examples}
\crefname{Construction}{Construction}{Constructions}

\newcommand{\Fib}{\mathrm{Fib}}             
\newcommand{\twoFib}{2\mathcal{F}\mathrm{ib}} 
\newcommand{\oneDFib}{\mathrm{DFib}}          
\newcommand{\DblFib}{\mathrm{DblFib}} 

\renewcommand{\epsilon}{\varepsilon}
\renewcommand{\phi}{\varphi}

\DeclareMathOperator{\id}{id}

\DeclareMathOperator{\Prof}{Prof}

\renewcommand{\AA}{\mathcal{A}}

\newcommand{\CC}{\mathcal{C}}
\newcommand{\DD}{\mathcal{D}}

\newcommand{\EE}{\mathcal{E}}

\newcommand{\Cat}{\mathrm{Cat}}
\newcommand{\Set}{\mathrm{Set}}

\newcommand{\dC}{\mathbb{C}}
\newcommand{\dD}{\mathbb{D}}
\newcommand{\dE}{\mathbb{E}}
\newcommand{\dF}{\mathbb{F}}
\newcommand{\dH}{\mathbb{H}}
\newcommand{\dV}{\mathbb{V}}
\newcommand{\HH}{\mathbf{H}}

\newcommand{\fib}{\mathrm{fib}}
\newcommand{\TSdiscFib}{\mathrm{TSFib}} 
\newcommand{\dFib}{\mathrm{D}\mathcal{F}\mathrm{ib}}
\newcommand{\dSet}{\mathbb{S}\mathrm{et}}
\newcommand{\dCat}{\mathbb{C}\mathrm{at}}
\newcommand{\dSpan}{\mathbb{S}\mathrm{pan}}
\newcommand{\DblCat}{\mathrm{DblCat}}
\DeclareMathOperator{\lax}{lax}
\DeclareMathOperator{\nlax}{nlax}

\newcommand{\dHom}[1]{\llbracket #1 \rrbracket}
\newcommand{\dHomLax}[1]{\llbracket #1 \rrbracket^{\lax}}
\newcommand{\dHomnLax}[1]{\llbracket #1 \rrbracket^{\nlax}}
\DeclareMathOperator{\op}{op}
\DeclareMathOperator{\co}{co}
\DeclareMathOperator{\coop}{coop}
\newcommand{\gro}{\textstyle{\int}}

\renewcommand{\P}[1]{\mathcal{P}_{\mathbb{#1}}} 

\newcommand{\dgro}{\int\!\!\!\!\int}
\newcommand{\ddel}{\partial\!\!\partial}
\newcommand{\bulletarrow}{{\longrightarrow\!\!\!\!\!\!\!\!\bullet\ \ }}
\newcommand{\Bulletarrow}{{\Longrightarrow\!\!\!\!\!\!\!\!\bullet\ \ }}
\newcommand{\edgesquare}[4]{\left[{#1}^{\ {#2}\ }_{\ {#3} }{#4}\right]}
\newcommand{\nodesquare}[4]{\left[{}^{{#1}}_{{#2}}\ {}^{{#3}}_{{#4}}\right]}

\DeclareMathOperator{\Ver}{Ver}
\DeclareMathOperator{\Hor}{Hor}

\title[Yoneda lemma and representation theorem
for double categories]{Yoneda lemma and representation theorem \\
for double categories}

\author{Benedikt Fr\"ohlich}
\address{Fakultät für Mathematik, Universität Regensburg, Regensburg, Germany}
\email{benedikt.froehlich@stud.uni-regensburg.de}

\author{Lyne Moser}
\address{Fakultät für Mathematik, Universität Regensburg, Regensburg, Germany}
\email{lyne.moser@ur.de}

\begin{document}

\begin{abstract}
    We study (vertically) normal lax double functors valued in the weak double category $\mathbb{C}\mathrm{at}$ of small categories, functors, profunctors and natural transformations, which we refer to as \emph{lax double presheaves}. We show that for the theory of double categories they play a similar role as 2-functors valued in $\mathrm{Cat}$ for 2-categories.
    We first introduce representable lax double presheaves and establish a Yoneda lemma. 
    Then we build a Grothendieck construction which gives a 2-equivalence between lax double presheaves and discrete double fibrations over a fixed double category.
    Finally, we prove a representation theorem showing that a lax double presheaf is represented by an object if and only if its Grothendieck construction has a double terminal object.
\end{abstract}
\maketitle

\setcounter{tocdepth}{1}
\tableofcontents

\section{Introduction}

\subsection*{The classical story}

Universal constructions are the core of the study of category theory. Such constructions are determined by universal properties obtained by requiring that a certain presheaf is representable. A \emph{presheaf} on a category $\CC$ is a functor $X\colon \CC^{\op}\to \Set$ into the category of sets and maps, and it is said to be \emph{represented} by an object $\hat{x}$ in $\CC$ if it is isomorphic to the representable functor $\CC(-,\hat{x})\colon \CC^{\op}\to \Set$ taking as values the hom sets of~$\CC$ with target $\hat{x}$. Presheaves therefore play an important role in category theory and, for this reason, they have been extensively studied. 

One of the most fundamental results in ordinary category theory is the Yoneda lemma. This result says that, given a presheaf $X\colon \CC^{\op}\to \Set$, natural transformations $\CC(-,\hat{x})\Rightarrow X$ are completely determined by a single element of the set $X\hat{x}$, for every object $\hat{x}$ in $\CC$. More explicitly, this says that we have an isomorphism of sets 
\[ [\CC^{\op},\Set](\CC(-,\hat{x}),X)\cong X\hat{x}; \]
see e.g.~\cite[Theorem 2.2.4]{Context}. As a consequence, we obtain a Yoneda embedding $\CC\hookrightarrow \Set^{\CC^{\op}}$, allowing us to infer universal constructions in $\CC$ from an analogous setting in presheaves, where a very explicit computation of such constructions is possible.

Another approach to presheaves is given by studying them from a fibrational perspective through the Grothendieck construction. This construction encodes the structure of a presheaf $X\colon \CC^{\op}\to \Set$ in a different way by compressing all of its data into a single category $\int_\CC X$ coming with a canonical projection onto $\CC$. This projection $\int_\CC X\to \CC$ is the prototypical example of a discrete fibration. A \emph{discrete fibration} is a functor $P\colon \EE\to \CC$ such that every morphism in $\CC$ of the form $f\colon x\to Pe$ admits a unique lift $e'\to e$ in $\EE$. In fact, the Grothendieck construction induces an equivalence 
\[ \textstyle\int_\CC \colon \Set^{\CC^{\op}}\xrightarrow{\simeq} \Fib(\CC)\subseteq \Cat/\CC, \]
where $\Fib(\CC)$ denotes the full subcategory of the slice $\Cat/\CC$ of categories over $\CC$ spanned by the discrete fibrations; see e.g.~\cite[Theorem 2.1.2]{CatFibrations}.

A powerful consequence of this equivalence is that one can derive another criterion to detect representability of a presheaf by testing whether its Grothendieck construction has a terminal object; see e.g.~\cite[Proposition 2.4.8]{Context}. We refer to this result as the \emph{representation theorem}.

\subsection*{A drama in $2$-categories}

The categories that one considers in practice often have more structure. Typically, the category $\Cat$ of categories and functors further comes with a notion of natural transformations between its morphisms. This is an example of a \emph{$2$-category}, which, in addition to objects and morphisms, also have $2$-morphisms between their morphisms. In this context, morphisms between two objects in a $2$-category $\CC$ now form a category rather than a set, and so the representable presheaves of a $2$-category take values in the $2$-category $\Cat$. We refer to $2$-functors $\CC^{\op}\to \Cat$ as \emph{$2$-presheaves}. This allows for a refined version of universal properties by requiring that certain $2$-presheaves $X\colon \CC^{\op}\to \Cat$ are \emph{represented} by an object $\hat{x}$ in $\CC$ in this higher sense, i.e., there is an isomorphism between $X$ and the representable $2$-presheaf $\CC(-,\hat{x})\colon \CC^{\op}\to \Cat$. This gives a better-behaved definition of universal constructions in the $2$-categorical context, as these are now compatible with the higher structure of the $2$-categories involved.

One can then formulate a $2$-categorical version of the Yoneda lemma by upgrading isomorphisms of sets into isomorphisms of categories. It says that, given a $2$-presheaf $X\colon \CC^{\op}\to \Cat$ and an object $\hat{x}$ in $\CC$, there is an isomorphism of categories 
\[ [\CC^{\op},\Cat](\CC(-,\hat{x}),X)\cong X\hat{x}; \]
see e.g.~\cite[Lemma 8.3.16]{2Dim_Categories}. Again, this yields a Yoneda embedding $\CC\hookrightarrow [\CC^{\op},\Cat]$, which shows how the role of $\Set$ has been replaced by $\Cat$ in the $2$-categorical world. 

On the other hand, a first $2$-categorical version of the Grothendieck construction was introduced by Buckley. As shown in \cite[Theorem 2.2.11]{Buckley}, this Grothendieck construction induces a $2$-equivalence of $2$-categories 
\[ \textstyle\int_\CC\colon [\CC^{\op},\Cat]\xrightarrow{\simeq} \twoFib(\CC)\subseteq 2\Cat/\CC, \] 
where $\twoFib(\CC)$ denotes the $2$-subcategory of the slice $2\Cat/\CC$ of $2$-categories onto $\CC$ spanned by certain \emph{$2$-fibrations}; see \cite[Definition 2.1.10]{Buckley}. While this gives a fibrational perspective on $2$-presheaves, there is a priori no obvious generalization of the representation theorem in this setting. Indeed, as clingman and the second author show in \cite{2Limits}, it is not true that a $2$-presheaf~$X$ is represented by an object if and only if its Grothendieck construction $\int_\CC X$ admits a $2$-terminal object. Instead, one would need to generalize results by Gagna--Harpaz--Lanari \cite{Gagna_Harpaz_Lanari} to obtain a more appropriate characterization. 

\subsection*{The savior: double categories}

However, passing to a double categorical context allows for such a characterization. Double categories are another type of $2$-dimensional categorical structure, namely, the internal categories to $\Cat$, which have objects, two different kinds of morphisms between objects---the \emph{horizontal} and \emph{vertical} morphisms---, and $2$-dimensional morphisms called \emph{squares}. In particular, every $2$-category can be seen as a double category with only trivial vertical morphisms. Then, Grandis--Par\'e introduce in \cite[\textsection 1.2]{persistentflexible} a double categorical Grothendieck construction which takes a $2$-presheaf $X\colon \CC^{\op}\to \Cat$ to a double category $\dgro_\CC X$ coming with a canonical projection onto $\CC$ (seen as a double category). Analogously to ordinary categories, this projection $\dgro_\CC X\to \CC$ is a \emph{discrete double fibration}, i.e., the induced functors on the categories of objects and horizontal morphisms, and on the categories of vertical morphisms and squares are discrete fibrations of categories. Then, in \cite[Theorem 5.1]{internal_GC}, the second author with Sarazola--Verdugo show that this double Grothendieck construction also induces a $2$-equivalence of $2$-categories 
\[\textstyle \dgro_\CC\colon [\CC^{\op},\Cat]\xrightarrow{\simeq} \dFib(\CC)\subseteq \DblCat_v/\CC, \]
where $\dFib(\CC)$ denotes the $2$-full $2$-subcategory of $\DblCat_v/\CC$ spanned by the discrete double fibrations. Here, $\DblCat_v$ denotes the $2$-category of double categories, double functors, and vertical transformations; see 
\cref{not: 2-cat DblCat}.

Using this double Grothendieck construction, a representation theorem can now be formulated by saying that a $2$-presheaf is represented by an object if and only if its double Grothendieck construction admits a double terminal object; see \cite[Theorem 6.8]{clingmanmoser2022limits} and \cite[Theorem 6.12]{internal_GC}. Since these results already require the language of double categories, one might wonder whether the whole situation extends to double categories and a suitable notion of double presheaves.

\subsection*{The story continues for double categories}

In the double categorical world, Par\'e \cite{Yoneda_double_cats} and Fiore--Gambino--Kock \cite{fioregambinokock} prove a Yoneda lemma for ``set-valued'' double presheaves. In this context, the target of the double presheaves is given by the (weak) double category $\dSet$ of sets, maps, spans, and maps of spans. Given a double category $\dC$, the representable presheaf at an object $\hat{x}$ in  $\dC$ is then given by associating to each object $x$ in $\dC$ the set of horizontal morphisms $x\to \hat{x}$ in $\dC$. Moreover, to each vertical morphism $u\colon x\bulletarrow x'$ in $\dC$, one can associate a span of sets whose middle object is the set of all squares in $\dC$ of the form
\begin{equation} \label{square intro}
    \begin{tikzcd}
        x \arrow[d,"\bullet" marking,"u"'] \arrow[r, "f"] \arrow[rd, "\alpha", phantom] & \hat{x} \arrow[d,"\bullet" marking, equal] \\
        x' \arrow[r, "f'"']  & \hat{x}  
    \end{tikzcd}
\end{equation}
However, such an assignment fails to be functorial in vertical morphisms, and therefore this forces the correct notion of double presheaves to be lax in the vertical direction. Hence a \emph{lax double presheaf} is a lax double functor $\dC^{\op}\to \dSet$. Using this notion, the Yoneda lemma proven as \cite[Theorem 2.3]{Yoneda_double_cats} and \cite[Proposition 3.10]{fioregambinokock}, respectively, shows that, for every object $\hat{x}$ in $\dC$, there is an isomorphism of sets
\[ [\dC^{\op},\dSet]^{\lax}(\dC(-,\hat{x}),X)\cong X\hat{x}. \]
However, since those presheaves only detect sets of morphisms, when taking $\dC=\CC$ to be a $2$-category, there is no obvious way on how to retrieve the $2$-categorical Yoneda lemma from this statement.  

In this paper, we want to enhance the structure of those double presheaves by considering ``category-valued'' double presheaves instead. For this, we consider the (weak) double category $\dCat$ of categories, functors, profunctors, and natural transformations, in place of $\dSet$. Then, the representable presheaf at an object $\hat{x}$ in $\dC$ will now associate to each object $x$ in $\dC$ the category $\dC(x,\hat{x})$ of horizontal morphisms $x\to \hat{x}$ and globular squares in $\dC$ of the form
\[
    \begin{tikzcd}
        x \arrow[d,"\bullet" marking,equal] \arrow[r, "f"] \arrow[rd, "\alpha", phantom] & \hat{x} \arrow[d,"\bullet" marking, equal] \\
        x \arrow[r, "f'"']  & \hat{x}  
    \end{tikzcd}
\]
Then, similarly to before, to each vertical morphism $u\colon x\bulletarrow x'$ in $\dC$, one can associate a profunctor which assigns to any two horizontal morphisms $f\colon x\to \hat{x}$ and $f'\colon x'\to \hat{x}$, the set of squares of the form \eqref{square intro}. While this assignment is still not compatible with vertical composition (see \cref{rem: repr. not natural in vert. morphisms}), by adding the extra structure of a category to the values of the representable presheaf, it now becomes \emph{normal}, i.e., it preserves vertical identities strictly. Therefore, we define \emph{lax double presheaves} as normal lax double functors $\dC^{\op}\to \dCat$.

We then obtain the following Yoneda lemma, which appears as \cref{thm:Yoneda}. Here we denote by 
$[\dC^{\op},\dCat]^{\nlax}$ the $2$-category of (normal) lax double presheaves, horizontal transformations, and globular modifications; see \cref{not: PC}.

\begin{TheoremA}\label{intro: Yoneda}
    Given a double category $\dC$, an object $\hat{x}$ in $\dC$, and a lax double presheaf $X\colon\dC^{\op}\to\dCat$, there is an isomorphism of categories
    \[
   [\dC^{\op},\dCat]^{\nlax}(\dC(-,\hat{x}),X)\cong X\hat{x},
    \]
    which is $2$-natural in $\hat{x}$ and $X$.
\end{TheoremA}

This gives a $2$-categorical enhancement of the Yoneda lemma for double categories shown by Paré and Fiore--Gambino--Kock. Moreover, by taking $\dC=\CC$ to be a $2$-category, we retrieve as a straightforward consequence the Yoneda lemma for $2$-categories mentioned above, as shown in \cref{Yoneda for 2-categories}.

Continuing the story, Par\'e introduces in \cite[\textsection 3.7]{Yoneda_double_cats} a Grothendieck construction for ``set-valued'' lax double presheaves $X\colon\dC^{\op}\to\dSet$ in the form of double categories $\dgro_\dC X$ with a canonical projection onto $\dC$. Lambert then shows in \cite[Lemma 2.10]{discrete_double_fibrations} that this projection $\dgro_\dC X\to \dC$ is a discrete double fibration, and in \cite[Theorem 2.27]{discrete_double_fibrations} that the Grothendieck construction induces an equivalence of categories 
\begin{equation} \label{Lambert eq}
\textstyle
\dgro_\dC\colon [\dC^{\op},\dSet]^{\lax}\xrightarrow{\simeq}\oneDFib(\dC)\subseteq \DblCat/\dC, 
\end{equation}
where $\oneDFib(\dC)$ denotes the full ($1$-)subcategory of $\DblCat/\dC$ spanned by the discrete double fibrations.

In this paper, we extend Paré's Grothendieck construction to our ``category-valued'' lax double presheaves. Given a lax double presheaf $X\colon\dC^{\op}\to\dCat$, we also build in \cref{def: dgro} a double category $\dgro_\dC X$ coming with a canonical projection onto $\dC$. Surprisingly, despite the fact that $X$ is a (normal) lax double functor rather than a strict one, its Grothendieck construction $\dgro_\dC X$ is an actual strict double category; note that this already happens for Par\'e's Grothendieck construction (see \cite[Theorem 3.8]{Yoneda_double_cats}). Even though we now consider presheaves valued in categories rather than in sets, the projection $\dgro_\dC X\to\dC$ is still a discrete double fibration. 

Furthermore, by extending the Grothendieck construction to a $2$-functor, we can then show the following equivalence with discrete double fibrations, which appears as \cref{thm: grothendieck equivalence}.

\begin{TheoremA}\label{intro: Grothendieck equivalence}
    Given a double category $\dC$, the Grothendieck construction induces a $2$-equivalence of $2$-categories
    \[
    \textstyle\dgro_\dC\colon [\dC^{\op}, \dCat]^{\nlax} \xrightarrow{\simeq}\dFib(\dC)\subseteq \DblCat_v/\dC,
    \]
    which is pseudo-natural in $\dC$.
\end{TheoremA}

Again, by taking $\dC=\CC$ to be a $2$-category, we retrieve as a straightforward consequence the $2$-equivalence between $2$-presheaves and discrete double fibrations from \cite{internal_GC}, as shown in \cref{GC for 2-categories}.

There is now an intriguing relation between Lambert's equivalence from \eqref{Lambert eq} and our $2$-equivalence from \cref{intro: Grothendieck equivalence}. At its heart is the fact that the fibers of a discrete double fibration are categories---seen as double categories with trivial horizontal morphisms and squares---; see \cref{prop: Px category}. Therefore, to obtain an equivalence with a certain type of double presheaves, one needs to encode both the objects and the morphisms of those fibers into the presheaf. While the inverse of the equivalence \eqref{Lambert eq} only detects the underlying sets of objects of the fibers as the values of the presheaf, the morphisms are encoded in the lax unitality condition. On the other hand, the inverse of the equivalence from \cref{intro: Grothendieck equivalence} directly chooses the values of the presheaf to be the categories given by the fibers, and the lax unitality condition now becomes strict. In fact, the relation between the $2$-categories $[\dC^{\op},\dCat]^{\nlax}$ and $[\dC^{\op},\dSet]^{\lax}$ can be understood from the universal property of the adjunction developed by Cruttwell--Shulman in \cite[Proposition 5.14]{cruttwellshulman}, using that $\mathbb{M}\mathrm{od}(\dSet)=\dCat$.

The idea of enhancing ``set-valued'' lax double presheaves to ``category-valued'' ones is already studied in a paper by Cruttwell--Lambert--Pronk--Szyld, where they consider lax double \emph{pseudo} functors $\dC^{\op}\to\dSpan(\Cat)$ into the (weak) double category $\dSpan(\Cat)$ of categories, functors, spans of categories, and morphisms of spans. For such presheaves, they show in \cite[Theorem 3.45]{Double_Fibrations} that there is an equivalence of categories
\[
[\dC^{\op},\dSpan(\Cat)]^{\lax}\xrightarrow{\simeq}\DblFib(\dC)\subseteq \DblCat/\dC,
\]
where $\DblFib(\dC)$ denotes the subcategory of $\DblCat/\dC$ spanned by the \emph{double fibrations}; see \cite[Definition 3.38]{Double_Fibrations}. These double fibrations generalize discrete double fibrations in the same vein that Grothendieck fibrations generalize discrete fibrations for
ordinary categories, meaning that their fibers are double categories instead of categories. By restricting the above equivalence to discrete double fibrations, they retrieve Lambert's equivalence \eqref{Lambert eq}. Hence, this suggests that the approach taken in this paper is different from theirs.

Finally, we use the Grothendieck equivalence from \cref{intro: Grothendieck equivalence} to prove a representation theorem for double categories. Namely, we want to find a criterion for a lax double presheaf $X\colon\dC^{\op}\to\dCat$ to be \emph{represented} by an object $\hat{x}$ of $\dC$, i.e., isomorphic to the representable lax double presheaf $\dC(-,\hat{x})\colon \dC^{\op}\to \dCat$. As for the case of $2$-categories, these will be determined by double terminal objects in the Grothendieck construction. The following result appears as \cref{thm: representability theorem}.

\begin{TheoremA}\label{intro: representation theorem}
    A lax double presheaf $X\colon\dC^{\op}\to\dCat$ is represented by an object if and only if its Grothendieck construction $\dgro_\dC X$ has a double terminal object.
\end{TheoremA}

Again, by choosing $\dC=\CC$ to be a $2$-category, we retrieve the representation theorem for $2$-categories from \cite{clingmanmoser2022limits,internal_GC}, as shown in \cref{Representation for 2-categories}.

\subsection*{Notations}

Throughout the paper, we will assume basic knowledge about $2$-category theory. We refer the reader to \cite{2Dim_Categories} for a complete account of the theory of $2$-categories. We use the following notations:
\begin{itemize}[leftmargin=1cm]
    \item we write $\Set$ for the category of sets and maps, 
    \item we write $\Cat$ for the ($2$-)category of categories, functors, (and natural transformations),
    \item we write $2\Cat$ for the $2$-category of $2$-categories, $2$-functors, and $2$-natural transformations. 
\end{itemize}

\subsection*{Acknowledgment}

We would like to thank Denis-Charles Cisinski, Johannes Glossner, and Nima Rasekh for helpful discussions related to the subject of this paper. We are also grateful to Nathanael Akror and David Kern for insightful comments on a first draft of this paper, and in particular for pointing to the result of Cruttwell--Shulman \cite{cruttwellshulman} shedding light on the relation between normal lax double functors into $\dCat$ and lax double functors into $\dSet$. We also thank the anonymous referee for their careful reading, their insightful comments, and for catching many typos.

During the realization of this work, the second author was a member of the Collaborative Research Centre ``SFB 1085: Higher
Invariants'' funded by the Deutsche Forschungsgemeinschaft (DFG).

\section{Background on double categories}
In this section, we give a concise introduction of the main notions of double category theory needed in this paper. We refer the reader to \cite{Grandis2019} for more details.

As our lax double presheaves will be lax double functors taking value in the weak double category $\dCat$, we start by introducing in \cref{subsec: weak dbl} the notions of weak double categories and (vertically) lax double functors between them. We further introduce horizontal and vertical transformations between such functors, as well as modifications.

By imposing strictness conditions, we retrieve in \cref{subsec: dbl} the notions of double categories, double functors, horizontal and vertical transformations between double functors, and modifications.

\subsection{Weak double categories}
\label{subsec: weak dbl}

We start by introducing (vertically) weak double categories. 

\begin{Definition} \label{def: weak double}
    A \emph{weak double category} $\dC$ consists of
    \begin{itemize}[leftmargin=1cm]
        \item objects $x,y,x',y',\ldots$, 
        \item horizontal morphisms $f\colon x\to y$, with a horizontal identity $1_x$ at each object $x$, and an associative and unital composition $gf$ for all composable horizontal morphisms $f,g$,
        \item vertical morphisms $u\colon x\bulletarrow x'$, with a vertical identity $e_x$ at each object $x$, and a composition $u'\bullet u$ for all composable vertical morphisms $u,u'$,
        \item squares $\alpha$ as depicted below, written inline as $\alpha\colon\edgesquare{u}{f}{f'}{v}\colon\nodesquare{x}{x'}{y}{y'}$ or simply as $\alpha\colon\edgesquare{u}{f}{f'}{v}$,
        \[
        \begin{tikzcd}
            x \arrow[r, "f"] \arrow[d,"\bullet" marking, "u"'] \arrow[rd, phantom, "\alpha"] & y \arrow[d,"\bullet" marking, "v"] \\
            x' \arrow[r, "f'"'] & y'              
        \end{tikzcd}
        \]
         with a horizontal identity square $1_u$ at each vertical morphism $u$, and an associative and unital horizontal composition $\beta\circ\alpha$ along vertical morphisms for all horizontally composable squares~$\alpha,\beta$, as well as a vertical identity square $e_f$ at each horizontal morphism $f$, and a vertical composition $\alpha'\bullet\alpha$ along horizontal morphisms for all vertically composable squares~$\alpha,\alpha'$. 
         
        If the vertical boundaries of a square are identities, i.e., $x=x'$, $y=y'$, $u=e_x$ and $v=e_y'$, we call such a square \emph{globular}. 
         
        \item for all composable vertical morphisms
        $x\overset{u}{\bulletarrow}x'\overset{u'}{\bulletarrow}x''\overset{u''}{\bulletarrow}x'''$, a horizontally invertible \emph{associator square}
        \[ \alpha_{u,u',u''}\colon \edgesquare{(u''\bullet u')\bullet u}{1_x}{1_{x'''}}{u''\bullet(u'\bullet u)}, \]
        \item for every vertical morphism 
        $u\colon x\bulletarrow x'$, horizontally invertible \emph{left} and \emph{right unitor squares}
        \[ \lambda_u\colon \edgesquare{u\bullet e_x}{1_x}{1_{x'}}{u} \quad \text{and} \quad \rho_u\colon \edgesquare{e_{x'}\bullet u}{1_x}{1_{x'}}{u}, \]
    \end{itemize}
    satisfying the following conditions:
    \begin{enumerate}[leftmargin=1cm]
        \item horizontal and vertical compositions of squares satisfy the interchange law, 
        \item for all composable horizontal morphisms $f,g$, we have $e_g\circ e_f=e_{gf}$, and for every object $x$, we have $e_{1_x}=1_{e_x}$,
        \item for all composable vertical morphisms
        $u,u'$, we have $1_{u'}\bullet 1_{u}=1_{u'\bullet u}$, 
        \item the associator square $\alpha_{u,u',u''}$ is natural in 
        $(u,u',u'')$, and satisfies the pentagon axiom,
        \item the unitor squares 
        $\lambda_u$ and $\rho_u$ are natural in $u$, compatible with associator squares, and for an object~$x$, one has
        $\lambda_{e_x}=\rho_{e_x}$.
    \end{enumerate}

    If the unitor squares are horizontal identities, we call $\dC$ a \emph{unitary weak double category}.
    
    For a more detailed description of the coherence conditions, we refer the reader to \cite[Definition 3.3.1]{Grandis2019}.
\end{Definition}

Next, we introduce a suitable notion of functors between weak double categories, namely the notion of a lax double functor. As horizontal composition in a weak double category is strictly associative and unital while the vertical one is not, it is not surprising that lax double functors only preserve strictly horizontal compositions and identities while, in the vertical direction, we get comparison squares.
 
\begin{Definition} \label{def: lax double}
    Given weak double categories $\dC,\dD$, a \emph{lax double functor} $X\colon\dC\to\dD$ consists of
    \begin{itemize}[leftmargin=1cm]
        \item assignments on objects, horizontal morphisms, vertical morphisms, and squares, which are compatible with sources and targets, 
        \item for all composable vertical morphisms 
        $x\overset{u}{\bulletarrow}x'\overset{u'}{\bulletarrow}x''$ in $\dC$, a \emph{composition comparison square} in $\dD$
        \[
        \mu_{u,u'}\colon\edgesquare{Xu'\bullet Xu}{1_{Xx}}{1_{Xx''}}{X(u'\bullet u)}\colon
        \nodesquare{Xx}{Xx''}{Xx}{Xx''},
        \]
        \item for every object $x$ in $\dC$, an \emph{identity comparison square} in $\dD$
        \[
        \epsilon_{x}\colon
        \edgesquare{e_{Xx}}{1_{Xx}}{1_{Xx}}{Xe_x}\colon
        \nodesquare{Xx}{Xx}{Xx}{Xx},
        \]
    \end{itemize}
    satisfying the following conditions:
    \begin{enumerate}[leftmargin=1cm]
    \item it preserves horizontal identities and horizontal compositions strictly,
    \item composition comparison squares $\mu_{u,u'}$ are natural with respect to $(u,u')$, and compatible with associator squares,
    \item identity comparison squares $\epsilon_x$ are natural with respect to $x$, and compatible with composition comparison squares and unitor squares.
    \end{enumerate}
    
    If all identity comparison squares are horizontal identities, i.e., if $X$ preserves strictly vertical identities, we call $X$ a \emph{normal lax double functor}. Consequently, the comparison squares $\mu_{u,e_{x'}}$ and $\mu_{e_x,u}$ agree with the identity.
    
   For a more detailed description of the coherence conditions, we refer the reader to \cite[Definition 3.5.1]{Grandis2019}.
\end{Definition}

Lax double functors between two weak double categories assemble into a weak double category, and we introduce here its horizontal morphisms, vertical morphisms, and squares.

\begin{Definition} \label{def: hor transf}
    Given lax double functors 
    $X,Y\colon\dC\to\dD$, a \emph{horizontal transformation} $F\colon X\Rightarrow Y$ consists of
    \begin{itemize}[leftmargin=1cm]
        \item for every object $x$ in $\dC$, a horizontal morphism 
        $Fx\colon Xx\to Yx$ in $\dD$,
        \item for every vertical morphism
        $u\colon x\bulletarrow x'$ in $\dC$, a square
        $F_u\colon\edgesquare{Xu}{F_x}{F_{x'}}{Yu}$ in $\dD$,
    \end{itemize}
     such that the components $F_x$ are natural in $x$, and the components $F_u$ are natural in $u$ and compatible with composition and identity comparison squares.

     For a more detailed description of the coherence conditions, we refer the reader to \cite[\textsection 3.5.4]{Grandis2019}.
\end{Definition}

\begin{Definition} \label{def: pseudo vert transf}
    Given lax double functors $X,X'\colon\dC\to\dD$, a \emph{colax vertical transformation} $U\colon X{\Bulletarrow}X'$ consists of 
    \begin{itemize}[leftmargin=1cm]
        \item for every object 
        $x$ in $\dC$, a vertical morphism 
        $U_x\colon Xx\bulletarrow X'x$ in $\dD$,
        \item for every horizontal morphism 
        $f\colon x\to y$ in $\dC$, a square
        $U_f\colon\edgesquare{U_x}{Xf}{X'f}{U_y}$ in $\dD$,
        \item for every vertical morphism 
        $u\colon x\bulletarrow x'$ in $\dC$, a \emph{naturality comparison square} in $\dD$
        \[
        U_u\colon
        \edgesquare{U_{x'}\bullet Xu}{1_{Xx}}{1_{X'x'}}{X'u\bullet U_x},
        \]
    \end{itemize}
    satisfying the following conditions:
    \begin{enumerate}[leftmargin=1cm]
        \item the components $U_f$ are compatible with composition of horizontal morphisms and horizontal identities,
        \item the components $U_u$ are natural in $u$ and compatible with composition and identity comparison squares.
    \end{enumerate}
    If the naturality comparison squares are invertible, we call $U$ a \emph{pseudo vertical transformation}.

    Moreover, if a pseudo vertical transformation is strictly natural in the vertical direction, i.e., if the natural comparison squares $U_u$ are horizontal identities, we call $U$ a \emph{vertical transformation}.
    
    For a more detailed description of the coherence conditions, we refer the reader to the transposed version of \cite[Definition 2.2]{bifunctor}. See also
    \cite[Definition 3.8.2]{Grandis2019}.
\end{Definition}

\begin{Definition}\label{def: modification}
    Given lax double functors 
    $X,X',Y,Y'\colon\dC\to\dD$, horizontal transformations $F\colon X\Rightarrow Y$, $F'\colon X'\Rightarrow Y'$ and colax vertical transformations 
    $U\colon X\Bulletarrow X'$, $V\colon Y\Bulletarrow Y'$, then a \emph{modification} 
    $A\colon\edgesquare{U}{F}{F'}{V}$ consists of, for every object $x$ in $\dC$, a square in $\dD$
    \[
    A_x\colon
    \edgesquare{U_x}{F_x}{F'_x}{V_x}\colon
    \nodesquare{Xx}{X'x}{Yx}{Y'x}
    \]
    satisfying the following conditions:
    \begin{enumerate}[leftmargin=1cm]
        \item (horizontal compatibility) 
        for every horizontal morphism $f\colon x\to y$ in $\dC$, one has
        \[
        A_y\circ U_f = V_f \circ A_x,
        \]
        \item (vertical compatibility)
        for every vertical morphism 
        $u\colon x\bulletarrow x'$ in $\dC$, one has
        \[
        (F'_u\bullet A_x)\circ U_u=
        V_u\circ(A_{x'}\bullet F_u).
        \]
        \end{enumerate}

    If the vertical boundaries of a modification are identities, i.e., $X=X'$, $Y=Y'$, $U=e_X$ and $V=e_Y$, we call such a modification \emph{globular}. 
\end{Definition}

As claimed above, lax double functors assemble into a weak double category. As we will not make use of the vertical structure of this double category throughout the paper, we state the result as it appears in \cite[Theorem 3.8.4]{Grandis2019}, where the vertical morphisms are given by the pseudo vertical transformation. A version with colax vertical transformations also exists, as proven in \cite[\textsection 2]{bifunctor}.

\begin{Proposition} \label{prop:laxhom}
    Given two weak double categories $\dC$ and $\dD$, there is a weak double category $\dHomLax{\dC,\dD}$ of lax double functors $\dC\to\dD$, horizontal transformations, pseudo vertical transformations and modifications.

    We denote by $\HH\dHomLax{\dC,\dD}$ its underlying $2$-category of lax double functors, horizontal transformations, and globular modifications.
\end{Proposition}

\subsection{(Strict) double categories}

\label{subsec: dbl}

To simplify computations, except for the weak double category~$\dCat$, all of our double categories will be strict. Therefore, we further present here the strict version of double categories. We refer the reader to \cite[\textsection 3.2]{Grandis2019} for a more detailed introduction to strict double categories.

\begin{Definition}
    A \emph{double category} is a weak double category, where the composition of vertical morphisms and the vertical composition of squares are strictly associative and unital. In other words, the associator and unitor squares in \cref{def: weak double} are horizontal identities.
\end{Definition}

\begin{Definition}
    A \emph{double functor} is a lax double functor between double categories which preserves vertical compositions and identities strictly. In other words, the composition and identity comparison squares in \cref{def: lax double} are horizontal identities.
\end{Definition}

\begin{Notation}
    We denote by $\DblCat$ the category of double categories and double functors.
\end{Notation}

Due to the strictness of vertical compositions, we can now define two different underlying categories of a double category. 

\begin{Definition}
    Given a double category $\dC$, we define    \begin{itemize}[leftmargin=1cm]
        \item its \emph{underlying vertical category} $\Ver_0\dC$ to be the category of objects and vertical morphisms in $\dC$,
        \item its \emph{underlying horizontal category} $\Hor_0\dC$ to be the category of objects and horizontal morphisms in $\dC$. 
    \end{itemize}
    These constructions extend to functors 
    \[ \Ver_0,\Hor_0\colon \DblCat\to \Cat. \]
\end{Definition}

\begin{Remark}\label{rem: vertical/horizontal dbl cat}
The functors $\Ver_0,\Hor_0\colon \DblCat\to \Cat$ admit left adjoints 
\[ \dV,\dH\colon \Cat\to \DblCat,  \]
respectively. We call a category in the image of $\dV$ (resp.~$\dH$) a \emph{vertical} (resp.~\emph{horizontal}) double category. 
\end{Remark}

\begin{Remark} \label{adjunction VV-HH}
    Given a double category $\dC$, we can further define a category $\Ver_1\dC$ of horizontal morphisms and squares, with composition given by the vertical composition of squares in $\dC$. We can then see $\dC$ as an internal category to $\Cat$
    \[
    \begin{tikzcd}
        \Ver_0\dC \arrow[r, "i" description] & \Ver_1\dC \arrow[l, "t", shift left=2] \arrow[l, "s"', shift right=2] & \Ver_1\dC\times_{\Ver_0\dC}\Ver_1\dC. \arrow[l, "c"']
    \end{tikzcd}
    \]
    
\end{Remark}

We can, in fact, upgrade the underlying horizontal category to a $2$-category as follows.

\begin{Definition}
    Given a double category $\dC$, we define its \emph{underlying horizontal $2$-category} $\HH\dC$ to be the $2$-category of objects, horizontal morphisms, and globular squares in $\dC$. This construction extends to a functor 
    \[ \HH\colon \DblCat\to 2\Cat. \]
\end{Definition}

\begin{Remark} \label{adjunction HH}
    The functor $\HH\colon \DblCat\to 2\Cat$ admits a left adjoint 
    \[ \dH\colon 2\Cat\to \DblCat, \]
    which sees a $2$-category as a double category with only trivial vertical morphisms.
\end{Remark}

The following result can be deduced from \cite[Lemma B2.3.15(ii)]{elephant}, using that double categories are internal categories in $\Cat$.

\begin{Proposition}
\label{prop: cartesian closed}
    The category $\DblCat$ is cartesian closed.
\end{Proposition}

\begin{Notation}
    Given double categories $\dC$ and $\dD$, we denote by $\dHom{\dC,\dD}$ the internal hom in~$\DblCat$. It is the double category whose 
    \begin{itemize}[leftmargin=1cm]
        \item objects are double functors $\dC\to \dD$, 
        \item horizontal morphisms are horizontal transformations between (strict) double functors as defined in \cref{def: hor transf}, 
        \item vertical morphisms are vertical transformations between double functors as defined in \cref{def: pseudo vert transf},
        \item squares are modifications as defined in \cref{def: modification}.
    \end{itemize}
    We refer the reader to \cite[\textsection 3.2.7]{Grandis2019} for more details.
\end{Notation}

Using these structures, we can define two different $2$-categories of double categories, by picking the $2$-morphisms to be either the horizontal or vertical transformations. Each of them will play an important role in this paper: the one with horizontal transformations will be giving the naturality of the constructions, while the one with the vertical transformations will be used to define the $2$-category of double categorical discrete fibrations. 
    
\begin{Notation}
\label{not: 2-cat DblCat}
    The category $\DblCat$ can be upgraded into a $2$-category in two different ways: 
    \begin{itemize}[leftmargin=1cm]
        \item we write $\DblCat_h$ for the $2$-category of double categories, double functors, and horizontal transformations,
        \item we write $\DblCat_v$ for the $2$-category of double categories, double functors, and vertical transformations.
    \end{itemize}
\end{Notation}

\begin{Remark} \label{2-adjunctions} 
    By \cite[Proposition 2.5]{Model_Structure}, the adjunctions $\dV\dashv \Ver_0$ from \cref{rem: vertical/horizontal dbl cat} and $\dH\dashv \HH$ from \cref{adjunction HH} extend to $2$-adjunctions
    \[ \dV\colon \Cat\leftrightarrows \DblCat_v\colon \Ver_0 \quad \text{and}\quad \dH\colon 2\Cat\leftrightarrows \DblCat_h \colon \HH. \]
\end{Remark}

We will also make use of the horizontal opposite of a double category, which we now recall.

\begin{Definition}
    We define a $2$-functor $(-)^{\op}\colon\DblCat_h^{\co}\to\DblCat_h$, where $\DblCat_h^{\co}$ is the $2$-category obtained from $\DblCat_h$ by reversing the $2$-morphisms, sending
    \begin{itemize}[leftmargin=1cm]
        \item a double category 
        $\dC$ to its \emph{horizontal opposite double category} $\dC^{\op}$ which consists of the same objects and vertical morphisms as $\dC$, but where the direction of horizontal morphisms and squares is horizontally reversed,
        \item a double functor $G\colon\dC\to\dD$ to the double functor 
        $G^{\op}\colon\dC^{\op}\to\dD^{\op}$,
        which acts as $G$ on objects, horizontal morphisms, vertical morphisms, and squares,
        \item a horizontal transformation 
        $B\colon G\Rightarrow G'\colon\dC\to\dD$ to the horizontal transformation 
        \[ B^{\op}\colon G'^{\op}\Rightarrow G^{\op}\colon\dC^{\op}\to\dD^{\op}, \]
        whose
        components are the same as those of $B$.
    \end{itemize}
    It is straightforward to check that this construction is $2$-functorial.
\end{Definition}

\section{Lax double presheaves}

In this section, we introduce our lax double presheaves as lax double functors $\dC^{\op}\to \dCat$ with~$\dC$ a (strict) double category and $\dCat$ the weak double category of categories. To introduce the weak double category $\dCat$, we need two notions of morphisms between categories. The horizontal ones will simply be the functors, and the vertical ones will be the profunctors. For this, we first recall in \cref{subsec: profunctors} the notion of profunctors and their relations to two-sided discrete fibrations. Then, in \cref{subsec: dCat}, we review the construction of the weak double category $\dCat$. 

In \cref{subsec: lax presheaf}, we introduce the notion of lax double presheaves, and finally, in \cref{subsec: repr presheaf}, we study a first class of examples of lax double presheaves, the \emph{representable lax double presheaves}.

\subsection{Profunctors and two-sided discrete fibrations}

\label{subsec: profunctors}

We start by recalling profunctors, as well as their weakly associative and unital composition.

\begin{Definition}
    A \emph{profunctor} $U\colon\CC\bulletarrow\CC'$ between categories $\CC$ and $\CC'$ is a functor \[ U\colon\CC^{\op}\times\CC'\to\Set. \]
\end{Definition}

\begin{Notation}
    We denote by $\Prof(\CC,\CC')$ the category of profunctors $\CC^{\op}\times \CC'\to \Set$ and natural transformations between them. 
\end{Notation}

Composition of profunctors is defined via coends; see e.g.~\cite[\textsection 1.2]{riehl2014categoricalhomotopy} for a definition. As the category $\Set$ is cocomplete, coends in $\Set$ always exist and are given very explicitly by the following formula, as mentioned in \cite[(1.2.4)]{riehl2014categoricalhomotopy}.

\begin{Proposition}
\label{prop: coend construction}
    Given a functor $U\colon\CC^{\op}\times\CC\to\Set$, a coend of $U$ exists and can be computed as the coequalizer
    \[
    \gro^{x} U(x,x)\cong \mathrm{coeq}
    \left(
    \begin{tikzcd}
        \bigsqcup\limits_{x\xrightarrow{f}x'\in \CC}U(x',x)
        \arrow[r,yshift = 3pt]
        \arrow[r,yshift = -3pt,swap]
        &
        \bigsqcup\limits_{x\in \CC}
        U(x,x)
    \end{tikzcd}
    \right),
    \]
    where the two parallel maps are induced by $U(f,x)$ and $U(x',f)$, respectively.
\end{Proposition}

\begin{Remark}
    As universal constructions, coends are unique up to a unique isomorphism. In what follows, we choose a specific coend for each functor $U\colon\CC^{\op}\times\CC\to\Set$, and speak of \emph{the coend} of~$U$.
\end{Remark}

We can now use the notion of coends to define composition of profunctors.

\begin{Construction}\label{def: comp of prof}
    Given profunctors 
        $U\colon\CC\bulletarrow\CC'$ and $U'\colon\CC'\bulletarrow\CC''$, their composition is the profunctor $U'\bullet U\colon\CC^{\op}\times\CC''\to\Set$ sending
        \begin{itemize}[leftmargin=1cm]
            \item an object $(x,x'')$ in $\CC\times\CC''$ to the set 
            $\gro^{x'\in\CC'}U'(x',x'')\times U(x,x')$ given by the coend of the functor $U'(-,x'')\times U(x,-)\colon \CC'^{\op}\times \CC'\to \Set$, 
            \item a morphism $(f,f'')\colon (x,x'')\to (y,y'')$ in $\CC\times\CC''$ to the unique induced map between coends
            \[
            \gro^{x'\in\CC'} U'(x',f'')\times U(f,x')\colon           
            \gro^{x'\in\CC'} U'(x',x'')\times U(y,x')\to
            \gro^{x'\in\CC'} U'(x',y'')\times U(x,x').
            \]
        \end{itemize}
        This construction extends to a functor
        \[
    \bullet\colon\Prof(\CC,\CC')\times\Prof(\CC',\CC'')\to
    \Prof(\CC,\CC'').
    \]
\end{Construction}

\begin{Remark}
    Composition of profunctors is not strictly associative as the order in which we take coequalizers matters. However, as coends are unique up to a unique isomorphism, composition of profunctors is associative up to a unique invertible comparison cell.
\end{Remark}

Moreover, this composition of profunctors admits as identities the following. 

\begin{Definition}
    The \emph{identity profunctor} $e_\CC\colon \CC\bulletarrow \CC$ at a category $\CC$ is given by the hom functor $\CC(-,-)\colon\CC^{\op}\times\CC\to\Set$.
\end{Definition}

It will be convenient if composition of profunctors is strictly unital. The following result implies that we can pick $U\bullet e_\CC\coloneqq U$ and $e_{\CC'}\bullet U\coloneqq U$, for every profunctor $U\colon \CC^{\op}\times \CC'\to \Set$. 

\begin{Lemma}
\label{lem: comp of prof is unital}
    Given a profunctor 
    $U\colon\CC^{\op}\times\CC'\to\Set$ and objects $x$ in $\CC$, $x''$ in $\CC'$, there are canonical natural bijections
    \[ U(x,x'')\cong \gro^{x'\in\CC}U(x',x'') \times \CC(x,x')
    \quad \text{and} \quad
    U(x,x'')\cong \gro^{x'\in\CC}\CC'(x',x'')\times U(x,x').
    \]
    In particular, this implies that there are canonical natural isomorphisms of profunctors
    \[ U\cong U\bullet e_\CC \quad \text{and} \quad U\cong e_{\CC'}\bullet U. \]
\end{Lemma}

\begin{proof}
    The first bijection can be shown by proving that the set $U(x,x'')$ together with the maps
    \[
    \iota_{x'}\colon U(x',x'')\times\CC(x,x')\to 
    U(x,x''), \quad (u,f)\mapsto U(f,x'')(u), 
    \]
    for objects $x'$ in $\CC$, satisfies the universal property of the coend of $U(-,x'')\times\CC(x,-)\colon\CC^{\op}\times\CC\to\Set$. The second bijection can be shown analogously.
\end{proof}

In the remainder of this section, we study the connection between profunctors and two-sided discrete fibrations, which will be useful later. For this, let us first recall the definition of a two-sided discrete fibration.

\begin{Definition}
    Given a functor $P\colon\EE\to\CC$, we say that a morphism 
    $g\colon e\to e'$ in $\EE$
    \begin{itemize}[leftmargin=1cm]
        \item is a \emph{$P$-lift} of a morphism $f\colon c\to c'$ in $\CC$ if 
    $Pg=f$, 
    \item \emph{lies in the fiber of $P$ at an object $c$ in $\CC$} if it is a $P$-lift of the identity $1_c$.
    \end{itemize}
\end{Definition}

\begin{Definition}
\index{two-sided discrete fibration}
\label{def: two-sided fib}
    A functor $(P,Q)\colon \EE\to\CC\times\CC'$ is a \emph{two-sided discrete fibration} over $\CC\times\CC'$, if the following conditions hold:
    \begin{enumerate}[leftmargin=1cm]
        \item for every object $e$ in $\EE$ and every morphism $f\colon x\to Pe$ in $\CC$, there is a unique $P$-lift of $f$ with target $e$ lying in the fiber of $Q$ at $Qe$, i.e, there is a unique morphism 
        $P^*f\colon f^*e\to e$ in~$\EE$ such that
        $P(P^*f)=f$ and $Q(P^*f)=1_{Qe}$,
        \item for every object $e$ in $\EE$ and every morphism 
        $f'\colon Qe\to x'$ in $\CC'$, there is a unique $Q$-lift of $f'$ with source $e$ lying in the fiber of $P$ at $Pe$, i.e., there is a unique 
        morphism 
        $Q_!f'\colon e\to f'_!e$ in~$\EE$ such that 
        $Q(Q_!f')=f'$ and $P(Q_!f')=1_{Pe}$,
        \item for every morphism $g\colon e\to e'$ in $\EE$, the source of the unique $P$-lift of $Pg$ agrees with the target of the unique $Q$-lift of $Qg$ and their composite is $g$, i.e., one has 
        $(Pg)^*e'=(Qg)_!e$ and $P^*(Pg)\circ Q_!(Qg)=g$.
    \end{enumerate}
\end{Definition}

\begin{Notation}
    We denote by $\TSdiscFib(\CC,\CC')$ the full subcategory of the slice category $\Cat/_{\CC\times\CC'}$ spanned by the two-sided discrete fibrations over $\CC\times\CC'$.
\end{Notation}

There is an equivalence of categories between profunctors and two-sided discrete fibrations, that we now recall. 

\begin{Construction} \label{constr: functor fib}
    We construct a functor \[ \fib\colon\TSdiscFib(\CC,\CC')\to\Prof(\CC,\CC'). \]
    It sends a two-sided discrete fibration $(P,Q)\colon\EE\to\CC\times\CC'$
    to the profunctor $\fib(P,Q)\colon \CC^{\op}\times \CC'\to \Set$ sending
    \begin{itemize}[leftmargin=1cm]
        \item an object $(x,x')$ in $\CC\times\CC'$ to the fiber 
        $(P,Q)^{-1}(x,x')$ of $(P,Q)$ at $(x,x')$,
        \item a morphism $(f,f')\colon (x,x')\to (y,y')$ in $\CC^{\op}\times\CC'$ to the map
        $(P,Q)^{-1}(x,x')\to (P,Q)^{-1}(y,y')$ given by sending $e$ to 
        $f'_!f^*e=f^*f'_!e$,
    \end{itemize}
    and a morphism of two-sided discrete fibrations
    \[
    \begin{tikzcd}
        \EE
        \arrow[rr,"F"] 
        \arrow[rd,swap,"{(P,Q)}"]
        &&
        \EE'
        \arrow[ld,"{(P',Q')}"]
        \\
        &
        \CC\times\CC'
        &
    \end{tikzcd}
    \]
    to the natural transformation $\fib(F)\colon\fib(P,Q)\Rightarrow\fib(P',Q')$ whose component at an object $(x,x')$ in $\CC\times\CC'$ is the unique induced map between fibers \[ \fib(F)_{x,x'}\coloneqq F_{x,x'}\colon (P,Q)^{-1}(x,x')\to (P',Q')^{-1}(x,x'). \]
\end{Construction}

A proof of the following result can be found in \cite[Theorem 2.3.2]{CatFibrations}.

\begin{Theorem}
\label{thm: TSdisc vs Prof}
    The functor $\fib$ induces an equivalence of categories
    \[
    \begin{tikzcd}
         \TSdiscFib(\CC,\CC')
        \arrow[r,"\fib"]
        \arrow[r,phantom,swap,yshift= -5pt,"\simeq"]
        &
       \Prof(\CC,\CC').
    \end{tikzcd}
    \]
\end{Theorem}

In particular, this equivalence is compatible with precomposition and pullback, as follows. 

\begin{Proposition}
\label{prop: naturality of fib}
    Given functors $G\colon \CC\to \DD$ and $G'\colon \CC'\to \DD'$, the following diagram of functors commutes up to natural isomorphism
\[
    \begin{tikzcd}
        \TSdiscFib(\DD,\DD')\arrow[r,"{\fib}"]\arrow[r,phantom,swap,yshift= -5pt,"\simeq"]\arrow[d,swap,"{(G\times G')^*}"] &
        \Prof(\DD,\DD')\arrow[d,"{(G^{\op}\times G')^*}"] 
        \\
        \TSdiscFib(\CC,\CC')\arrow[r,"{\fib}"]\arrow[r,phantom,swap,yshift= -5pt,"\simeq"]
        &
        \Prof(\CC,\CC')
    \end{tikzcd}
    \]
    where the left-hand functor is induced by taking pullbacks along $G\times G'\colon \CC\times \CC'\to \DD\times \DD'$ and the right-hand functor is induced by precomposing along $G^{\op}\times G'\colon \CC^{\op}\times \CC'\to \DD^{\op}\times \DD'$. 
\end{Proposition}

As a consequence, we get the following result. 

\begin{Corollary}
\label{cor: nat.trf. for TS square}
    Given two-sided discrete fibrations $(P,Q)\colon \EE\to \CC\times \CC'$ and $(R,S)\colon \mathcal{F}\to \DD\times \DD'$, a commutative square of functors of the form
\[
    \begin{tikzcd}
        \EE\arrow[r,"{F}"]\arrow[d,swap,"{(P,Q)}"] & [10pt]
        \mathcal{F}  \arrow[d,"{(R,S)}"]
        \\
        \CC\times \CC'\arrow[r,"{G\times G'}"']
        &
        \DD\times \DD'
    \end{tikzcd}
    \]
   corresponds to a natural transformation 
\[ F\colon \fib(P,Q)\Rightarrow \fib(R,S)(G^{\op}\times G')\colon \CC^{\op}\times \CC'\to \Set, \]
 whose component at an object $(x,x')$ in $\CC^{\op}\times \CC'$ is given by the unique induced map between fibers 
\[ F_{x,x'}\colon (P,Q)^{-1}(x,x')\to (R,S)^{-1}(Gx,Gx'). \]
\end{Corollary}

\begin{proof}
    By the universal property of pullback, such a square of functors corresponds to a functor $\hat{F}\colon\EE\to(G\times G')^*\mathcal{F}$ over $\CC\times \CC'$ as depicted in the following diagram:
    \[
    \begin{tikzcd}
        \EE \arrow[rrd, "F",bend left] \arrow[rdd, "{(P,Q)}"',bend right] \arrow[rd, "\hat{F}", dashed] &  &  \\
              & (G\times G')^*\mathcal{F} \arrow[r] \arrow[d] 
              \arrow[rd,near start, phantom,"\lrcorner"] & \mathcal{F}\arrow[d, "{(R,S)}"]  \\
        & \CC\times\CC' \arrow[r, "G\times G'"']  & \DD\times\DD'  
\end{tikzcd}
    \]
    Using \cref{prop: naturality of fib}, the functor $\hat{F}$ over $\CC\times \CC'$ corresponds to a natural transformation
    \[
    \fib(P,Q)\Rightarrow \fib((G\times G')^*(R,S))\cong
    (G^{\op}\times G')^*(\fib(R,S))=\fib(R,S)(G^{\op}\times G').
    \]
    Note that the description of the components of this natural transformation follows directly from the definition of the functor $\fib$.
\end{proof}

Finally, similarly to profunctors one can also define a composition for two-sided discrete fibrations as follows.

\begin{Construction}
Given two-sided discrete fibrations $(P,Q)\colon \EE\to \CC\times \CC'$ and $(P',Q')\colon \EE'\to \CC'\times \CC''$, let $\mathcal{A}$ be the wide subcategory of the pullback $\EE\times_{\CC'} \EE'$ with morphisms given by 
    \[ \{ (e, g^* e') \xrightarrow{(Q_! g, P'^*g)} (g_!e, e') \mid Qe\xrightarrow{g} P'e'\in \CC' \}. \]
    We define the composition $(P',Q')\bullet (P,Q)\colon \EE\times_{\CC'} \EE'_{/\sim}\to \CC\times \CC''$ to be the unique functor obtained by the universal property of the pushout $\EE\times_{\CC'} \EE'_{/\sim}$ defined by the following commutative diagram in $\Cat$.
     \[
        \begin{tikzcd}
           \mathcal{A}  \arrow[d] \arrow[r]
           \arrow[rd,phantom,near end,"\ulcorner"]
           & \EE\times_{\CC'} \EE' \arrow[rdd,bend left,"{(P\pi_1,Q'\pi_2)}"] \arrow[d] \\
            \pi_0\mathcal{A} \arrow[rrd,bend right] \arrow[r]& \EE\times_{\CC'} \EE'_{/\sim} \arrow[rd,dashed, "{(P',Q')\bullet (P,Q)}"'] \\
            & & \CC\times \CC''
        \end{tikzcd}
        \]
    Here $\pi_i$, for $i=1,2$ denote the canonical projections from the pullback $\EE\times_{\CC'} \EE'$, and $\pi_0\mathcal{A}$ is the set of path components of the category $\mathcal{A}$ and $\mathcal{A}\to \pi_0\mathcal{A}$ is the canonical projection. Note that the composite 
    \[ \AA\to \EE\times_{\CC'} \EE'\xrightarrow{(P\pi_1,Q'\pi_2)} \CC\times \CC'' \]
    factors through $\AA\to \pi_0\AA$ since, for every morphism $g\colon Qe\to P'e'$ in $\CC'$, we have
    \[ P\pi_1(Q_! g, P'^*g)=P(Q_! g)=\id_{Pe}  \quad \text{and} \quad Q'\pi_2(Q_! g, P'^*g)=Q'(P'^*g)=\id_{Q'e'}, \]
    and so we get an outer commutative diagram, as desired.
    
    Using \cref{lem:composite is two-sided disc fib} below, this construction extends to a functor 
    \[
    \bullet\colon\TSdiscFib(\CC,\CC')\times
    \TSdiscFib(\CC',\CC'')\to\TSdiscFib(\CC,\CC'').
    \]
\end{Construction}

\begin{Remark} \label{rem: unpack comp of two sided}
    Unpacking the universal property of the pushout, we observe that the category $\EE\times_{\CC'} \EE'_{/\sim}$ is obtained from the pullback $\EE\times_{\CC'} \EE'$ by making all morphisms of $\mathcal{A}$ into identities. In particular, for every morphism $g\colon Qe\to P'e'$ in $\CC'$, the objects $(e,g^* e')$ and $(g_!e,e')$ in $\EE\times_{\CC'} \EE'$ are identified in $\EE\times_{\CC'} \EE'_{/\sim}$.
\end{Remark}

\begin{Lemma} \label{lem:composite is two-sided disc fib}
    Given two-sided discrete fibrations $(P,Q)\colon \EE\to \CC\times \CC'$ and $(P',Q')\colon \EE'\to \CC'\times \CC''$, their composition
    \[ (P',Q')\bullet (P,Q)\colon \EE\times_{\CC'} \EE'_{/\sim}\to \CC\times \CC''\] 
    is a two-sided discrete fibration.
\end{Lemma}

\begin{proof}
   By definition, the functor 
    $(P',Q')\bullet(P,Q)\colon \EE\times_{\CC'} \EE'_{/\sim}\to \CC\times \CC''$ acts as $P\colon \EE\to \CC$ and $Q'\colon \EE'\to \CC''$ on each component, and so we write $(P',Q')\bullet(P,Q)=(P\pi_1,Q'\pi_2)$, where $\pi_i$, for $i=1,2$, denote the canonical projections from the pullback $\EE\times_{\CC'} \EE'$.
    
    We show that $(P\pi_1,Q'\pi_2)\colon \EE\times_{\CC'} \EE'_{/\sim}\to \CC\times \CC'$ satisfies the conditions of a two-sided discrete fibration from \cref{def: two-sided fib}. To prove (1), let $[\hat{e},\hat{e}']$ be an object in $\EE\times_{\CC'} \EE'_{/\sim}$ and $f\colon x\to P\hat{e}$ be a morphism in $\CC$. Then a $P\pi_1$-lift of 
    $f$ with target $[\hat{e},\hat{e}']$ that lies in the fiber of~$Q'\pi_2$ at $Q'\hat{e}'$ is given by the morphism in $\EE\times_\CC'\EE'_{/\sim}$
    \[
    \left[P^*f,1_{\hat{e}'}\right]\colon [f^*\hat{e},\hat{e}']\to [\hat{e},\hat{e}'].
    \]
    Note that this is well-defined as $Q(P^* f)=1_{Q\hat{e}}=1_{P'\hat{e}'}=P'(1_{\hat{e}'})$ in $\CC'$. Now let 
    \[ [g,g']\colon [e,e']\to [\hat{e},\hat{e}'],  \]
    be another $P\pi_1$-lift of $f$ that lies in the fiber of $Q'\pi_2$ at $Q'\hat{e}'$, i.e., a morphism $[g,g']$ in $\EE\times_{\CC'}\EE'_{/\sim}$ such that $Pg=f$ and $Q'g'=1_{Q'\hat{e}'}$. By definition of the pullback $\EE\times_{\CC'}\EE'$, we further have $Qg=P'g'$ in~$\CC'$. Using condition (3) of \cref{def: two-sided fib} for the two-sided discrete fibrations $(P,Q)$ and $(P',Q')$, we have factorizations
    \[ g= 
        P^*(Pg)\circ Q_!(Qg)=P^*f\circ Q_!(Qg) \]
        and 
        \[ g'= P'^*(P'g')\circ Q'_!(Q'g')=
        P'^*(Qg)\circ Q'_!(1_{Q'\hat{e}'})=
        P'^*(Qg)\circ 1_{\hat{e}'}=P'^*(Qg), \]
        where we used that $Q'_!(1_{Q'\hat{e}'})=1_{\hat{e}'}$ by unicity of lifts.
    Since the morphism $[Q_!(Qg),P'^*(Qg)]$ is, by definition, an identity in $\EE\times_{\CC'}\EE'_{/\sim}$, we have that 
    
    \[ [g,g']=[P^*f\circ Q_!(Qg),P'^*(Qg)]=[P^*f,1_{\hat{e'}}]\circ [Q_!(Qg),P'^*(Qg)]=[P^*f,1_{\hat{e}'}]. \]
    This shows that $[P^*f,1_{\hat{e}'}]$ is the unique such lift. Condition (2) can be shown analogously.
    
    To show (3), let 
    $[g,g']\colon [e,e']\to [\hat{e},\hat{e}']$ be a morphism in $\EE\times_{\CC'}\EE'_{/\sim}$. Then the unique lifts of $Pg$ and $Q'g'$ provided by conditions (1) and (2) are the morphisms in $\EE\times_{\CC'}\EE'_{/\sim}$
    \[
    [P^*(Pg),1_{\hat{e}'}]\colon [(Pg)^*\hat{e},\hat{e}']\to
    [\hat{e},\hat{e}']
    \quad \text{and}\quad 
    [1_e,Q'_!(Q'g')]\colon [e,e']\to[e,(Q'g')_!e']. \]
    Recall that, by definition of the pullback $\EE\times_{\CC'}\EE'$, we have $Qg=P'g'$ in $\CC'$. Using property (3) for the two-sided discrete fibrations $(P,Q)$ and $(P',Q')$, we observe that
    
    \[ (Pg)^*\hat{e}=(Qg)_!e \quad \text{and} \quad (Q'g')_!e'=(P'g')^*\hat{e}'=(Qg)^*\hat{e}', \]
    and so the source of the first lift agrees with the target of the second lift by \cref{rem: unpack comp of two sided}. Therefore we can form their composite and we get
    \begin{align*}
        [P^*(Pg),&1_{\hat{e}'}]\circ[1_e,Q'_!(Q'g')]
        &
        \\
        &=
        [P^*(Pg),1_{\hat{e}'}]\circ[Q_!(Qg),P'^*(Qg)]\circ [1_e,Q'_!(Q'g')]
        &
        [Q_!(Qg),P'^*(Qg)]\sim 1_{[(Qg)_!e,\hat{e}']}
        \\
        &=[P^*(Pg),\circ Q_!(Qg),P'^*(Qg)\circ Q'_!(Q'g')] & \text{Composition in }\EE\times_{\CC'}\EE'_{/\sim} \\
        &=
        [g,g'].
        &
        \text{(3) for }(P,Q)\text{ and }(P',Q')
    \end{align*}
   This shows the desired result.
\end{proof}

\begin{Remark}
    As for profunctors, composition of two-sided discrete fibrations is not strictly associative. However, as pullbacks and pushouts are unique up to unique isomorphism, composition of two-sided discrete fibrations is associative up to a unique invertible comparison cell.
\end{Remark}

Composition of two-sided discrete fibrations is defined in such a way that it corresponds to composition of profunctors through the equivalence $\fib$.

\begin{Proposition}
\label{prop: fib compatible with composition}
   Given categories $\CC,\CC',\CC''$, the following diagram of functors commutes up to natural isomorphism
    \[
    \begin{tikzcd}        {\TSdiscFib(\CC,\CC')\times\TSdiscFib(\CC',\CC'')} \arrow[d, "\bullet"'] \arrow[r, "{\fib \times\fib}"]\arrow[r,phantom,swap,yshift= -5pt,"\simeq"] & {\Prof(\CC,\CC')\times\Prof(\CC',\CC'')} \arrow[d, "\bullet"] \\
        {\TSdiscFib(\CC,\CC'')} \arrow[r, "{\fib}"] \arrow[r,phantom,swap,yshift= -5pt,"\simeq"]                                                 & {\Prof(\CC,\CC'').}                                         
    \end{tikzcd}
    \]
\end{Proposition}

\begin{proof}
    Given two-sided discrete fibrations $(P,Q)\colon\EE\to\CC\times\CC'$ and $(P',Q')\colon\EE'\to\CC'\times\CC''$, we want to show that there is an isomorphism in $\Prof(\CC,\CC'')$
    \begin{equation}\label{iso comp}
        \fib((P',Q')\bullet(P,Q))\cong\fib(P',Q')\bullet\fib(P,Q)
    \end{equation}
    which is natural in $(P,Q)$ and $(P',Q')$. When evaluated at an object $(x,x'')$ in $\CC\times \CC''$, this amounts to showing that there is an isomorphism in $\Set$
    \begin{equation} \label{iso eq}\fib((P',Q')\bullet(P,Q))(x,x'')\cong\fib(P',Q')\bullet\fib(P,Q)(x,x'') 
    \end{equation}
    which is natural in $(x,x'')$. By definition, we have that the set $\fib(P',Q')\bullet\fib(P,Q)(x,x'')$ is the coequalizer of the diagram
    \[ \begin{tikzcd}
        \bigsqcup\limits_{x'\xrightarrow{g}\hat{x}'\in \CC'} (P',Q')^{-1}(\hat{x}',x'')\times (P,Q)^{-1}(x,x')
        \arrow[r,yshift = 3pt]
        \arrow[r,yshift = -3pt,swap]
        &
        \bigsqcup\limits_{x'\in \CC'}
        (P',Q')^{-1}(x',x'')\times (P,Q)^{-1}(x,x')
    \end{tikzcd} \]
    where the two parallel maps are induced by $g^*\times 1_{(P,Q)^{-1}(x,x')}$ and $1_{(P',Q')^{-1}(\hat{x}',x'')}\times g_!$, respectively. Hence, to obtain the isomorphism \eqref{iso eq}, it suffices to show that the set $\fib((P',Q')\bullet(P,Q))(x,x'')$ is also a coequalizer of the above diagram. However, by \cref{rem: unpack comp of two sided}, we see that the fiber $\fib((P',Q')\bullet(P,Q))(x,x'')$ consists of the quotient of the set \[ \bigsqcup\limits_{x'\in \CC'}
        (P',Q')^{-1}(x',x'')\times (P,Q)^{-1}(x,x') \]
        by the relation $(e,g^*e')\sim(g_!e,e')$, for all morphisms $g\colon Qe\to P'e'$ in $\CC'$. But these are precisely the relations enforced by the coequalizer. Hence we have a canonical isomorphism \eqref{iso eq}. Moreover, it is natural in $(x,x'')$ since, for both profunctors, the actions of morphisms in $\CC$ and $\CC''$ are determined by taking unique lifts along $P$ and $Q'$, therefore yielding the natural isomorphism \eqref{iso comp}.   

        Finally, the naturality of the isomorphism \eqref{iso comp} in $(P,Q)$ and $(P',Q')$ follows directly from the universal property of the coequalizers.
\end{proof}

Moreover, this composition of two-sided discrete fibration admits as identities the following. 

\begin{Definition} 
The \emph{identity two-sided discrete fibration} at a category $\CC$ is the two-sided discrete fibration $(s,t)\colon \CC^{[1]}\to \CC\times \CC$, where $[1]$ denotes the category associated with the poset $\{0<1\}$, often called the \emph{walking morphism}---so $\CC^{[1]}$ is the category of morphisms and commutative squares of morphisms in $\CC$---and the functor $(s,t)$ is induced by taking source and target. 
\end{Definition}

\begin{Remark} \label{rmk: image of identity}
    Note that, given a category $\CC$, the image of the identity two-sided discrete fibration $(s,t)\colon \CC^{[1]}\to \CC\times \CC$ under the equivalence $\fib\colon \TSdiscFib(\CC,\CC)\to \Prof(\CC,\CC)$ is the identity profunctor $e_\CC=\CC(-,-)\colon \CC^{\op}\times \CC\to \Set$. 
\end{Remark}

As a consequence, we get the following. 

\begin{Lemma}
\label{prop: composition of TS is unital}
    Given a two-sided discrete fibration $(P,Q)\colon \EE\to \CC\times \CC'$, there are canonical isomorphisms in $\TSdiscFib(\CC,\CC')$
    \[
    \begin{tikzcd}
        \EE \arrow[rr, "\cong"] \arrow[rd, "{(P,Q)}"'] &  & {\CC^{[1]}\times_\CC\EE_{/\sim}} \arrow[ld, "{(P,Q)\bullet(s,t)}"] 
        \\
        & \CC\times\CC' &
    \end{tikzcd}
    \quad\quad \text{and} \quad\quad
    \begin{tikzcd}
        \EE \arrow[rr, "\cong"] \arrow[rd, "{(P,Q)}"'] &  & {\EE\times_{\CC'}{\CC'^{[1]}}_{/\sim}} \arrow[ld, "{(s,t)\bullet(P,Q)}"] \\
         & \CC\times\CC' &  
    \end{tikzcd}
    \]
\end{Lemma}

\begin{proof}
    We show the first isomorphism, and the second can be shown analogously.

    Let $(s,t)\colon \CC^{[1]}\to \CC\times \CC$ denote the identity two-sided discrete fibration. We have the following natural isomorphisms 
    \begin{align*}
        \fib((P,Q)\bullet (s,t)) & \cong\fib(P,Q)\bullet\fib(s,t) & \text{\cref{prop: fib compatible with composition}} \\
        & \cong\fib(P,Q)\bullet e_\CC & \text{\cref{rmk: image of identity}} \\
        & \cong \fib(P,Q) & \text{\cref{lem: comp of prof is unital}}
    \end{align*}
   Since $\fib$ is an equivalence, it reflects isomorphisms and so we get an isomorphism $\CC^{[1]}\times_\CC\EE_{/\sim}\cong \EE$ over $\CC\times \CC'$, as desired.
\end{proof}

\subsection{The weak double category \texorpdfstring{$\dCat$}{Cat}}

\label{subsec: dCat}

With the notion of profunctors at hand, we are now ready to introduce the weak double category $\dCat$. 

\begin{Definition}
    We define $\dCat$ to be the unitary weak double category whose
    \begin{itemize}[leftmargin=1cm]
        \item objects are (small) categories $\CC,\CC',\DD,\DD',\ldots$, 
        \item horizontal morphisms $\CC\to \DD$ are functors $F\colon \CC\to \DD$, with their ordinary identity and composition,
        \item vertical morphisms $\CC\bulletarrow \CC'$ are profunctors $U\colon \CC^{\op}\times \CC'\to \Set$, with identity at $\CC$ given by the hom functor $e_\CC\coloneqq \CC(-,-)\colon \CC^{\op}\times \CC\to \Set$ and composition defined as in \cref{def: comp of prof}; using \cref{lem: comp of prof is unital}, we impose $e_\CC\bullet U\coloneqq U$ and $U\bullet e_{\CC'}\coloneqq U$, for every profunctor $U\colon \CC\bulletarrow \CC'$, 
        \item squares
        \[
        \begin{tikzcd}
            \CC \arrow[d,,"\bullet" marking, "U"'] \arrow[r, "F"] \arrow[rd,phantom, "\alpha"] & \DD \arrow[d,"\bullet" marking, "V"] \\
            \CC' \arrow[r, "F'"']& \DD'
        \end{tikzcd}
        \]
        are natural transformations
        \[
        \begin{tikzcd}
            \CC^{\op}\times\CC' \arrow[rr, "F^{\op}\times F'"] \arrow[rd, "U"'] &      & \DD^{\op}\times\DD' \arrow[ld, "V"] \\
            {} \arrow[rru,shorten <=50pt, near end, "\alpha", Rightarrow] & \Set & 
        \end{tikzcd}
        \]
        with horizontal composition of horizontally composable squares
        $\alpha\colon\edgesquare{U}{F}{F'}{V}$ and
        $\beta\colon\edgesquare{V}{G}{G'}{W}$ given by the natural transformation
        \[
        (F^{\op}\times F')\beta\circ \alpha\colon U\Rightarrow
        W\circ ((GF)^{\op}\times (G'F')),
        \]
        and vertical composition of vertically composable squares
        $\alpha\colon\edgesquare{U}{F}{F'}{V}$ and
        $\alpha'\colon\edgesquare{U'}{F'}{F''}{V'}$ given by the natural transformation
        \[
        \alpha'\bullet\alpha\colon
        U'\bullet U\Rightarrow (V'\bullet V)\circ (F^{\op}\times F'')
        \]
        whose component at an object $(x,x'')$ of $\CC\times \CC''$ is the unique map between coends
        \[
            \gro^{x'\in\CC'} \alpha'_{x',x''}\times \alpha_{x,x'}\colon           
            \gro^{x'\in\CC'} U'(x',x'')\times U(x,x')\to
            \gro^{y'\in\DD'} V'(y',F''x'')\times V(Fx,y').
            \]
        induced by $\alpha'_{x',x''}\colon U'(x',x'')\to V'(F'x',F''x'')$ and $\alpha_{x,x'}\colon U(x,x')\to V(Fx,F'x')$,
        \item for composable profunctors $U,U',U''$, the associator square
        \[ \alpha_{U,U',U''}\colon\edgesquare{(U''\bullet U')\bullet U}{1_{\CC}}{1_{\CC'''}}{U''\bullet (U'\bullet U)} \]
        is the natural isomorphism
        \[
        \alpha_{U,U',U''}\colon
        ((U''\bullet U')\bullet U)\xRightarrow{\cong}
        (U''\bullet (U'\bullet U))
        \]
        whose component at an object $(x,x''')$ in $\CC\times \CC'''$ is the unique isomorphism between coends
        \[ 
            (\alpha_{U,U',U''})_{x,x'''}
            \colon
            ((U''\bullet U')\bullet U)(x,x''')\xrightarrow{\cong}
            (U''\bullet (U'\bullet U))(x,x''').
        \]
    \end{itemize}
\end{Definition}

The following result is mentioned in \cite[\textsection 3.4.3]{Grandis2019}, and can be deduced from \cite[\textsection 2]{Benabou}.

\begin{Lemma}
    The construction $\dCat$ is a unitary weak double category. 
\end{Lemma}

Finally, we give a more convenient description of the globular squares in $\dCat$.

\begin{Lemma} \label{globular squares in Cat}
    A globular square $\alpha\colon \edgesquare{e_{\CC}}{F}{F'}{e_{\DD}}$ in $\dCat$ is equivalently a natural transformation $\alpha\colon F\Rightarrow F'\colon \CC\to \DD$. 
\end{Lemma}

\begin{proof}
    By definition, a globular square $\alpha\colon \edgesquare{e_{\CC}}{F}{F'}{e_{\DD}}$ in $\dCat$ is a natural transformation 
    \[ \alpha\colon\CC(-,-)\Rightarrow\DD(F(-),F'(-))\colon \CC^{\op}\times \CC\to \Set, \]
    corresponding uniquely to a family of natural transformations
    $\{\alpha_{-,x}\colon\CC(-,x)\Rightarrow\DD(F(-),F'x)\}_{x\in\CC}$ natural in $x$. By the usual Yoneda lemma, such a family corresponds then uniquely to a family 
    $\{\alpha_{x}\in\DD(Fx,F'x)\}_{x\in\CC}$ natural in $x$. But such a family defines precisely the components of a natural transformation $\alpha\colon F\Rightarrow F'$.
\end{proof}

\begin{Remark} \label{identity 2cell}
    Under the bijection of \cref{globular squares in Cat}, the identity natural transformation $\id_F$ at a functor $F\colon \CC\to \DD$ corresponds to the globular square $\id_F\colon \edgesquare{e_{\CC}}{F}{F}{e_{\DD}}$ in $\dCat$ given by the natural transformation $\id_F\colon \CC(-,-)\Rightarrow \DD(F(-),F(-))\colon \CC^{\op}\times \CC\to \Set$ whose component at objects $x,x'$ in $\CC$ 
    is the map of sets 
    \[ \CC(x,x')\to \DD(Fx,Fx') \]
    sending a morphism $g\colon x\to x'$ in $\CC$ to the morphism $Fg\colon Fx\to Fx'$ in $\DD$.
\end{Remark}

\begin{Remark} \label{rem: underlying 2-cat of dCat}
    As a consequence of \cref{globular squares in Cat}, we see that the underlying horizontal $2$-category $\HH\dCat$ is simply the $2$-category $\Cat$ itself. 
\end{Remark}

\subsection{The \texorpdfstring{$2$}{2}-category of lax double presheaves}

\label{subsec: lax presheaf}

Having constructed the weak double category~$\dCat$, we can now introduce our notion of lax double presheaves. 

\begin{Definition}
    Given a (strict) double category $\dC$, a \emph{lax double presheaf} over $\dC$ is a normal lax double functor $X\colon\dC^{\op}\to\dCat$.
\end{Definition}

\begin{Notation}
\label{not: PC}
    We denote by 
    $\P{C}\coloneqq \HH\dHomnLax{\dC^{\op},\dCat}$ the $2$-category of lax double presheaves over~$\dC$, horizontal transformations, and globular modifications. Here $\dHomnLax{\dC^{\op},\dCat}$ is the full double subcategory of the double category $\dHomLax{\dC^{\op},\dCat}$ from \cref{prop:laxhom} spanned by the normal lax double functors and, consequently, $\HH\dHomnLax{\dC^{\op},\dCat}$ is the corresponding $2$-subcategory of the $2$-category $\HH\dHomLax{\dC^{\op},\dCat}$ from \cref{prop:laxhom}.
\end{Notation}

In the case where $\dC=\dH\CC$ with $\CC$ a $2$-category, we can provide another description of $\P{\dH\CC}$.

\begin{Notation}
    Given a $2$-category $\CC$, we denote by $[\CC^{\op},\Cat]$ the $2$-category of $2$-presheaves, i.e., $2$-functors $\CC^{\op}\to \Cat$, $2$-natural transformations, and modifications; see e.g.~\cite[\textsection B.2]{Elements}.
\end{Notation}

\begin{Lemma} \label{lem: PC for 2-categories}
    Given a $2$-category $\CC$, there is an isomorphism of $2$-categories 
    \[ \P{\dH\CC}\cong [\CC^{\op},\Cat], \]
    which is natural in $\CC$. 
\end{Lemma}

\begin{proof}
    By \cref{2-adjunctions}, we have a $2$-adjunction 
    \[ \dH\colon 2\Cat\leftrightarrows \DblCat_h\colon \HH. \]
    Hence this gives an isomorphism of categories 
    \[ \Hor_0\dHom{\dH\CC^{\op},\dCat} \cong 2\Cat(\CC^{\op},\HH\dCat) \]
    which can be promoted to an isomorphism of $2$-categories
    \[ \HH\dHom{\dH\CC^{\op},\dCat}\cong [\CC^{\op},\HH\dCat]. \]
    Then, noticing that normal lax double functors out of $\dH\CC^{\op}$ are simply (strict) double functors and using \cref{rem: underlying 2-cat of dCat}, we get an isomorphism of $2$-categories 
    \[ \P{\dH\CC}=\HH\dHomnLax{\dH\CC^{\op},\dCat}= \HH\dHom{\dH\CC^{\op},\dCat}\cong [\CC^{\op},\HH\dCat]=[\CC^{\op},\Cat],  \]
    as desired.
\end{proof}

To simplify computations later, we unpack here the data of the $2$-category $\P{C}$. 

\begin{Remark}
A lax double presheaf $X\colon \dC^{\op}\to \dCat$ consists of
\begin{itemize}[leftmargin=1cm]
    \item for every object $x$ in $\dC$, a category $Xx$,
    \item for every horizontal morphism 
    $f\colon x\to y$ in $\dC$, a functor $Xf\colon Xy\to Xx$,
    \item for every vertical morphism 
    $u\colon x\bulletarrow x'$ in $\dC$, a profunctor 
    $Xu\colon Xx^{\op}\times Xx'\to\Set$,
    \item for every square $\alpha\colon\edgesquare{u}{f}{f'}{v}$ in $\dC$, a natural transformation
    $X\alpha\colon Xv\Rightarrow Xu (Xf^{\op}\times Xf')$, 
    \item for all composable vertical morphisms $x\overset{u}{\bulletarrow}x'\overset{u'}{\bulletarrow}x''$ in $\dC$, a composition comparison natural transformation 
\[ \mu_{u,u'}\colon Xu'\bullet Xu\Rightarrow X(u'\bullet u)\colon Xx^{\op}\times Xx''\to\Set. \]
\end{itemize}
such that the following conditions are satisfied:
\begin{enumerate}[leftmargin=1cm]
    \item it preserves horizontal identities and horizontal compositions strictly,
    \item composition comparison natural transformations $\mu_{u,u'}$ are natural with respect to $(u,u')$ and compatible with the associator natural transformations in $\dCat$,
    \item it preserves vertical identities strictly; and, for every vertical morphism $u\colon x\bulletarrow x'$ in $\dC$, the natural transformations $\mu_{u,e_{x'}}$ and $\mu_{e_x,u}$ agree with the identity natural transformation at $Xu$.
\end{enumerate}
\end{Remark}

\begin{Remark}
Given lax double presheaves $X, Y\colon \dC^{\op}\to \dCat$, then a horizontal transformation $F\colon X\Rightarrow Y$ consists of
    \begin{itemize}[leftmargin=1cm]
        \item for every object $x$ in $\dC$, a functor $F_x\colon Xx\to Yx$,
        \item for every vertical morphism 
        $u\colon x\bulletarrow x'$ in $\dC$, a natural transformation
        \[ F_u\colon Xu\Rightarrow Yu\circ(F_x^{\op}\times F_{x'})\colon Xx^{\op}\times Xx'\to\Set, \]
    \end{itemize}
   such that the following conditions are satisfied:
    \begin{enumerate}[leftmargin=1cm]
        \item the functors $F_x$ are natural in $x$: for every horizontal morphism $f\colon x\to y$ in $\dC$, the following diagram of functors commutes
        \[
        \begin{tikzcd}
            Xy \arrow[r,"Xf"] \arrow[d, "F_y"']                  & Xx \arrow[d, "F_x"]                 \\
            Yy \arrow[r, "Yf"']  & Yx 
        \end{tikzcd}
        \]
        \item the natural transformations 
        $F_u$ are natural in $u$: for every square $\alpha\colon\edgesquare{u}{f}{f'}{v}\colon\nodesquare{x}{x'}{y}{y'}$ in $\dC$, the following diagram in $\Prof(Xy, Xy')$ commutes,
        \[
        \begin{tikzcd}
            Xv \arrow[d, "F_v"', Rightarrow] \arrow[r, "X\alpha", Rightarrow]                  
            &[40pt] Xu(Xf^{\op}\times Xf')
             \arrow[d, "F_u(Xf^{\op}\times Xf')", Rightarrow] \\
            Yv(F_{y}^{\op}\times F_{y'}) \arrow[r, "Y\alpha(F_{y}^{\op}\times F_{y'})"', Rightarrow] 
            &[30pt]
            Yu((Yf\circ F_y)^{\op}\times (Yf'\circ F_{y'}))              \end{tikzcd}
        \]
        where we use that $Yf\circ F_y=F_x\circ Xf$ and $Yf'\circ F_{y'}=F_{x'}\circ Xf'$ by (1),  
        \item the natural transformations $F_u$ are compatible with composition comparisons:
        for all composable vertical morphisms 
        $x\overset{u}{\bulletarrow} x'\overset{u'}{\bulletarrow}x''$ in $\dC$, the following diagram in 
        $\Prof(Xx, Xx'')$ commutes,
        \[
        \begin{tikzcd}
            Xu'\bullet Xu \arrow[d, "F_{u'}\bullet F_u"', Rightarrow] \arrow[r, "{\mu_{u,u'}}", Rightarrow] & [40pt] X(u'\bullet u)  \arrow[d, "F_{u'\bullet u}", Rightarrow] \\ (Yu'\bullet Yu) (F_x^{\op}\times F_{x''})
             \arrow[r, "{\mu_{u,u'}  (F_x^{\op}\times F_{x''})}"' {yshift=-2pt}, Rightarrow]                                        & Y(u'\bullet u) (F_x^{\op}\times F_{x''})
        \end{tikzcd}
        \]
        \item the natural transformations $F_u$ are compatible with vertical identities: for every object~$x$ in~$\dC$, the natural transformation $F_{e_x}$ is the identity at $Fx$.
    \end{enumerate}
\end{Remark}

\begin{Remark}
Given horizontal transformations $F,F'\colon X\Rightarrow Y\colon \dC^{\op}\to \dCat$, then a globular modification $A\colon\edgesquare{e_X}{F}{F'}{e_{Y}}$ consists of, for every object $x$ in $\dC$, a natural transformation
\[ A_x\colon F_x\Rightarrow F'_x\colon Xx\to Yx, \] 
such that the following conditions are satisfied:
\begin{enumerate}[leftmargin=1cm]
    \item the natural transformations $A_x$ are compatible with horizontal morphisms: for every horizontal morphism $f\colon x\to y$ in $\dC$, the following pasting diagram of functors and natural transformations commutes,
    \[
    \begin{tikzcd}
        & Xy \arrow[rr, "Xf"] \arrow[dd, "F_y"', bend right] \arrow[dd, "F'_y", bend left] && Xx \arrow[dd, "F_x"', bend right] \arrow[dd, "F'_x", bend left] &    
        \\
        {} \arrow[rr,shorten <=30pt,shorten >=30pt, "A_y", Rightarrow] & & {} \arrow[rr,shorten <=30pt,shorten >=30pt, Rightarrow,"A_x"] &   & {} 
        \\
        & Yy \arrow[rr, "Yf"']& & Yx &   
        \end{tikzcd}
    \]
    \item the natural transformations $A_x$ are compatible with vertical morphisms: for every vertical morphism $u\colon x\bulletarrow x'$ in $\dC$, the following diagram in 
    $\Prof(Xx, Xx')$ commutes.
    \[
    \begin{tikzcd}
        Xu \arrow[r, "F_u", Rightarrow] \arrow[d, "F'_u"', Rightarrow]          &[40pt]
        Yu (F_x^{\op}\times F_{x'}) \arrow[d, "Yu (F_x^{\op}\times A_{x'})", Rightarrow] \\
        Yu ({F'}_{x}^{\op}\times F'_{x'}) \arrow[r, "Yu (A_x^{\op}\times F'_{x'})"', Rightarrow] 
        &[20pt]
        Yu (F_x^{\op}\times F'_{x'}) 
        \end{tikzcd}
    \]
\end{enumerate}
\end{Remark}

Finally, we extend the construction $\P{C}$ into a $2$-functor. 

\begin{Construction} \label{def: 2-functor P}
    We define a $2$-functor
    \[
    \mathcal{P}\colon\DblCat_h^{\coop}\to 2\Cat,
    \]
    where $\DblCat_h^{\coop}$ is the $2$-category obtained from $\DblCat_h$ by reversing the morphisms and the $2$-morphisms, which sends
    \begin{itemize}[leftmargin=1cm]
        \item a double category $\dC$ to the weak double category $\P{C}$ of lax double presheaves over $\dC$, horizontal transformations, and globular modifications,
        \item a double functor $G\colon\dC\to\dD$ to the $2$-functor 
        \[
        \mathcal{P}_G\coloneqq (G^{\op})^*\colon\P{D}\to\P{C},
        \]
        induced by precomposition with $G^{\op}$,
        \item a horizontal transformation 
        $B\colon G\Rightarrow G'\colon\dC\to\dD$ to the $2$-natural transformation
        \[
        \mathcal{P}_B\coloneqq (B^{\op})^*\colon\mathcal{P}_{G'}\Rightarrow\mathcal{P}_{G}
        \]
        whose component at a lax double presheaf $X\colon\dD^{\op}\to\dCat$ is the horizontal transformation \[ (\mathcal{P}_B)_X\coloneqq X\circ B^{\op}\colon\mathcal{P}_{G'}(X)=X\circ G'^{\op}\to\mathcal{P}_G(X)=X\circ G^{\op} \]
        of lax double presheaves $\dC^{\op}\to\dCat$.
    \end{itemize}
    It is straightforward to check that this construction is $2$-functorial.
\end{Construction}

\subsection{Representable lax double presheaves}
\label{subsec: repr presheaf}

We now want to introduce our first class of examples. They will be induced by the following construction. 

\begin{Construction}
    Given a double category $\dC$, we define a normal lax double functor \[ \dC(-,-)\colon\dC^{\op}\times\dC\to\dCat \]
    by the following data:
    \begin{itemize}[leftmargin=1cm]
        \item it sends a pair $(x,\hat{x})$ of objects in $\dC$ to the category $\dC(x,\hat{x})$ whose 
        \begin{itemize}
            \item objects are horizontal morphisms $g\colon x\to\hat{x}$ in $\dC$,
            \item morphisms $g\to g'$ are globular squares $\eta\colon \edgesquare{e_{x}}{g}{g'}{e_{\hat{x}}}$ in $\dC$, 
            \item composition is given by the vertical composition of squares in $\dC$
        \end{itemize}
        \item it sends a pair $(x\xrightarrow{f}y,\hat{x}\xrightarrow{\hat{f}}\hat{y})$ of horizontal morphisms in $\dC$ to the functor
        \[
        \dC(f,\hat{f})\colon\dC(y,\hat{x})\to\dC(x,\hat{y})
        \]
        sending
        \begin{itemize}
            \item a horizontal morphism $g\colon y\to\hat{x}$ to the composite $\hat{f}\circ g\circ f\colon x\to\hat{y}$,
            \item a square $\eta\colon\edgesquare{e_y}{g}{g'}{e_{\hat{x}}}$ to the square 
            $e_{\hat{f}}\circ\eta\circ e_f\colon\edgesquare{e_x}{\hat{f}\circ g\circ f}{\hat{f}\circ g'\circ f}{e_{\hat{y}}}$,
        \end{itemize}
        \item it sends a pair $(x\overset{u}{\bulletarrow}x',\hat{x}\overset{\hat{u}}{\bulletarrow}\hat{x}')$ of vertical morphisms in $\dC$ to the profunctor
        \[
        \dC(u,\hat{u})\colon\dC(x,\hat{x})^{\op}\times\dC(x',\hat{x}')\to\Set
        \]
        sending
        \begin{itemize}
            \item an object 
            $(x\xrightarrow{g}\hat{x},x'\xrightarrow{g'}\hat{x}')$ of $\dC(x,\hat{x})\times\dC(x',\hat{x}')$ to the set of squares 
            $\left\{\edgesquare{u}{g}{g'}{\hat{u}}\colon\nodesquare{x}{x'}{\hat{x}}{\hat{x}'}\right\}$,
            \item a pair of squares 
            $\left(\alpha\colon\edgesquare{e_x}{g}{h}{e_{\hat{x}}},\alpha'\colon\edgesquare{e_{x'}}{g'}{h'}{e_{\hat{x}'}}\right)$ to the map 
            \[ \dC(u,\hat{u})(h,g')\to\dC(u,\hat{u})(g,h') \]
            sending a square $\eta\colon\edgesquare{u}{h}{g'}{\hat{u}}$ to the vertical composite of squares
             \[
            \begin{tikzcd}
            x 
        \arrow[d,"\bullet" marking, equal] \arrow[r, "g"] 
        \arrow[rd, "\alpha", phantom]           & 
        \hat{x} \arrow[d,"\bullet" marking,equal] \\
        x \arrow[d,swap,"\bullet" marking, "u"] \arrow[r, "h"] 
        \arrow[rd, "\eta", phantom]                
        & 
        \hat{x}
        \arrow[d,"\bullet" marking, "\hat{u}"]   \\
        x' \arrow[r, "g'"] 
        \arrow[d,"\bullet" marking, no head, equal] 
        \arrow[rd, "\alpha'", phantom] 
        &
        \hat{x}'
        \arrow[d,"\bullet" marking, equal] \\
        x' \arrow[r, "h'"']                          & \hat{x}' 
        \end{tikzcd}
        \]
        \end{itemize}
        \item it sends a pair $(\alpha\colon\edgesquare{v}{f}{f'}{u},\hat{\alpha}\colon\edgesquare{\hat{u}}{\hat{f}}{\hat{f}'}{\hat{v}})$ of squares in $\dC$ to the natural transformation 
        \[
        \dC(\alpha,\hat{\alpha})\colon\dC(u,\hat{u})\Rightarrow\dC(v,\hat{v})(\dC(f,\hat{f})^{\op}\times\dC(f',\hat{f}'))\colon\dC(y,\hat{x})^{\op}\times\dC(y',\hat{x}')\to\Set
        \] whose component at an object 
        $(y\xrightarrow{g}\hat{x},y'\xrightarrow{g'}\hat{x}')$ of 
        $\dC(y,\hat{x})\times\dC(y',\hat{x}')$ is the map
        \[
        \dC(\alpha,\hat{\alpha})_{g,g'}\colon\dC(u,\hat{u})(g,g')\to\dC(v,\hat{v})(\hat{f}\circ g\circ f,\hat{f}'\circ g'\circ f'),
        \]
        sending a square $\eta\colon\edgesquare{u}{g}{g'}{\hat{u}}$ to the horizontal composite of squares
       \[ 
        \begin{tikzcd}
            x \arrow[r, "f"] \arrow[d,"\bullet" marking, "v"'] \arrow[rd,phantom, "\alpha"] & y \arrow[d,"\bullet" marking, "u"'] \arrow[r, "g"] \arrow[rd,phantom, "\eta"] & \hat{x} \arrow[d, "\bullet" marking,"\hat{u}"] \arrow[r, "\hat{f}"] \arrow[rd,phantom, "\hat{\alpha}"] & \hat{y} \arrow[d,"\bullet" marking, "\hat{v}"] \\
            x' \arrow[r, "f'"']    & y' \arrow[r,"g'"']   & \hat{x}' \arrow[r, "\hat{f}'"'] & \hat{y}'                    
        \end{tikzcd}
        \]
        \item its composition comparison square at a pair 
        $(x,\hat{x})\overset{(u,\hat{u})}{\bulletarrow}(x',\hat{x}')\overset{(u',\hat{u}')}{\bulletarrow}(x'',\hat{x}'')$ of composable vertical morphisms in 
        $\dC$ is the natural transformation
        \[
        \mu_{(u,\hat{u}),(u',\hat{u}')}\colon
        \dC(u',\hat{u}')\bullet\dC(u,\hat{u})\Rightarrow
        \dC(u'\bullet u,\hat{u}'\bullet\hat{u})\colon\dC(x,\hat{x})^{\op}\times\dC(x'',\hat{x}'')\to\Set,
        \]
        whose component at an object 
        $(x\xrightarrow{g}\hat{x},x''\xrightarrow{g''}\hat{x}'')$ of $\dC(x,\hat{x})\times\dC(x'',\hat{x}'')$ is the map
        \[
        \dC(u',\hat{u}')\bullet\dC(u,\hat{u})(g,g'')\to
        \dC(u'\bullet u,\hat{u}'\bullet u')(g,g''),
        \]
        induced by the universal property of the coend from the family of maps, given at an object $x'\xrightarrow{g'}\hat{x}'$ of $\dC(x',\hat{x}')$, by the map
        \[
        \dC(u',\hat{u}')(g',g'')\times\dC(u,\hat{u})(g,g')\longrightarrow
        \dC(u'\bullet u,\hat{u}'\bullet\hat{u})(g,g''),
        \]
       sending an element $(\alpha'\colon \edgesquare{u'}{g'}{g''}{\hat{u}'},\alpha\colon \edgesquare{u}{g}{g'}{\hat{u}})$ of $\dC(u',\hat{u}')(g',g'')\times\dC(u,\hat{u})(g,g')$ to the vertical composite $\alpha'\bullet\alpha$ in $\dC(u'\bullet u,\hat{u}'\bullet\hat{u})(g,g'')$.
    \end{itemize}
\end{Construction}

\begin{Lemma}
    The construction 
    \[ \dC(-,-)\colon\dC^{\op}\times\dC\to\dCat \]
    is a normal lax double functor.
\end{Lemma}

\begin{proof}
    It is straightforward to see that $\dC(-,-)\colon\dC^{\op}\times\dC\to\dCat$ preserves horizontal compositions and identities. The fact that, given vertical morphisms $u,\hat{u}$ in $\dC$, the composition comparison natural transformations $\mu_{u,u'}$ are natural in $(u,u')$ follows from the interchange law that the compositions of squares in $\dC$ satisfy, and the fact that the natural transformations $\mu_{u,u'}$ are compatible with associator natural transformations follows from the associativity of vertical composition of squares in $\dC$. Finally, to see that $\dC(-,-)$ is normal, observe that, given objects $x,\hat{x}$ in $\dC$, then $\dC(e_x,e_{\hat{x}})$ is the identity profunctor $e_{\dC(x,\hat{x})}$, and similarly, given horizontal morphisms $f,\hat{f}$ in $\dC$, then $\dC(e_f,e_{\hat{f}})$ is the identity natural transformation at $\dC(f,\hat{f})$.
\end{proof}

\begin{Remark}
    We note that $\dC(-,-)$ is in general not a pseudo double functor. If this was the case, the composition comparison natural transformations \[
        \mu_{(u,\hat{u}),(u',\hat{u}')}\colon
        \dC(u',\hat{u}')\bullet\dC(u,\hat{u})\Rightarrow
        \dC(u'\bullet u,\hat{u}'\bullet\hat{u})\colon\dC(x,\hat{x})^{\op}\times\dC(x'',\hat{x}'')\to\Set,
        \]
        would be invertible. This would imply that every square
    $\edgesquare{u'\bullet u}{g}{g''}{\hat{u}'\bullet\hat{u}}$ in $\dC$ could be factorized into a vertical composite of two squares of the form
    $\edgesquare{u}{g}{g'}{\hat{u}}$ and 
    $\edgesquare{u'}{g'}{g''}{\hat{u}'}$.
    However, such a factorization does not exist in general, as we show in \cref{ex: no factorization} below. 
\end{Remark}

\begin{Example}\label{ex: no factorization}
    We consider the double category $\dC$ generated by the following data
    \[
    \begin{tikzcd}
        x \arrow[d,"\bullet" marking, "u"'] \arrow[r, "g"] \arrow[rdd,phantom, "\alpha"] & \hat{x} \arrow[d, "\bullet" marking, "\hat{u}"]   \\
        x' \arrow[d,"\bullet" marking, "u'"'] & \hat{x}' \arrow[d,"\bullet" marking, "\hat{u}'"] \\
        x'' \arrow[r, "g''"'] & \hat{x}'' 
    \end{tikzcd}
    \]
    Then, note that we have 
    \[ \dC(x',\hat{x}')=\varnothing, \quad \dC(x,\hat{x})=\{g\} \quad \text{and} \quad \dC(x'',\hat{x}'')=\{g''\}. \]
    Hence we get that
    \[
    \dC(u',\hat{u}')\bullet\dC(u,\hat{u})(g,g'')
    =
    \gro^{g'\in\dC(x',\hat{x}')=\varnothing}\dC(u',\hat{u}')(g',g'')\times\dC(u,\hat{u})(g,g')=\varnothing,
    \]
    while 
    \[ \dC(u'\bullet u,\hat{u}'\bullet\hat{u})(g,g'')=\{\alpha\}. \]
    Hence $(\mu_{(u,\hat{u}),(u',\hat{u}')})_{g,g''}$ is not a bijection, and therefore $\mu_{(u,\hat{u}),(u',\hat{u}')}$ is not invertible.  
\end{Example}

By fixing components, we get the following.

\begin{Construction}
    Given a double category $\dC$ and an object~$\hat{x}$ in $\dC$, the \emph{representable lax double presheaf at $\hat{x}$} is the composite of normal lax double functors
    \[ \dC(-,\hat{x})\colon \dC^{\op}\cong \dC^{\op}\times [0] \xrightarrow{\dC^{\op}\times \hat{x}}\dC^{\op}\times \dC \xrightarrow{\dC(-,-)}\dCat. \]
\end{Construction}

This construction can be extended on horizontal morphisms and globular squares as follows.

\begin{Construction}
    Given a double category $\dC$ and a horizontal morphism $\hat{f}\colon \hat{x}\to \hat{y}$ in $\dC$, the \emph{representable horizontal transformation at $\hat{f}$}
    \[ \dC(-,\hat{f})\colon\dC(-,\hat{x})\Rightarrow\dC(-,\hat{y})\colon \dC^{\op}\to\dCat \]
    is the horizontal transformation whose 
    \begin{itemize}[leftmargin=1cm]
        \item component at an object $x$ in $\dC$ is the functor $\dC(x,\hat{f})\coloneqq \dC(1_x,\hat{f})\colon \dC(x,\hat{x})\to \dC(x,\hat{y})$, 
        \item component at a vertical morphism $u\colon x\bulletarrow x'$ in $\dC$ is the natural transformation 
        \[
        \dC(u,\hat{f})\coloneqq \dC(1_u,e_{\hat{f}})\colon
        \dC(u,e_{\hat{x}})\Rightarrow
        \dC(u,e_{\hat{y}}) (\dC(x,\hat{f})^{\op}\times\dC(x',\hat{f}))
        \colon
        \dC(x,\hat{x})^{\op}\times
        \dC(x',\hat{x})\to\Set.
        \]
    \end{itemize}
    The fact that $\dC(-,\hat{f})$ is a horizontal transformation follows from the coherences of $\dC(-,-)$.  
\end{Construction}

\begin{Construction}
    Given a double category $\dC$ and a globular square $\hat{\alpha}\colon\edgesquare{e_{\hat{x}}}{\hat{f}}{\hat{f}'}{e_{\hat{x}'}}$ in $\dC$, the \emph{representable globular modification at $\hat{\alpha}$}
    \[ \dC(-,\hat{\alpha})\colon\edgesquare{e_{\dC(-,\hat{x})}}{\dC(-,\hat{f})}{\dC(-,\hat{f}')}{e_{\dC(-,\hat{x}')}} \]  
    is the globular modification whose component at an object $x$ in $\dC$ is the natural transformation
    \[ \dC(x,\hat{\alpha})\coloneqq \dC(e_{1_x},\hat{\alpha})\colon
    \dC(x,\hat{f})\Rightarrow
    \dC(x,\hat{f}').
    \]
    The fact that $\dC(-,\hat{\alpha})$ is a globular modification follows from the coherences of $\dC(-,-)$.
\end{Construction}

However, this construction fails to extend to vertical morphisms. 

\begin{Remark}
\label{rem: repr. not natural in vert. morphisms}
    Given a double category $\dC$ and a vertical morphism $\hat{u}\colon \hat{x}\bulletarrow\hat{x}'$ in $\dC$, we could consider the ``vertical transformation'' whose 
    \begin{itemize}[leftmargin=1cm]
        \item component at an object $x$ in $\dC$ is the profunctor 
        \[
        \dC(x,\hat{u})\coloneqq\dC(e_{x},\hat{u})\colon
        \dC(x,\hat{x})^{\op}\times
        \dC(x,\hat{x}')\to\Set,
        \]
        \item component at a horizontal morphism $f\colon x\to y$ in $\dC$ is the natural transformation
        \[
        \dC(f,\hat{u})\coloneqq\dC(e_f,1_{\hat{u}})\colon
        \dC(y,\hat{u})\Rightarrow
        \dC(x,\hat{u})        (\dC(f,\hat{x})^{\op}\times
        \dC(f,\hat{x}')).
        \]
    \end{itemize}
    One might expect to get a colax vertical transformation $\dC(-,\hat{u})\colon \dC(-,\hat{x})\Bulletarrow \dC(-,\hat{x}')$. However, there are in general no naturality comparison squares. Indeed, given a vertical morphism $u\colon x\bulletarrow x'$ in~$\dC$, such a naturality comparison square would be a natural transformation of the form
    \[ \dC(x',\hat{u})\bullet \dC(u,\hat{x})\Rightarrow \dC(u,\hat{x}')\bullet \dC(x,\hat{u})\colon \dC(x,\hat{x})^{\op}\times \dC(x',\hat{x}')\to \Set. \]
    Thus, when evaluated at an object 
    $(x\xrightarrow{g}\hat{x},x'\xrightarrow{g'}\hat{x}')$ in $\dC(x,\hat{x})\times \dC(x',\hat{x}')$, we would get a map
    \begin{equation}
    \label{eq: C(-,u) not vertical}
    \begin{matrix}
        \gro^{s\in \dC(x',\hat{x})}\dC(x',\hat{u})(s,g')\times\dC(u,\hat{x})(g,s)
        &
        \longrightarrow
        &
        \gro^{t\in\dC(x,\hat{x}')}\dC(u,\hat{x}')(t,g')\times\dC(x,\hat{u})(g,t).
    \end{matrix}
    \end{equation}
    But in general, such a map does not exist, as we show in \cref{ex: rep at u is not pseudo} below. 

The core issue here is that elements in the left-hand set of \eqref{eq: C(-,u) not vertical} are represented by pairs $(\alpha,\alpha')$ of squares as depicted below left for some horizontal morphism $s\colon x'\to\hat{x}$ in $\dC$, while elements of the right-hand set of \eqref{eq: C(-,u) not vertical} are represented by pairs $(\beta,\beta')$ of squares as depicted below right for some horizontal morphism $t\colon x\to \hat{x}'$. 
    \[
    \begin{tikzcd}
        x \arrow[d,"\bullet" marking,swap, "u"] \arrow[r, "g"] \arrow[rd, phantom,"\alpha"] & \hat{x} \arrow[d,"\bullet" marking,equal] \\
      x' \arrow[d,"\bullet" marking,equal] \arrow[r, "s"] \arrow[rd, phantom,"\alpha'"]  & \hat{x} \arrow[d,"\bullet" marking,"\hat{u}"] \\
        x' \arrow[r, swap,"g'"]                    & \hat{x}',              
    \end{tikzcd} \quad \quad \quad \quad \begin{tikzcd}
        x \arrow[d, "\bullet" marking,equal] \arrow[r, "g"] \arrow[rd,phantom, "\beta"]         & \hat{x} \arrow[d,"\bullet" marking, "\hat{u}"] \\
        x \arrow[d,"\bullet" marking, "u"'] \arrow[r, "t"] \arrow[rd,phantom, "\beta'"] & \hat{x}' \arrow[d,"\bullet" marking,equal]           \\
        x' \arrow[r, "g'"']                                   & \hat{x}'                    
    \end{tikzcd}
    \]
    But there is no natural assignment between such pairs.
    
    Can we evade this problem by introducing a new notion of vertical transformation? Note that the vertical composites $\alpha'\bullet \alpha$ and $\beta'\bullet \beta$ of the above squares share the same boundaries, so the composition comparison transformations of $\dC(-,-)$ fit into the following picture
    \[
    \begin{tikzcd}
        {\dC(x,\hat{x})} \arrow[rd,"\bullet" marking] \arrow[r,"\bullet" marking, "{\dC(x,\hat{u})}"] \arrow[d, "\bullet" marking,"{\dC(u,\hat{x})}"'] & {\dC(x,\hat{x}')} \arrow[d, "\bullet" marking,"{\dC(u,\hat{x}')}"] \arrow[ld,shorten >=25pt,near start, Rightarrow] \\
        {\dC(x',\hat{x})} \arrow[r,"\bullet" marking, "{\dC(x',\hat{u})}"'] \arrow[ru,shorten >=25pt, near start, Rightarrow]                  & {\dC(x',\hat{x}')},          
    \end{tikzcd}
    \]
    where the diagonal vertical morphism is given by $\mathbb{C}(u,\hat{u})$, which might be a good candidate for replacing the naturality comparison squares. This fixes our issue regarding naturality comparison squares, but introduces yet another problem: such vertical transformations do not compose! Indeed, the fact that the $2$-morphisms in the above diagram face opposite directions prevents the composition of two such transformations.

    These issues are the core of our decision to work in the $2$-category $\P{C}$ of lax double presheaves, in which vertical transformations are not considered, rather than in the double category $\dHomnLax{\dC^{\op},\dCat}$.
\end{Remark}

\begin{Example}

\label{ex: rep at u is not pseudo}
    We consider the double category $\dC$ generated by the following data
    \[
    \begin{tikzcd}
        x \arrow[d,"\bullet" marking,swap, "u"] \arrow[r, "g"] \arrow[rd, phantom,"\alpha"] & \hat{x} \arrow[d,"\bullet" marking,equal] \\
      x' \arrow[d,"\bullet" marking,equal] \arrow[r, "h"] \arrow[rd, phantom,"\alpha'"]  & \hat{x} \arrow[d,"\bullet" marking,"\hat{u}"] \\
        x' \arrow[r, swap,"g'"]                    & \hat{x}',              
    \end{tikzcd}
    \]
    Then, note the we have     
    \[ \dC(x,\hat{x}')=\varnothing, \quad \dC(x,\hat{x})=\{g\}, \quad \dC(x',\hat{x})=\{h\} \quad \text{and} \quad \dC(x',\hat{x}')=\{g'\}. \]
    Hence we get that 
   \[ \gro^{s\in\dC(x',\hat{x})=\{h\}}\dC(x',\hat{u})(s,g')\times\dC(u,\hat{x})(g,s)
        =
        \{(\alpha',\alpha)\} \]
        while
        \[
        \gro^{t\in\dC(x,\hat{x}')=\varnothing}\dC(u,\hat{x}')(t,g')\times\dC(x,\hat{u})(g,t)
        =
        \varnothing. \]
    This implies that a map as in \eqref{eq: C(-,u) not vertical} cannot exist.
\end{Example}

\section{Yoneda lemma}

In this section, we show a $2$-categorical version of the Yoneda lemma for double categories using the weak double category $\dCat$. More precisely, given a double category $\dC$ and a double lax presheaf~$X$ over $\dC$, we prove that there is a $2$-natural isomorphism between the category of morphisms from a representable lax double presheaf $\dC(-,\hat{x})$ to $X$ and the category $X\hat{x}$.

In \cref{subsec: Yoneda1}, we construct the Yoneda comparison map and show that it is $2$-natural. In \cref{subsec: Yoneda2}, we construct its inverse henceforth proving the desired Yoneda lemma. Finally, in \cref{subsec: Yoneda3}, we deduce that there is a Yoneda embedding from the underlying horizontal $2$-category of $\dC$ into its $2$-category of lax double presheaves $\P{C}$.

\subsection{Statement of Yoneda lemma}
\label{subsec: Yoneda1}

We start by constructing the components of the Yoneda isomorphism. 

\begin{Construction} 
\label{constr:yonedafunctor}
    Given a double category $\dC$, a lax double presheaf $X\colon \dC^{\op}\to \dCat$, and an object $\hat{x}\in \dC$, we construct a functor 
    \[ \Psi_{\hat{x},X}\colon \P{C}(\dC(-,\hat{x}),X)\to X\hat{x}.
    \]
    It sends a horizontal transformation $\varphi\colon \dC(-,\hat{x})\Rightarrow X$ to the object $\varphi_{\hat{x}}(1_{\hat{x}})$ of $X \hat{x}$, i.e., the image of the identity $1_{\hat{x}}$ under the functor $\varphi_{\hat{x}}\colon \dC(\hat{x},\hat{x})\to X\hat{x}$. It sends a modification 
    \[
     \begin{tikzcd}
            \dC(-,\hat{x})
            \arrow[r,Rightarrow,"\phi"]
            \arrow[d,equal,"\bullet" marking]
            \arrow[rd,phantom,"\nu"]
            &
            X
            \arrow[d,equal,"\bullet" marking]
            \\
            \dC(-,\hat{x})
            \arrow[r,Rightarrow,"\phi'"']
            &
            X
    \end{tikzcd}
    \]
    to the morphism 
    $(\nu_{\hat{x}})_{1_{\hat{x}}}\colon\phi_{\hat{x}}(1_{\hat{x}})\to\phi'_{\hat{x}}(1_{\hat{x}})$ in $X\hat{x}$, i.e., the component of the natural transformation 
    $\nu_{\hat{x}}\colon\phi_{\hat{x}}\Rightarrow\phi'_{\hat{x}}\colon\dC(\hat{x},\hat{x})\to X\hat{x}$ at $1_{\hat{x}}$.
\end{Construction}

We now state the Yoneda lemma.

\begin{Theorem} \label{thm:Yoneda}
    Given a double category $\dC$, a lax double presheaf $X\colon \dC^{\op}\to \dCat$, and an object~$\hat{x}$ in $\dC$, the functor
    \[ \Psi_{\hat{x},X}\colon \P{C}(\dC(-,\hat{x}),X)\to X\hat{x}
    \]
    is an isomorphism, which is $2$-natural in $\hat{x}$ in $\HH\dC$ and in $X$ in $\P{C}$.
\end{Theorem}

We show that it is $2$-natural in $\hat{x}$ and $X$, separately. 

\begin{Proposition}
    The functors
    \[ \Psi_{\hat{x},X}\colon \P{C}(\dC(-,\hat{x}),X)\to X\hat{x}
    \]
    are $2$-natural in $\hat{x}$ in $\HH\dC$.
\end{Proposition}

\begin{proof}
    To prove $2$-naturality in $\hat{x}$, we first have to show that, given a horizontal morphism $\hat{f}\colon \hat{x}\to\hat{y}$ in $\dC$, the following diagram of functors commutes.
    \[
    \begin{tikzcd}
        {\P{C}(\dC(-,\hat{y}),X)} \arrow[d, "{\dC(-,\hat{f})^*}"'] \arrow[r, "{\Psi_{\hat{y},X}}"] & X\hat{y} \arrow[d, "X\hat{f}"] \\
        {\P{C}(\dC(-,\hat{x}),X)} \arrow[r, "{\Psi_{\hat{x},X}}"']                  & X\hat{x}         
    \end{tikzcd}
    \]\
    To do so, we evaluate both composites at a horizontal transformation
$\phi\colon\dC(-,\hat{y})\Rightarrow X$:
\begin{align*}
    \Psi_{\hat{x},X}(\dC(-,\hat{f})^*(\phi))
    &=
    \Psi_{\hat{x},X}(\phi\circ\dC(-,\hat{f}))
    &
    \text{Definition of }\dC(-,\hat{f})^*
    \\
    &=
    (\phi\circ\dC(-,\hat{f}))_{\hat{x}}(1_{\hat{x}})
    &
    \text{Definition of }\Psi_{\hat{x},X}
    \\
    &=
    \phi_{\hat{x}}(\dC(\hat{x},\hat{f})(1_{\hat{x}}))
    &
    \text{Composition of transformations}
    \\
    &=
    \phi_{\hat{x}}(\hat{f})
    &
    \text{Definition of }\dC(\hat{x},\hat{f})
    \\
    &=
    \phi_{\hat{x}}(\dC(\hat{f},\hat{y})(1_{\hat{y}}))
    &
    \text{Definition of }\dC(\hat{f},\hat{y})
    \\
    &=
    X\hat{f}(\phi_{\hat{y}}(1_{\hat{y}}))
    &
    \text{Naturality of }\phi
    \\
    &=
    X\hat{f}(\Psi_{\hat{y},X}(\phi)),
    &
    \text{Definition of }\Psi_{\hat{y},X}
\end{align*}
as well as at a globular modification 
$\nu\colon\edgesquare{e_{\dC(-,\hat{y})}}{\phi}{\phi'}{e_X}$:
\begin{align*}
    \Psi_{\hat{x},X}(\dC(-,\hat{f})^*(\nu))
    &=
    \Psi_{\hat{x},X}(\nu\circ\dC(-,\hat{f}))
    &
    \text{Definition of }\dC(-,\hat{f})^*
    \\
    &=
    ((\nu\circ\dC(-,\hat{f}))_{\hat{x}})_{1_{\hat{x}}}
    &
    \text{Definition of }\Psi_{\hat{x},X}
    \\
    &=
    (\nu_{\hat{x}})_{\dC(\hat{x},\hat{f})(1_{\hat{x}})}
    &
    \text{Whiskering}
    \\
    &=
    (\nu_{\hat{x}})_{\hat{f}\circ 1_{\hat{x}}}
    &
    \text{Definition of }\dC(\hat{x},\hat{f})
    \\
    &=
    (\nu_{\hat{x}})_{1_{\hat{y}}\circ\hat{f}}
    &
    \hat{f}\circ 1_{\hat{x}}=\hat{f}=1_{\hat{y}}\circ\hat{f}
    \\
    &=
    (\nu_{\hat{x}})_{\dC(\hat{f},\hat{y})(1_{\hat{y}})}
    &
    \text{Definition of }\dC(\hat{f},\hat{y})
    \\
    &=
    X\hat{f}((\nu_{\hat{y}})_{1_{\hat{y}}})
    &
    \text{Horizontal compatibility of }\nu
    \\
    &=
    X\hat{f}(\Psi_{\hat{y},X}(\nu)).
    &
    \text{Definition of }\Psi_{\hat{y},X}
\end{align*}
This shows that the two composites agree on objects and morphisms, as desired. 

Next, we have to show that, given a globular square $\hat{\alpha}\colon\edgesquare{e_{\hat{x}}}{\hat{f}}{\hat{f'}}{e_{\hat{y}}}$ in $\dC$, the following pasting diagram of functors and natural transformations commutes.
    \[
    \begin{tikzcd} 
& {\P{C}(\dC(-,\hat{y}),X)} \arrow[dd, swap,"{\dC(-,\hat{f})^*}", bend right=49] \arrow[dd, "{\dC(-,\hat{f}')^*}", bend left=49] \arrow[rr, "{\Psi_{\hat{y},X}}"] &      
& X\hat{y} \arrow[dd, swap,"X\hat{f}", bend right =49] \arrow[dd, "X\hat{f}'", bend left=49] &    
\\
{} \arrow[rr, "{\dC(-,\hat{\alpha})^*}", shorten <=50pt, shorten >=50pt,Rightarrow] 
&                     
&
{} \arrow[rr, "X\hat{\alpha}", shorten <= 25pt, shorten >=25pt, Rightarrow] &             & {} 
\\
& {\P{C}(\dC(-,\hat{x}),X)} \arrow[rr, "{\Psi_{\hat{x},X}}"]                           &
& X\hat{x}                                    &   
\end{tikzcd}
\]
To do so, we evaluate both whiskerings at a horizontal transformation $\phi\colon\dC(-,\hat{y})\Rightarrow X$:
\begin{align*}
    (\Psi_{\hat{x},X}\circ\dC(-,\hat{\alpha})^*)_{\phi}
    &=
    \Psi_{\hat{x},X}(\phi\circ\dC(-,\hat{\alpha}))
    &
    \text{Definition of }\dC(-,\hat{\alpha})^*
    \\
    &=
    ((\phi\circ\dC(-,\hat{\alpha}))_{\hat{x}})_{1_{\hat{x}}}
    &
    \text{Definition of }\Psi_{\hat{x},X}
    \\
    &=
    \phi_{\hat{x}}(\dC(\hat{x},\hat{\alpha})_{1_{\hat{x}}})
    &
    \text{Whiskering}
    \\
    &=
    \phi_{\hat{x}}(\hat{\alpha}\circ e_{1_{\hat{x}}})
    &
    \text{Definition of }\dC(\hat{x},\hat{\alpha})
    \\
    &=
    \phi_{\hat{x}}(e_{1_{\hat{y}}}\circ\hat{\alpha})
    &
    \hat{\alpha}\circ e_{1_{\hat{x}}}=\hat{\alpha}=
    e_{1_{\hat{y}}}\circ\hat{\alpha}
    \\
    &=
    \phi_{\hat{x}}(\dC(\hat{\alpha},\hat{y})_{1_{\hat{y}}})
    &
    \text{Definition of }\dC(\hat{\alpha},\hat{y})
    \\
    &=
    (X\hat{\alpha})_{\phi_{\hat{y}}(1_{\hat{y}})}
    &
    2\text{-Naturality of }\phi
    \\
    &=
    (X\hat{\alpha})_{\Psi_{\hat{y},X}(\phi)}
    &
    \text{Definition of }\Psi_{\hat{y},X}
    \\
    &=
    (X\hat{\alpha}\circ\Psi_{\hat{y},X})_\phi.
    &
    \text{Whiskering}
\end{align*}
This shows that the components of the two whiskerings agree, as desired. 
\end{proof}

\begin{Proposition}
    The functors
    \[ \Psi_{\hat{x},X}\colon \P{C}(\dC(-,\hat{x}),X)\to X\hat{x}
    \]
    are $2$-natural in $X$ in $\P{C}$.
\end{Proposition}

\begin{proof}
    To prove $2$-naturality in $X$, we first have to show that, given a horizontal transformation $F\colon X\Rightarrow Y$ of lax double presheaves over $\dC$, the following diagram of functors commutes. 
    \[
    \begin{tikzcd}
        {\P{C}(\dC(-,\hat{x}),X)} \arrow[d, "F_*"'] \arrow[r, "{\Psi_{\hat{x},X}}"] & X\hat{x} \arrow[d, "F_{\hat{x}}"] \\
        {\P{C}(\dC(-,\hat{x}),Y)} \arrow[r, "{\Psi_{\hat{x},Y}}"']                  & Y\hat{x}                         
    \end{tikzcd}
    \]
    To do so, we evaluate both composites at a horizontal transformation $\phi\colon\dC(-,\hat{x})\Rightarrow X$:
    \begin{align*}
        \Psi_{\hat{x},Y}(F_*(\phi))
        &=
        \Psi_{\hat{x},Y}(F\circ\phi)
        &
        \text{Definition of }F_*
        \\
        &=
        (F\circ\phi)_{\hat{x}}(1_{\hat{x}})
        &
        \text{Definition of }\Psi_{\hat{x},Y}
        \\
        &=
        F_{\hat{x}}(\phi_{\hat{x}}(1_{\hat{x}}))
        &
        \text{Composition of horizontal transformations}
        \\
        &=
        F_{\hat{x}}(\Psi_{\hat{x},X}(\phi))
        &
        \text{Definition of }\Psi_{\hat{x},X}
    \end{align*}
    as well as at a globular modification 
    $\nu\colon\edgesquare{e_{\dC(-,\hat{x})}}{\phi}{\phi'}{e_X}$:
    \begin{align*}
        \Psi_{\hat{x},Y}(F_*(\nu))
        &=
        \Psi_{\hat{x},Y}(F\circ\nu)
        &
        \text{Definition of }F_*
        \\
        &=
        ((F\circ\nu)_{\hat{x}})_{1_{\hat{x}}}
        &
        \text{Definition of }\Psi_{\hat{x},X}
        \\
        &=
        F_{\hat{x}}((\nu_{\hat{x}})_{1_{\hat{x}}})
        &
        \text{Whiskering}
        \\
        &=
        F_{\hat{x}}(\Psi_{\hat{x},X}(\nu)).
        &
        \text{Definition of }\Psi_{\hat{x},X}
    \end{align*}
    This shows that the two composites agree on objects and morphisms, as desired.

    Next, we have to show that, given a globular modification 
    $A\colon\edgesquare{e_X}{F}{F'}{e_Y}$ of lax double presheaves over $\dC$, the following pasting diagram of functors and natural transformations commutes.
    \[
    \begin{tikzcd} 
& {\P{C}(\dC(-,\hat{x}),X)} \arrow[dd, swap,"{F_*}", bend right=49] \arrow[dd, "{F'_*}", bend left=49] \arrow[rr, "{\Psi_{\hat{x},X}}"] &      
& X\hat{x} \arrow[dd, swap,"F_{\hat{x}}", bend right =49] \arrow[dd, "F'_{\hat{x}}", bend left=49] &    
\\
{} \arrow[rr, "{A_*}", shorten <=50pt, shorten >=50pt,Rightarrow] 
&                     
&
{} \arrow[rr, "A_{\hat{x}}", shorten <= 25pt, shorten >=25pt, Rightarrow] &             & {} 
\\
& {\P{C}(\dC(-,\hat{x}),Y)} \arrow[rr, "{\Psi_{\hat{x},Y}}"']                           &
& Y\hat{x}                                    & 
\end{tikzcd}
    \]
    To do so, we evaluate both whiskerings at a horizontal transformation $\phi\colon\dC(-,\hat{x})\Rightarrow X$:
    \begin{align*}
        (\Psi_{\hat{x},Y}\circ A_*)_{\phi}
        &=
        \Psi_{\hat{x},Y}(A\circ\phi)
        &
        \text{Definition of }A_*
        \\
        &=
        ((A\circ\phi)_{\hat{x}})_{1_{\hat{x}}}
        &
        \text{Definition of }\Psi_{\hat{x},Y}
        \\
        &=
        (A_{\hat{x}})_{\phi_{\hat{x}}(1_{\hat{x}})}
        &
        \text{Whiskering}
        \\
        &=
        (A_{\hat{x}})_{\Psi_{\hat{x},X}(\phi)}
        &
        \text{Definition of }\Psi_{\hat{x},X}
        \\
        &=
        (A_{\hat{x}}\circ\Psi_{\hat{x},X})_\phi.
        &
        \text{Whiskering}
    \end{align*}
    This shows that the components of the two whiskerings agree, as desired. 
\end{proof}

\subsection{Proof of Yoneda lemma}

\label{subsec: Yoneda2}

To prove the Yoneda lemma, namely \cref{thm:Yoneda}, let us fix a double category $\dC$, a lax double presheaf $X\colon \dC^{\op}\to \dCat$ and an object $\hat{x}\in \dC$. We construct an inverse $\Phi_{\hat{x},X}\colon X\hat{x}\to \P{C}(\dC(-,\hat{x}),X)$ of the functor $\Psi_{\hat{x},X}$ from \cref{constr:yonedafunctor}. We start with its assignment on objects.

\begin{Construction} \label{constr:onobjects}
    Given an object $x_-$ in $X\hat{x}$, we construct a horizontal transformation
    \[ \Phi_{\hat{x},X}(x_-)\colon \dC(-,\hat{x})\Rightarrow X\colon \dC^{\op}\to\dCat \]
    such that
    \begin{itemize}[leftmargin=1cm]
        \item its component at an object $x$ in $\dC$ is given by the functor 
        \[ \Phi_{\hat{x},X}(x_-)_x\colon \dC(x,\hat{x})\to Xx \]
        sending an object $x\xrightarrow{g}\hat{x}$ in $\dC(x,\hat{x})$ to the object $Xg(x_-)$ in $Xx$, i.e., the image of the object $x_-$ under the functor $Xg\colon X\hat{x}\to Xx$, and a square $\eta\colon\edgesquare{e_x}{g}{g'}{e_{\hat{x}}}$ to the morphism 
        $(X\eta)_{x_-}\colon Xg(x_-)\to Xg'(x_-)$ in $Xx$, i.e., the component at $x_-$ of the natural transformation 
        $X\eta\colon Xg\Rightarrow Xg'\colon X\hat{x}\to Xx$,
        \item its component at a vertical morphism $u\colon x\bulletarrow x'$ in $\dC$ is given by the natural transformation 
        \[ 
            \Phi_{\hat{x},X}(x_-)_u\colon
            \dC(u,\hat{x})\Rightarrow
            Xu (\Phi_{\hat{x},X}(x_-)_x^{\op}\times\Phi_{\hat{x},X}(x_-)_{x'})\colon
            \dC(x,\hat{x})^{\op}\times\dC(x',\hat{x})\to \Set
        \]
        whose component at an object $(x\xrightarrow{g}\hat{x},x'\xrightarrow{g'}\hat{x})$ of $\dC(x,\hat{x})\times\dC(x',\hat{x})$
        is the map
        \[
        \dC(u,\hat{x})(g,g')\longrightarrow
        Xu(Xg(x_-),Xg'(x_-))
        \]
        sending a square
        $\eta\colon\edgesquare{u}{g}{g'}{e_{\hat{x}}}$ to the element 
        $(X\eta)_{x_-,x_-}(1_{x_-})$, i.e., the image of the identity $1_{x_-}$ under the component $(X\eta)_{x_-,x_-}\colon X\hat{x}(x_-,x_-)\to Xu(Xg(x_-),Xg'(x_-))$ of the natural transformation $X\eta\colon Xe_{\hat{x}}=e_{X\hat{x}}\Rightarrow Xu(Xg^{\op}\times Xg')$ at $x_-$.
    \end{itemize}
\end{Construction}

\begin{Lemma}
    The construction \[ \Phi_{\hat{x},X}(x_-)\colon \dC(-,\hat{x})\Rightarrow X \] is a horizontal transformation. 
\end{Lemma}

\begin{proof}
     We first show naturality of the components $\Phi_{\hat{x},X}(x_-)_x$ in $x$. For this, let $f\colon y\to x$ be a horizontal morphism in $\dC$. We want to show that the following diagram of functors commutes. 
     \[
    \begin{tikzcd}
        \dC(x,\hat{x})
        \arrow[r,"{\dC(f,\hat{x})}"]
        \arrow[d,swap,"{\Phi_{\hat{x},X}(x_-)_x}"]
        &
        \dC(y,\hat{x})
        \arrow[d,"{\Phi_{\hat{x},X}(x_-)_y}"]
        \\
        Xx
        \arrow[r,swap,"{Xf}"]
        &
        Xy
    \end{tikzcd}
    \]
    To do so, we evaluate both composites at an object $x\xrightarrow{g}\hat{x}$ of $\dC(x,\hat{x})$:
    \begin{align*}
        \Phi_{\hat{x},X}(x_-)_y\circ\dC(f,\hat{x})(g)
        &=
        \Phi_{\hat{x},X}(x_-)_y(g\circ f)
        &
        \text{Definition of }\dC(f,\hat{x})
        \\
        &
        =X(g\circ f)(x_-)
        &
        \text{Definition of }\Phi_{\hat{x},X}(x_-)
        \\
        &=
        Xf\circ Xg(x_-)
        &
        \text{Functorality of }X
        \\
        &=
        Xf\circ\Phi_{\hat{x},X}(x_-)_x(g)
        &
        \text{Definition of }\Phi_{\hat{x},X}(x_-)
    \end{align*}
    as well as at a morphism $\eta\colon\edgesquare{e_x}{g}{g'}{e_{\hat{x}}}$ of $\dC(x,\hat{x})$: 
    \begin{align*}
        \Phi_{\hat{x},X}(x_-)_y\circ\dC(f,\hat{x})(\eta)
        &=
        \Phi_{\hat{x},X}(x_-)_y(\eta\circ e_f)
        &
        \text{Definition of }\dC(f,\hat{x})
        \\
        &=
        X(\eta\circ e_f)_{x_-,x_-}(1_{x_-})
        &
        \text{Definition of }
        \Phi_{\hat{x},X}(x_-)
        \\
        &=
        Xe_f\circ X\eta_{x_-,x_-}(1_{x_-})
        &
        \text{Functorality of }X
        \\
        &=
        Xf\circ X\eta_{x_-,x_-}(1_{x_-})
        &\text{Normality of } X
        \\
        &=
        Xf\circ\Phi_{\hat{x},X}(x_-)_x(\eta)
        &\text{Definition of }\Phi_{\hat{x},X}(x_-)
    \end{align*}
     This shows that the two composites agree on objects and morphisms, as desired. 

     Next, we show naturality of the components $\Phi_{\hat{x},X}(x_-)_u$ in $u$. For this, let $\alpha\colon\edgesquare{v}{f}{f'}{u}\colon\nodesquare{y}{y'}{x}{x'}$ be a square in $\dC$. We want to show that the following diagram in $\Prof(\dC(x,\hat{x}),\dC(x',\hat{x}))$ commutes.
    \[
    \begin{tikzcd}
        \dC(u,\hat{x})
        \arrow[r,Rightarrow,"{\dC(\alpha,\hat{x})}"]
        \arrow[d,swap,Rightarrow,"{\Phi_{\hat{x},X}(x_-)_u}"]
        & [40pt]
        \dC(v,\hat{x})(\dC(f,\hat{x})^{\op}\times \dC(f',\hat{x}))
        \arrow[d,Rightarrow,"{\Phi_{\hat{x},X}(x_-)_v(\dC(f,\hat{x})^{\op}\times \dC(f',\hat{x}))}"]
        \\
        Xu(\Phi_{\hat{x},X}(x_-)_x^{\op}\times \Phi_{\hat{x},X}(x_-)_{x'})
        \arrow[r,swap,Rightarrow,"{X\alpha(\Phi_{\hat{x},X}(x_-)_x^{\op}\times \Phi_{\hat{x},X}(x_-)_{x'})}" {yshift=-3pt}]
        &
        Xv((Xf\circ \Phi_{\hat{x},X}(x_-)_x)^{\op}\times (Xf'\circ\Phi_{\hat{x},X}(x_-)_{x'}))
    \end{tikzcd}
    \] 
    where we recall that, by the first part, we have the relations $Xf\circ \Phi_{\hat{x},X}(x_-)_x=\Phi_{\hat{x},X}(x_-)_y\circ \dC(f,\hat{x})$ and $Xf'\circ \Phi_{\hat{x},X}(x_-)_{x'}=\Phi_{\hat{x},X}(x_-)_{y'}\circ \dC(f',\hat{x})$. When evaluated at an object $(x\xrightarrow{g} \hat{x},x'\xrightarrow{g'}\hat{x})$ of $\dC(x,\hat{x})\times\dC(x',\hat{x})$, this amounts to showing that the following diagram in $\Set$ commutes. 
    \[
    \begin{tikzcd}
        \dC(u,\hat{x})(g,g')
        \arrow[r,"{(\dC(\alpha,\hat{x}))_{g,g'}}"]
        \arrow[d,swap,"{(\Phi_{\hat{x},X}(x_-)_u)_{g,g'}}"]
        & [45pt]
        \dC(u,\hat{x})(gf,g'f')
        \arrow[d,"{(\Phi_{\hat{x},X}(x_-)_v)_{gf,g'f'}}"]
        \\
        Xv(Xg(x_-),Xg'(x_-))
        \arrow[r,swap,"{(X\alpha)_{Xg(x_-),Xg'(x_-)}}"]
        &
        Xu(X(gf)(x_-),X(g'f')(x_-))
    \end{tikzcd}
    \]
    To do so, we evaluate both composites at an element 
    $\eta\colon\edgesquare{u}{g}{g'}{e_{\hat{x}}}$ in $\dC(u,\hat{x})(g,g')$:
    \begin{align*}
        (\Phi_{\hat{x},X}(x_-)_v)_{gf,g'f'}( \dC(\alpha,\hat{x})(\eta))
        &=
        X(\dC(\alpha,\hat{x})(\eta))_{x_-,x_-}(1_{x_-})
        &
        \text{Definition of }\Phi_{\hat{x},X}(x_-)_v
        \\
        &=
        X(\eta\circ\alpha)_{x_-,x_-}(1_{x_-})
        &
        \text{Definition of }\dC(\alpha,\hat{x})
        \\
        &=
        (X\alpha)_{Xg(x_-),Xg'(x_-)}((X\eta)_{x_-,x_-}(1_{x_-}))
        &
        \text{Functorality of }X
        \\
        &=
        (X\alpha)_{Xg(x_-),Xg'(x_-)}(\Phi_{\hat{x},X}(x_-)_u(\eta)).
        &
        \text{Definition of }\Phi_{\hat{x},X}(x_-)_u
    \end{align*}
    This shows that the two composites agree, as desired. 

    Next, we show compatibility of $\Phi_{\hat{x},X}(x_-)_u$ with composition comparison squares. For this, let $u\colon x\bulletarrow x'$ and $u'\colon x'\bulletarrow x''$ be composable vertical morphisms in $\dC$. We want to show that 
    the following diagram in $\Prof(\dC(x,\hat{x}),\dC(x'',\hat{x}))$ commutes.
    \[
    \begin{tikzcd}
        {\dC(u',\hat{x})\bullet\dC(u,\hat{x})} \arrow[r, "{\mu_{u,u'}}", Rightarrow] \arrow[d, "{\Phi_{\hat{x},X}(x_-)_{u'}\bullet\Phi_{\hat{x},X}(x_-)_u}"', Rightarrow] 
        &[50pt]
        {\dC(u'\bullet u,\hat{x})}
        \arrow[d, "{\Phi_{\hat{x},X}(x_-)_{u'\bullet u}}", Rightarrow] 
        \\
       (Xu'\bullet Xu)
        (\Phi_{\hat{x},X}(x_-)_x^{\op}\times
        \Phi_{\hat{x},X}(x_-)_{x''})  \arrow[r, "{\mu_{u,u'}(\Phi_{\hat{x},X}(x_-)_x^{\op}\times
        \Phi_{\hat{x},X}(x_-)_{x''})}"' {yshift=-3pt}, Rightarrow]                                                             & X(u'\bullet u)
        (\Phi_{\hat{x},X}(x_-)_x^{\op}\times
        \Phi_{\hat{x},X}(x_-)_{x''})
    \end{tikzcd}
    \]
    When evaluated at an object 
    $(x\xrightarrow{g}\hat{x},x''\xrightarrow{g''}\hat{x})$ of 
    $\dC(x,\hat{x})\times\dC(x'',\hat{x})$, this amounts to showing that the following diagram in $\Set$ commutes.
    \[
    \begin{tikzcd}
        {(\dC(u',\hat{x})\bullet\dC(u,\hat{x}))(g,g'')} \arrow[r, "{(\mu_{u,u'})_{g,g''}}"] \arrow[d, "{(\Phi_{\hat{x},X}(x_-)_{u'}\bullet\Phi_{\hat{x},X}(x_-)_u)_{g,g''}}"'] 
        &[60pt] {\dC(u'\bullet u,\hat{x})(g,g'')}
         \arrow[d, "{(\Phi_{\hat{x},X}(x_-)_{u'\bullet u})_{g,g''}}"] \\
       {(Xu'\bullet Xu)(Xg(x_-),Xg''(x_-)))}  \arrow[r, "{(\mu_{u,u'})_{Xg(x_-),Xg''(x_-)}}"']                 &
        {X(u'\bullet u)(Xg(x_-),Xg''(x_-))}    
\end{tikzcd}
    \]
    By the universal properties of the coequalizers, it is enough to show that, given an object $x'\xrightarrow{g'}\hat{x}$ of $\dC(x',\hat{x})$, the following diagram in $\Set$ commutes.
    \[
    \begin{tikzcd}
        {\dC(u',\hat{x})(g',g'')\times \dC(u,\hat{x})(g,g')} \arrow[r, "{\bullet}"] \arrow[d, "{(\Phi_{\hat{x},X}(x_-)_{u'})_{g',g''}\times (\Phi_{\hat{x},X}(x_-)_u)_{g,g'}}"'] 
        &[25pt] {\dC(u'\bullet u,\hat{x})(g,g'')}\arrow[d,"{(\Phi_{\hat{x},X}(x_-)_{u'\bullet u})_{g,g''}}"]
         \\
        {Xu'(Xg'(x_-),Xg''(x_-))\times Xu(Xg(x_-),Xg'(x_-))} \arrow[r,"{(\mu_{u,u'})_{Xg(x_-),Xg''(x_-)}}"' {yshift=-3pt}]                 &
        {X(u'\bullet u)(Xg(x_-),Xg''(x_-))}    
\end{tikzcd}
    \]
    To do so, we evaluate both composites at an element
    $(\eta'\colon\edgesquare{u'}{g'}{g''}{e_{\hat{x}}},\eta\colon\edgesquare{u}{g}{g'}{e_{\hat{x}}})$ of the set $
    \dC(u',\hat{x})(g',g'')\times\dC(u,\hat{x})(g,g')$:
    \begin{align*}
        (\mu_{u,u'})&_{Xg(x_-),Xg''(x_-)}((\Phi_{\hat{x},X}(x_-)_{u'}\times
        \Phi_{\hat{x},X}(x_-)_u)((\eta',\eta)))
        &
        \\
        &=
        (\mu_{u,u'})_{Xg(x_-),Xg''(x_-)}((X\eta')_{x_-,x_-}(1_{x_-}),(X\eta)_{x_-,x_-}(1_{x_-}))
        &
        \text{Definition of }\Phi_{\hat{x},X}(x_-)
        \\
        &=
        X(\eta'\bullet\eta)_{x_-,x_-}(1_{x_-})
        &
        \text{Naturality of }\mu_{u,u'}
        \\
        &=
        \Phi_{\hat{x},X}(x_-)(\eta'\bullet\eta)
        &
        \text{Definition of }\Phi_{\hat{x},X}(x_-)_{u'\bullet u}
    \end{align*}
    This shows that both composites agree, as desired. 
    
    Finally, we show compatibility of $\Phi_{\hat{x},X}(x_-)_u$ with vertical identities. Given an object $x$ of $\dC$, we want to show that the natural transformation of functors $\dC(x,\hat{x})^{\op}\times \dC(x,\hat{x})\to \Set$ 
    \[ \Phi_{\hat{x},X}(x_-)_{e_x}\colon \dC(e_x,\hat{x})\Rightarrow Xe_x(\Phi_{\hat{x},X}(x_-)_x^{\op}\times \Phi_{\hat{x},X}(x_-)_x)\] 
    is the identity at $\Phi_{\hat{x},X}(x_-)_x$. When evaluated at an object $(x\xrightarrow{g}\hat{x},x\xrightarrow{g'}\hat{x})$ of $\dC(x,\hat{x})^{\op}\times \dC(x,\hat{x})$, this is given by the map of sets 
    \[ (\Phi_{\hat{x},X}(x_-)_{e_x})_{g,g'}\colon \dC(x,\hat{x})(g,g')\to Xx(Xg(x_-),Xg'(x_-))\]
    sending an element $\eta\colon \edgesquare{e_x}{g}{g'}{e_{\hat{x}}}$ in $\dC(x,\hat{x})(g,g')$ to $(X\eta)_{x_-,x_-}(1_{x_-})
    =\Phi_{\hat{x},X}(x_-)_x(\eta)$.
    Hence, by \cref{identity 2cell}, we see that $\Phi_{\hat{x},X}(x_-)_{e_x}$ is the identity at $\Phi_{\hat{x},X}(x_-)_x$.
\end{proof}

We now turn to the assignment of $\Psi_{\hat{x},X}$ on morphisms. 

\begin{Construction}\label{constr:onmorphisms}
    Given a morphism $u_-\colon x_-\to x'_-$ in $X\hat{x}$, we construct a globular modification in $\P{C}$
    \[
    \begin{tikzcd}
        \dC(-,\hat{x})
        \arrow[r,Rightarrow,"{\Phi_{\hat{x},X}(x_-)}"]
        \arrow[d,"\bullet" marking,equal]
        \arrow[rd,phantom,"{\Phi_{\hat{x},X}(u_-)}"]
        & [10pt]
        X
        \arrow[d,"\bullet" marking,equal]
        \\
        \dC(-,\hat{x})
        \arrow[r,swap,Rightarrow,"{\Phi_{\hat{x},X}(x'_-)}"]
        &
        X
    \end{tikzcd}
    \]
    such that its component at an object $x\in\dC$ is given by the natural transformation
    \[
    \Phi_{\hat{x},X}(u_-)_x\colon
    \Phi_{\hat{x},X}(x_-)_x\Rightarrow
    \Phi_{\hat{x},X}(x'_-)_x
    \colon
    \dC(x,\hat{x})\to Xx
    \]
    whose component at an object $x\xrightarrow{g}\hat{x}$ of $\dC(x,\hat{x})$ is given by the morphism
    \[
    Xg(u_-)\colon Xg(x_-)=\Phi_{\hat{x},X}(x_-)_x(g)\longrightarrow
    Xg(x'_-)=\Phi_{\hat{x},X}(x'_-)_x(g),
    \]
    i.e., the image of the morphism $u_-$ under the functor $Xg\colon X\hat{x}\to Xx$.
\end{Construction}

\begin{Lemma}
    The construction $\Phi_{\hat{x},X}(u_-)$ is a modification. 
\end{Lemma}

\begin{proof}
    We start by showing horizontal compatibility of $\Phi_{\hat{x},X}(u_-)$. Given a horizontal morphism 
    $f\colon x\to y$ in $\dC$, we have to show that the following pasting diagram of functors and natural transformations commutes.
     \[
    \begin{tikzcd} 
& {\dC(y,\hat{x})} \arrow[dd, swap,"{\Phi_{\hat{x},X}(x_-)_y}", bend right=52] \arrow[dd, "{\Phi_{\hat{x},X}(x'_-)_y}", bend left=52] \arrow[rr, "{\dC(f,\hat{x})}"] 
&[40pt]      
& {\dC(x,\hat{x})} \arrow[dd, swap,"{\Phi_{\hat{x},X}(x_-)_x}", bend right =52] \arrow[dd, "{\Phi_{\hat{x},X}(x'_-)_x}", bend left=52] &    
\\
{} \arrow[rr,yshift = -9pt, "{\Phi_{\hat{x},X}(u_-)_y}" {xshift = -20pt}, shorten <=30pt, shorten >=70pt,Rightarrow] 
&                     
&[40pt]
{} \arrow[rr, yshift = -9pt, "{\Phi_{\hat{x},X}(u_-)_x}", shorten <= 30pt, shorten >=30pt, Rightarrow] &             & {} 
\\
& {Xy} \arrow[rr, swap,"{Xf}"]                           &[40pt]
& {Xx}                                  & 
\end{tikzcd}
    \]
    To do so, we evaluate both whiskerings at an object $y\xrightarrow{g}\hat{x}$ of $\dC(y,\hat{x})$:
    \begin{align*}
        (Xf \circ \Phi_{\hat{x},X}(u_-)_y)_g
        &=
        Xf((\Phi_{\hat{x},X}(u_-)_y)_g)
        &
        \text{Whiskering}
        \\
        &=
        Xf(Xg(u_-))
        &
        \text{Definition of }\Phi_{\hat{x},X}(u_-)_y
        \\
        &=
        X(gf)(u_-)
        &
        \text{Functorality of }X
        \\
        &=
        (\Phi_{\hat{x},X}(u_-)_x)_{gf}
        &
        \text{Definition of }\Phi_{\hat{x},X}(u_-)_x
        \\
        &=
        (\Phi_{\hat{x},X}(u_-)_x)_{\dC(f,\hat{x})(g)}
        &
        \text{Definition of }\dC(f,\hat{x})
        \\
        &=
        (\Phi_{\hat{x},X}(u_-)_x \circ \dC(f,\hat{x}))_g.
        &
        \text{Whiskering}
    \end{align*}
    This shows that the components of the two whiskerings agree, as desired. 

    Next, we show vertical compatibility of $\Phi_{\hat{x},X}(u_-)$. Given a vertical morphism $u\colon x\bulletarrow x'$ in $\dC$, we have to show that the following diagram commutes in $[\dC(x,\hat{x})^{\op}\times\dC(x',\hat{x}),\Set]$:
    \[
        \begin{tikzcd}
            {\dC(u,\hat{x})} \arrow[r, "{\Phi_{\hat{x},X}(x_-)_u}", Rightarrow] \arrow[d, "{\Phi_{\hat{x},x}(x'_-)_u}"', Rightarrow]                                                           & [60pt] {Xu (\Phi_{\hat{x},X}(x_-)_x^{\op}\times\Phi_{\hat{x},X}(x_-)_{x'})} \arrow[d, "{Xu(\Phi_{\hat{x},X}(x_-)_x^{\op}\times\Phi_{\hat{x},X}(u_-)_{x'})}" {yshift=-3pt}, Rightarrow] \\
            {Xu (\Phi_{\hat{x},X}(x'_-)_x^{\op}\times\Phi_{\hat{x},X}(x'_-)_{x'})} \arrow[r, "{Xu (\Phi_{\hat{x},X}(u_-)_x^{\op}\times\Phi_{\hat{x},X}(x'_-)_{x'})}"' {yshift=-3pt}, Rightarrow] & {Xu (\Phi_{\hat{x},X}(x_-)_x^{\op}\times\Phi_{\hat{x},X}(x'_-)_{x'})}     
        \end{tikzcd}
        \]
        When evaluated at an object $(x\xrightarrow{g}\hat{x},x'\xrightarrow{g'}\hat{x})$ of 
        $\dC(x,\hat{x})\times\dC(x',\hat{x})$, this amounts to showing that the following diagram in $\Set$ commutes.
        \[
        \begin{tikzcd}
            {\dC(u,\hat{x})(g,g')} \arrow[d, "{(\Phi_{\hat{x},X}(x'_-)_u)_{g,g'}}"'] \arrow[r, "{(\Phi_{\hat{x},X}(x_-)_u)_{g,g'}}"] & [20pt]{Xu(Xg(x_-),Xg'(x_-))} \arrow[d, "Xg'(u_-)_*"] \\
            {Xu(Xg(x'_-),Xg'(x'_-))} \arrow[r, "Xg(u_-)^*"']                       & {Xu(Xg(x_-),Xg'(x'_-))}               
        \end{tikzcd}
        \]
        To do so, we evaluate both composites at an element
        $\eta\colon\edgesquare{u}{g}{g'}{e_{\hat{x}}}$ in $\dC(u,\hat{x})(g,g')$:
        \begin{align*}
            Xg'(u_-)_*((\Phi_{\hat{x},X}(x_-)_u)_{g,g'}(\eta))
            &=
            Xg'(u_-)_*((X\eta)_{x_-,x_-}(1_{x_-}))
            &
            \text{Definition of }\Phi_{\hat{x},X}(x_-)_u
            \\
            &=
            (X\eta)_{x_-,x'_-}((u_-)_*(1_{x_-}))
            &
            \text{Naturality of }X\eta
            \\
            &=
            (X\eta)_{x_-,x'_-}(u_-)
            &
            \text{Definition of }(u_-)_*
            \\
            &=
            (X\eta)_{x_-,x'_-}((u_-)^*(1_{x'_-}))
            &
            \text{Definition of }(u_-)^*
            \\
            &=
            Xg(u_-)^*((X\eta)_{x'_-,x'_-}(1_{x'_-}))
            &
            \text{Naturality of }X\eta
            \\
            &=
            Xg(u_-)^*((\Phi_{\hat{x},X}(x'_-)_u)_{g,g'}(\eta))
            &
            \text{Definition of }\Phi_{\hat{x},X}(x'_-)_u
        \end{align*}
        This shows that the two composites agree, as desired. 
\end{proof}

Putting everything together, we get the following. 

\begin{Construction} \label{constr:yonedainverse}
    Given a double category $\dC$, a lax double presheaf $X\colon \dC^{\op}\to \dCat$, and an object $\hat{x}\in \dC$, we construct a functor 
    \[ \Phi_{\hat{x},X}\colon X\hat{x}\to \P{C}(\dC(-,\hat{x}),X)
    \]
    which sends 
    \begin{itemize}[leftmargin=1cm]
        \item an object $x_-$ in $X\hat{x}$ to the horizontal transformation $\Phi_{\hat{x},X}(x_-)\colon \dC(-,\hat{x})\Rightarrow X$ as in \cref{constr:onobjects}, 
        \item a morphism $u_-\colon x_-\to x'_-$ in $X\hat{x}$ to the modification
    $\Phi_{\hat{x},X}(u_-)\colon\edgesquare{e_{\dC(-,\hat{x})}}{\Phi_{\hat{x},X}(x_-)}{\Phi_{\hat{x},X}(x'_-)}{e_X}$ as in
    \cref{constr:onmorphisms}.
    \end{itemize}
\end{Construction}

We are now ready to prove the Yoneda lemma.

\begin{proof}[Proof of \cref{thm:Yoneda}]
We show that the functors 
\[ \Psi_{\hat{x},X}\colon \P{C}(\dC(-,\hat{x}),X)\to X\hat{x}\quad \text{and} \quad \Phi_{\hat{x},X}\colon X\hat{x}\to \P{C}(\dC(-,\hat{x}),X) \]
   from \cref{constr:yonedafunctor,constr:yonedainverse} are inverses to each other. 
   
    We first consider the composite 
\[  X\hat{x} \xrightarrow{\Phi_{\hat{x},X}}\P{C}(\dC(-,\hat{x}),X)\xrightarrow{\Psi_{\hat{x},X}} X\hat{x}. \]
It sends an object $x_-$ in $X\hat{x}$ to the object 
\begin{align*}
    \Psi_{\hat{x},X}(\Phi_{\hat{x},X}(x_-))
    &=
    \Phi_{\hat{x},X}(x_-)_{\hat{x}}(1_{\hat{x}})
    &
    \text{Definition of }\Psi_{\hat{x},X}
    \\
    &=
    X1_{\hat{x}}(x_-)
    &
    \text{Definition of }\Phi_{\hat{x},X}
    \\
    &=
    1_{X\hat{x}}(x_-)=x_-
    &
    \text{Functorality of }X
\end{align*}
and a morphism $u_-\colon x_-\to x'_-$ in $X\hat{x}$ to the morphism
\begin{align*}
    \Psi_{\hat{x},X}(\Phi_{\hat{x},X}(u_-))
    &=
    (\Phi_{\hat{x},X}(u_-)_{\hat{x}})_{1_{\hat{x}}}
    &
    \text{Definition of }\Psi_{\hat{x},X}
    \\
    &=
    X1_{\hat{x}}(u_-)
    &
    \text{Definition of }\Phi_{\hat{x},X}(u_-)
    \\
    &=
    1_{X\hat{x}}(u_-)=u_-.
    &
    \text{Functorality of } X
\end{align*}
This shows that the composite $\Psi_{\hat{x},X}\circ \Phi_{\hat{x},X}$ is the identity.

We then consider the composite 
\[ \P{C}(\dC(-,\hat{x}),X)\xrightarrow{\Psi_{\hat{x},X}} X\hat{x} \xrightarrow{\Phi_{\hat{x},X}}\P{C}(\dC(-,\hat{x}),X) \]
It sends a horizontal transformation $\varphi\colon \dC(-,\hat{x})\Rightarrow X$ to the horizontal transformation 
\[ \Phi_{\hat{x},X}(\Psi_{\hat{x},X}(\varphi))=\Phi_{\hat{x},X}(\varphi_{\hat{x}}(1_{\hat{x}}))\colon \dC(-,\hat{x})\Rightarrow X\]
whose 
\begin{itemize}[leftmargin=1cm]
    \item component at an object $x$ in $\dC$ is given by the functor 
    \[ \Phi_{\hat{x},X}(\varphi_{\hat{x}}(1_{\hat{x}}))_x\colon \dC(x,\hat{x})\to Xx \]
    defined as follows. It sends an object $x\xrightarrow{g}\hat{x}$ in $\dC(x,\hat{x})$ to the object 
    \begin{align*}
        \Phi_{\hat{x},X}(\phi_{\hat{x}}(1_{\hat{x}}))_x(g)
        &=
        Xg(\varphi_{\hat{x}}(1_{\hat{x}}))
        &
        \text{Definition of }\Phi_{\hat{x},X}
        \\
        &=
        \varphi_{x}(\dC(g,\hat{x})(1_{\hat{x}}))
        &
        \text{Naturality of }\phi
        \\
        &=
        \varphi_x(1_{\hat{x}}\circ g)
        =
        \varphi_{x}(g)
        &
        \text{Definition of }\dC(g,\hat{x})
    \end{align*}
    and a morphism $\eta\colon\edgesquare{e_x}{g}{g'}{e_{\hat{x}}}$ in $\dC(x,\hat{x})$ to the morphism
    \begin{align*}
        \Phi_{\hat{x},X}(\phi_{\hat{x}}(1_{\hat{x}}))_x(\eta)
        &=
        (X\eta)_{\phi_{\hat{x}}(1_{\hat{x}})}
        &
        \text{Definition of }
        \Phi_{\hat{x},X}
        \\
        &=
        \phi_x(\dC(\eta,\hat{x})_{1_{\hat{x}}})
        &
        \text{Naturality of }\phi
        \\
        &=
        \phi_x(\eta).
        &
        \text{Definition of }\dC(\eta,\hat{x})
    \end{align*}
    This shows that 
    $\Phi_{\hat{x},X}(\Psi_{\hat{x},X}(\phi))_x=\Phi_{\hat{x},X}(\phi_{\hat{x}}(1_{\hat{x}}))=\phi_x$.
    \item component at a vertical morphism 
    $u\colon x\bulletarrow x'$ in $\dC$ is given by the natural transformation
    \[
    \Phi_{\hat{x},X}(\phi_{\hat{x}}(1_{\hat{x}}))_u\colon
    \dC(u,\hat{x})
    \Rightarrow Xu\left(\Phi_{\hat{x},X}(\varphi_{\hat{x}}(1_{\hat{x}}))_x^{\op}\times\Phi_{\hat{x},X}(\varphi_{\hat{x}}(1_{\hat{x}}))_{x'}\right)=
    Xu(\phi_x^{\op}\times\phi'_x)
    \]
    defined as follows. Its component at an object   $(x\xrightarrow{g}\hat{x},x'\xrightarrow{g'}\hat{x})$ of $\dC(x,\hat{x})\times\dC(x',\hat{x})$ is the map 
    \[ (\Phi_{\hat{x},X}(\phi_{\hat{x}}(1_{\hat{x}}))_u)_{g,g'}\colon \dC(u,\hat{x})(g,g')\to Xu(\varphi_x(g),\varphi'_x(g')) \]
    which sends an element 
    $\eta\colon\edgesquare{u}{g}{g'}{e_{\hat{x}}}$ of $\dC(u,\hat{x})(g,g')$ to the element
    \begin{align*}
        (\Phi_{\hat{x},X}(\phi_{\hat{x}}(1_{\hat{x}}))_u)_{g,g'}(\eta) &= (X\eta)_{\phi_{\hat{x}}(1_{\hat{x}}),\phi_{\hat{x}}(1_{\hat{x}})}(1_{\phi_{\hat{x}}(1_{\hat{x}})}) & \text{Definition of } \Phi_{\hat{x},X} \\
        &=(X\eta)_{\phi_{\hat{x}}(1_{\hat{x}}),\phi_{\hat{x}}(1_{\hat{x}})}(e_{1_{\hat{x}}})
        &
        \text{Functorality of }\phi_{\hat{x}}
        \\
        &=
        \phi_u(\dC(\eta,\hat{x})_{1_{\hat{x}}, 1_{\hat{x}}}(e_{1_{\hat{x}}}))
        &
        \text{Naturality of }\phi
        \\
        &=
        \phi_u(e_{1_{\hat{x}}}\circ\eta)=
        \phi_u(\eta).
        &
        \text{Definition of }\dC(\eta,\hat{x})
        \\
    \end{align*}
    This shows that $\Phi_{\hat{x},X}(\Psi_{\hat{x},X}(\phi))_u=\Phi_{\hat{x},X}(\phi_{\hat{x}}(1_{\hat{x}}))_u=\phi_u$.
\end{itemize}
Putting these two together, we deduce that $\Phi_{\hat{x},X}(\Psi_{\hat{x},X}(\phi))=\phi$.

The composite $\Phi_{\hat{x},X}\circ \Psi_{\hat{x},X}$ sends a globular modification 
$\nu\colon\edgesquare{e_{\dC(-,\hat{x})}}{\phi}{\phi'}{e_X}$ to the globular modification
\[
\begin{tikzcd}
    \dC(-,\hat{x})
    \arrow[d,"\bullet" marking,equal]
    \arrow[r,Rightarrow,"\Phi_{\hat{x},X}(\phi_{\hat{x}}(1_{\hat{x}}))=\phi"]
    \arrow[rd,phantom,
    "\Phi_{\hat{x},X}\circ\Psi_{\hat{x},X}(\nu)=
    \Phi_{\hat{x},X}((\nu_{\hat{x}})_{1_{\hat{x}}})"]
    &[140pt]
    X
    \arrow[d,"\bullet" marking,equal]
    \\
    \dC(-,\hat{x})
    \arrow[r,swap,Rightarrow,"\Phi_{\hat{x},X}(\phi'_{\hat{x}}(1_{\hat{x}}))=\phi'"]
    &[140pt]
    X
\end{tikzcd}
\]
defined as follows. Its component at an object $x$ in $\dC$ is given by the natural transformation
\[
\Phi_{\hat{x},X}((\nu_{\hat{x}})_{1_{\hat{x}}})_x
\colon
\phi_x\Rightarrow\phi'_x\colon\dC(x,\hat{x})\to Xx.
\]
whose component at an object 
$x\xrightarrow{g}\hat{x}$ of $\dC(x,\hat{x})$ is
\begin{align*}
    (\Phi_{\hat{x},X}((\nu_{\hat{x}})_{1_{\hat{x}}})_x)_g
    &=
    Xg((\nu_{\hat{x}})_{1_{\hat{x}}})
    &
    \text{Definition of }\Phi_{\hat{x},X}
    \\
    &=
    (\nu_x)_{\dC(g,\hat{x})(1_{\hat{x}})}
    &
    \text{Horizontal compatibility of }\nu
    \\
    &=
    (\nu_x)_{1_{\hat{x}}\circ g}=(\nu_x)_g.
    &
    \text{Definition of }\dC(g,\hat{x})
\end{align*}
This shows that $\Phi_{\hat{x},X}(\Phi_{\hat{x},X}(\nu))=\nu$. Hence this proves that the composite $\Phi_{\hat{x},X}\circ \Psi_{\hat{x},X}$ is the identity, and concludes the proof.
\end{proof}

As a consequence of \cref{lem: PC for 2-categories,thm:Yoneda}, when taking $\dC=\dH\CC$ with $\CC$ a $2$-category, we retrieve the usual Yoneda lemma for $2$-categories; see \cite[Lemma 8.3.16]{2Dim_Categories}.

\begin{Corollary} \label{Yoneda for 2-categories}
    Given a $2$-category $\CC$, a $2$-presheaf $X\colon \CC^{\op}\to \Cat$, and an object~$\hat{x}$ in $\CC$, the functor
    \[ \Psi_{\hat{x},X}\colon [\CC^{\op},\Cat](\CC(-,\hat{x}),X)\to X\hat{x}
    \]
    is an isomorphism, which is $2$-natural in $\hat{x}$ in $\CC$ and in $X$ in $[\CC^{\op},\Cat]$.
\end{Corollary}

\subsection{Yoneda embedding}

\label{subsec: Yoneda3}

Using the Yoneda lemma, we can construct a Yoneda embedding from the underlying horizontal $2$-category of a double category $\dC$ to its $2$-category of lax double presheaves.

\begin{Construction}
    Given a double category $\dC$, there is a Yoneda $2$-functor
    \[ \mathcal{Y}_\dC\colon \HH\dC\to \P C \]
    sending 
    \begin{itemize}[leftmargin=1cm]
        \item an object $\hat{x}$ of $\dC$ to the representable lax double presheaf $\dC(-,\hat{x})\colon \dC^{\op}\to \dCat$, 
        \item a horizontal morphism 
        $\hat{f}\colon \hat{x}\to \hat{y}$ of $\dC$ to the representable horizontal transformation \[ \dC(-,\hat{f})\colon \dC(-,\hat{x})\Rightarrow\dC(-,\hat{y}), \]
        \item a globular square $\hat{\alpha}\colon\edgesquare{e_{\hat{x}}}{\hat{f}}{\hat{f}'}{e_{\hat{y}}}$ to the representable globular modification \[ \dC(-,\hat{\alpha})\colon\edgesquare{e_{\dC(-,\hat{x})}}{\dC(-,\hat{f})}{\dC(-,\hat{f'})}{e_{\dC(-,\hat{y})}}. \]
    \end{itemize}
    It is straightforward to check that this construction is $2$-functorial.
\end{Construction}

\begin{Theorem}
    Given a double category $\dC$, the Yoneda $2$-functor
    \[ \mathcal{Y}_\dC\colon \HH\dC\to \P C \]
    is an embedding of $2$-categories.
\end{Theorem}

\begin{proof}
     Let $\hat{x}$ and $\hat{y}$ be objects in $\dC$. Then by \cref{thm:Yoneda}, we have a canonical isomorphism of categories 
    \[ \dC(\hat{x},\hat{y})\cong \P C(\dC(-,\hat{x}),\dC(-,\hat{y})). \]
    This shows that $\mathcal{Y}_\dC\colon \HH\dC\to \P C$ is $2$-fully-faithful, as desired.
\end{proof}

\section{Discrete double fibrations}

We now want lax double presheaves from a fibrational point of view. For this, we recall in this section the notion of discrete double fibrations. In \cref{subsec: doublediscfib}, we recollect the definition and main properties of discrete double fibrations. In particular, we construct a $2$-category of discrete double fibrations over a double category $\dC$ and show that it is $2$-functorial in $\dC$. In \cref{subsec: fibers}, we study the fibers of a discrete double fibration at both an object and a vertical morphism, which will be useful in the next section to construct a lax double presheaf from a discrete double fibration. Finally, in \cref{subsec: repdiscfib}, we present a first class of examples of discrete double fibrations, which are given by canonical projections from slice double categories. As these will correspond to representable lax double presheaves, we refer to them as \emph{representable discrete double fibrations}.

\subsection{The \texorpdfstring{$2$}{2}-category of discrete double fibrations}
\label{subsec: doublediscfib}

We start by introducing the $2$-category of discrete double fibrations. 

\begin{Definition}
\label{def: discrete double fibration}
\index{discrete double fibration}
    A double functor $P\colon\dE\to\dC$ is a \emph{discrete double fibration} if the following commutative square in $\Cat$ is a pullback
    \[
    \begin{tikzcd}
        \Ver_1\dE \arrow[r, "t"] \arrow[d,swap, "\Ver_1P"] \arrow[rd,phantom,near start, "\lrcorner"] & \Ver_0\dE\arrow[d, "\Ver_0P"] \\
        \Ver_1\dC \arrow[r, "t"']                & \Ver_0\dC              
    \end{tikzcd}
    \]
\end{Definition}

\begin{Notation}
     We denote by 
    $\dFib(\dC)$ the $2$-full $2$-subcategory of the slice $2$-category $\DblCat_v/\dC$  spanned by the discrete double fibrations over $\dC$. In other words,
    \begin{itemize}
        \item an object in $\dFib(\dC)$ is a discrete double fibration $P\colon \dE\to \dC$ over $\dC$, 
        \item given discrete double fibrations $P\colon \dE\to \dC$ and $Q\colon \dF\to \dC$, a morphism from $P$ to $Q$ in $\dFib(\dC)$ is a double functor $F\colon \dE\to \dF$ such that $Q\circ F=P$,
        \item given discrete double fibrations $P\colon \dE\to \dC$, $Q\colon \dF\to \dC$ and double functors $F\colon \dE\to \dF$, $F'\colon \dE\to \dF$ with $Q\circ F=P=Q\circ F'$, a $2$-morphism from $F$ to $F'$ in $\dFib(\dC)$ is a vertical transformation $A\colon F\Bulletarrow F'$ such that $Q \circ A= P$, i.e., $Q(A_x)=e_{Px}$ for all objects $x$ in $\dE$ and $Q(A_f)=e_{Pf}$ for all horizontal morphisms $f$ in $\dE$.
    \end{itemize}
\end{Notation}

\begin{Remark} \label{remark: dbldiscfib lifts}
    Unpacking the pullback condition, a double functor $P\colon \dE\to \dC$ is a discrete double fibration if and only if the following conditions hold:
    \begin{enumerate}[leftmargin=1cm]
        \item for every object $y_-$ in $\dE$ and every horizontal morphism 
        $f\colon x\to Py_-$ in $\dC$, there is a unique horizontal morphism 
        $f_-\colon x_-\to y_-$ in $\dE$ such that
        $Pf_-=f$. We write $f^*y_-\coloneqq x_-$ and 
        $P^*f\coloneqq f_-$.
        \item for every vertical morphism 
        $v_-\colon x_-\bulletarrow x'_-$ in $\dE$ and every square 
        $\alpha\colon\edgesquare{u}{f}{f'}{Pv_-}$ in $\dC$, there is a unique square
        $\alpha_-\colon\edgesquare{u_-}{f_-}{f'_-}{v_-}$ in $\dE$ such that 
        $P\alpha_-=\alpha$. We write $\alpha^* v_-\coloneqq u_-$ and
        $P^*\alpha\coloneqq \alpha_-$.
    \end{enumerate}
\end{Remark}

\begin{Remark}
    Note that, by unicity of the lifts, if $f\colon x\to y$ and $g\colon y\to z$ are two composable horizontal morphisms in $\dC$ and $z_-$ is an object of $\dE$ such that $Pz_-=z$, then \[ (g\circ f)^* z_-=f^*g^* z_-. \]
    Similarly, if $\alpha\colon\edgesquare{u}{f}{f'}{v}$ and $\beta\colon \edgesquare{v}{g}{g'}{w}$ are two horizontally composable squares in $\dC$ and $w_-$ is a vertical morphism in $\dE$ such that $Pw_-=w$, then \[ (\beta\circ\alpha)^*w_-=\alpha^*\beta^*w_-. \] 
\end{Remark}

The construction $\dFib(\dC)$ extends to a pseudo-functor $\DblCat_h^{\coop}\to 2\Cat$. To define its action on morphisms, we use
the following pullback stability of discrete double fibrations. This result follows from \cite[Proposition 4.11]{internal_GC} by taking $\mathcal{V}=\Cat$ and the fact that pullbacks in $2\Cat$ are $2$-categorical pullbacks.

\begin{Lemma}
\label{lem: pb preserve fibrations}
    Discrete double fibrations are stable under pullback. In particular, by universality of pullbacks, a double functor 
    $G\colon\dC\to\dD$ induces a $2$-functor $G^*\colon\dFib(\dD)\to\dFib(\dC)$ by taking pullbacks along $G$.
\end{Lemma}

We now describe its action on $2$-morphisms.

\begin{Construction} \label{constr:Fib(B)}
    Given a horizontal transformation $B\colon G\Rightarrow G'\colon\dC\to\dD$, we define a $2$-natural transformation
        \[
        B^*\colon G'^*\Rightarrow G^*\colon\dFib(\dD)\to\dFib(\dC),
        \]
        whose component at a discrete double fibration 
        $P\colon\dE\to\dD$ is given by the morphism of discrete double fibrations over $\dC$
        \[
        \begin{tikzcd}
        G'^*\dE \arrow[rd, "G'^*P"'] \arrow[rr, "B^*_P"] &     & G^*\dE \arrow[ld, "G^*P"] \\
        & \dC & 
        \end{tikzcd}
    \]
    which sends
    \begin{itemize}[leftmargin=1cm]
        \item an object 
        $(x,x_-)$ in $G'^*\dE$, i.e., a pair of objects $x$ in $\dC$ and $x_-$ in $\dE$ such that $Px_-=G'x$, to the object 
        $(x,(B_x)^*x_-)$ in $G^*\dE$, where $(B_x)^*x_-$ denotes the source of the unique lift of the horizontal morphism $B_x\colon Gx\to G'x$ in $\dD$ with target $x_-$,
        \item a horizontal morphism 
        $(f,f_-)\colon (x,x_-)\to (y,y_-)$ in $G'^*\dE$, i.e., a pair of horizontal morphisms $f\colon x\to y$ in $\dC$ and $f_-\colon x_-\to y_-$ in $\dE$ such that $Pf_-=G'f$, to the horizontal morphism in~$G^*\dE$
        \[ (f,P^*Gf)\colon (x,(B_x)^*x_-)\to (y,(B_y)^*y_-), \]
        where $P^*Gf\colon (B_x)^*x_-\to (B_y)^*y_-$ denotes the unique lift of $Gf$,
        \item a vertical morphism 
        $(u,u_-)\colon (x,x_-)\bulletarrow(x',x'_-)$ in $G'^*\dE$, i.e., a pair of vertical morphisms $u\colon x\bulletarrow x'$ in $\dC$ and $u_-\colon x_-\bulletarrow x'_-$ in $\dE$ such that $Pu_-=G'u$, to the vertical morphism in~$G^*\dE$
        \[ (u,(B_u)^* u_-)\colon (x,(B_x)^*x_-)
        \bulletarrow(x',(B_{x'})^* x'_-), \]
        where $(B_u)^* u_-$ denotes the source of the unique lift of the square $B_u\colon \edgesquare{Gu}{B_x}{B_{x'}}{G'u}$ in $\dD$ with target $u_-$,
        \item a square 
        $(\alpha,\alpha_-)\colon\edgesquare{(u,u_-)}{(f,f_-)}{(f',f'_-)}{(v,v_-)}\colon\nodesquare{(x,x_-)}{(x',x'_-)}{(y,y_-)}{(y',y'_-)}$ in $G'^*\dE$, i.e., a pair of squares $\alpha\colon \edgesquare{u}{f}{f'}{v}$ in $\dC$ and $\alpha_-\colon \edgesquare{u_-}{f_-}{f'_-}{v_-}$ in $\dE$ such that $P\alpha_-=G'\alpha$, to the square in $G^*\dE$
        \[
        \begin{tikzcd}
            {(x,(B_x)^*x_-)} \arrow[d,"\bullet" marking, "{(u,(B_u)^*u_-)}"'] \arrow[r, "{(f,P^*f)}"] \arrow[rd, phantom,"{(\alpha,P^*G\alpha)}"] 
            &[30pt]
            {(y,(B_y)^*y_-)} \arrow[d,"\bullet" marking, "{(v,(B_v)^*v_-)}"] \\
            {(x',(B_{x'})^* x'_-)} \arrow[r, "{(f',P^*f')}"'] &[30pt]
            {(y',(B_{y'})^*y'_-)}  
            \end{tikzcd} \]
            where $P^*G\alpha\colon \edgesquare{(B_u)^*u_-}{P^*f}{P^*f'}{(B_v)^*v_-}$ denotes the unique lift of $G\alpha$.
            \end{itemize}
\end{Construction}

\begin{Lemma}
    The construction $B^*\colon G'^*\Rightarrow G^*\colon\dFib(\dD)\to\dFib(\dC)$ is a $2$-natural transformation.
\end{Lemma}

\begin{proof}
    First note that the components $B^*_P\colon G'^*\dE\to G^*\dE$ are well-defined double functors over $\dC$. The fact that it preserves compositions and identities is a consequence of the unicity of lifts. 

    Similarly, the $2$-naturality of $B^*_P$ in $P$ follows from the unicity of lifts. 
\end{proof}

Putting everything together, we get the following.

\begin{Construction}
\label{def: dFib}
    We define a pseudo-functor
    \[
    \dFib\colon\DblCat_h^{\coop}\to 2\Cat,
    \]
    which sends
    \begin{itemize}[leftmargin=1cm]
        \item a double category $\dC$ to the $2$-category $\dFib(\dC)$ of discrete double fibrations, double functors over $\dC$, and vertical transformations over $\dC$,
        \item a double functor $G\colon\dC\to\dD$ to the $2$-functor
        \[
        \dFib(G)\coloneqq G^*\colon\dFib(\dD)\to\dFib(\dC)
        \]
        induced by taking pullback along $G$,
        \item a horizontal transformation $B\colon G\Rightarrow G'\colon\dC\to\dD$ to the $2$-natural transformation
        \[
        \dFib(B)\colon\dFib(G')=G'^*\Rightarrow\dFib(G)=G^*\colon\dFib(\dD)\to\dFib(\dC)
        \]
        from \cref{constr:Fib(B)}.
    \end{itemize}
    It is straightforward to check that this construction is pseudo-functorial.
\end{Construction}

Finally, we state the following result, which will be useful later on. A proof can be found in \cite[Proposition 4.12]{internal_GC} by taking $\mathcal{V}=\Cat$.

\begin{Lemma}
\label{lem: iso of ddcis fib can be checked on Ver}
    A morphism $F\colon\dE\to\dF$ between discrete double fibrations $P\colon\dE\to\dC$ and $Q\colon\dF\to\dC$ is an isomorphism of double categories 
    if and only if the induced functor $\Ver_0 F\colon\Ver_0\dE\to\Ver_0\dF$ is an isomorphism of categories.
\end{Lemma}

\subsection{Fibers of discrete double fibrations} 
\label{subsec: fibers}

We now want to study the fibers of a discrete double fibration at both an object and a vertical morphism, which will be useful to construct a lax double presheaf out of a discrete double fibration. Let us first define these. 

\begin{Construction}
\label{constr: definition of P_u}
Given a discrete double fibration $P\colon \dE\to \dC$, an object $x$ in $\dC$, and a vertical morphisms $u\colon x\bulletarrow x'$ in $\dC$, we define the following pullbacks in $\DblCat$,
\[
    \begin{tikzcd}
        P^{-1}x \arrow[r] \arrow[d] \arrow[rd,phantom,near start, "\lrcorner"] & \dE \arrow[d, "P"] \\
        {[0]} \arrow[r, "x"']                                & \dC               
    \end{tikzcd}\quad\quad
    \begin{tikzcd}
        P^{-1}u \arrow[r] \arrow[d] \arrow[rd, phantom,near start,"\lrcorner"] & {\dHom{\dV[1],\dE}} \arrow[d, "{\dHom{\dV[1],P}}"] \\
        {[0]} \arrow[r, "u"']                                & {\dHom{\dV[1],\dC}}                          
    \end{tikzcd}
    \]
    where $\dV[1]$ denotes the double category free on a vertical morphism.
Note that there is a canonical induced map $P_u\colon P^{-1}u\to P^{-1}x\times P^{-1}x'$ obtained using the universal property of pullback as follows: 
    \[
    \begin{tikzcd}
    P^{-1}u \arrow[rr] \arrow[dd] 
    \arrow[rd, "\exists! P_u", dashed,swap] &                                                       & {\dHom{\dV[1],\dE}} \arrow[dd, "{\dHom{\dV[1],P}}" {yshift=-10pt}] \arrow[rd, "{(s,t)}"] &                                      \\
    & P^{-1}x\times P^{-1}x' \arrow[rr] \arrow[dd] \arrow[rd,phantom,near start, "\lrcorner"] & & \dE\times\dE \arrow[dd, "P\times P"] \\
    {[0]} \arrow[rr, "u"' {xshift=20pt}] \arrow[rd,equal]  && {\dHom{\dV[1],\dC}} \arrow[rd, "{(s,t)}"] & \\
    & {[0]} \arrow[rr, "{(x,x')}"']  && \dC\times\dC 
    \end{tikzcd}
    \]
\end{Construction}

We start by showing that the fiber $P^{-1}x$ at an object is a category, seen as an object in~$\DblCat$ through the vertical embedding $\dV\colon \Cat\to \DblCat$. Explicitly, this means that its horizontal morphisms and squares are all trivial. 

\begin{Proposition} \label{prop: Px category}
    Given a discrete double fibration $P\colon \dE\to \dC$ and an object $x$ in $\dC$, the double category $P^{-1}x$ is a vertical double category. We also denote by $P^{-1}x$ the corresponding category. 
\end{Proposition}

\begin{proof}
    Consider the following diagram in $\Cat$
    \[ \begin{tikzcd}
        \Ver_1 P^{-1}x \arrow[r] \arrow[d] \arrow[rd,phantom,near start, "\lrcorner"] & \Ver_1\dE \arrow[rd,phantom,near start, "\lrcorner"] \arrow[d, "\Ver_1 P"'] \arrow[r,"t"] & \Ver_0\dE \arrow[d, "\Ver_0 P"] \\
        {[0]} \arrow[r, "x"']                                & \Ver_1\dC \arrow[r,"t"'] & \Ver_0\dC             
    \end{tikzcd}
    = \begin{tikzcd}
        \Ver_1 P^{-1}x \arrow[d] \arrow[r,"t"] & \Ver_0 P^{-1}x \arrow[r] \arrow[d] \arrow[rd,phantom,near start, "\lrcorner"] & \Ver_0\dE \arrow[d, "\Ver_0 P"] \\
        {[0]} \arrow[r,equal] & {[0]} \arrow[r, "x"']                                & \Ver_0\dC               
    \end{tikzcd} \]
    Here, the leftmost and rightmost squares are pullbacks by definition of $P^{-1}x$ and the fact that the functor $\Ver_i$ preserves pullbacks for $i=0,1$, and the middle left square is a pullback by definition of $P$ being a discrete double fibration. Hence, by cancellation of pullbacks, the middle right square is also a pullback and so $\Ver_1 P^{-1}x\cong \Ver_0 P^{-1}x$, showing that $P^{-1}x$ is a vertical double category. 
\end{proof}

Unpacking definition, we have the following explicit description of the fiber $P^{-1}x$. 

\begin{Remark}
    Given an object $x$ in $\dC$, the category $P^{-1}x$ is the category whose 
    \begin{itemize}[leftmargin=1cm]
        \item objects are objects $x_-$ in $\dE$ such that $Px_-=x$, 
        \item morphisms are vertical morphisms $s_-\colon x_-\bulletarrow \hat{x}_-$ in $\dE$ such that $Ps_-=e_x$. 
    \end{itemize}
\end{Remark}

Moreover, every horizontal morphism in $\dC$ acts on fibers by taking unique lifts as follows. 

\begin{Construction}\label{def: f on fibers}
    Given a horizontal morphism $f\colon x\to y$ in $\dC$, there is an induced functor
    \[ f^*\colon P^{-1}y\to P^{-1}x \]
    given by sending
    \begin{itemize}[leftmargin=1cm]
        \item an object $y_-$ in $P^{-1}y$ to the object $f^*y_-$ in $P^{-1}x$, i.e., the source of the unique lift of the horizontal morphism $f$ with target $y_-$,
        \item a morphism $s_-\colon y_-\bulletarrow \hat{y}_-$ in $P^{-1}y$ to the morphism $e_f^*s_-\colon f^*y_-\bulletarrow f^*\hat{y}_-$ in $P^{-1}x$, i.e., the source of the unique lift of the identity square $e_f$ with target $s_-$.
    \end{itemize}
\end{Construction}

Next, we show that the fiber $P^{-1}u$ at a vertical morphism is also a category, embedded vertically in $\DblCat$.

\begin{Lemma}
\label{lem: Ver commutes with [1]}
    Given a double category $\dC$ and $i=0,1$, there is an isomorphism of categories 
    \[ \Ver_i \dHom{\dV[1],\dE} \cong (\Ver_i\dE)^{[1]}\]
    natural in $\dC$. 
\end{Lemma}

\begin{proof}
    By \cref{2-adjunctions}, we have a $2$-adjunction
    \[ \dV\colon \Cat \leftrightarrows \DblCat_v \colon \Ver_0. \]
    In particular, for the $\Cat$-enriched homs 
    $\Ver_0\dHom{-,-}$ of $\DblCat_v$ and 
    $(-)^{(-)}$ of $\Cat$, this means that we have an isomorphism of categories
    \[
    \Ver_0\dHom{\dV[1],\dE}\cong (\Ver_0\dE)^{[1]}.
    \]
    This shows the case $i=0$. For the case $i=1$, observe that, for every double category $\dC$, we have $\Ver_1\dC\cong\Ver_0\dHom{\dH[1],\dC}$, where $\dH[1]$ denotes the double category free on a horizontal morphism. Then we have natural isomorphisms of categories
    \begin{align*}
        \Ver_1\dHom{\dV[1],\dE}
        &\cong
        \Ver_0\dHom{\dH[1],\dHom{\dV[1],\dE}}
        & \text{Observation}
        \\
        &\cong
        \Ver_0\dHom{\dH[1]\times\dV[1],\dE}
        &
        (-)\times\dV[1]\ \dashv\ \dHom{\dH[1],-}
        \\
        &\cong
        \Ver_0\dHom{\dV[1],\dHom{\dH[1],\dE}}
        &
        \dH[1]\times(-)\ \dashv\ \dHom{\dH[1],-}
        \\
        &\cong
        (\Ver_0\dHom{\dH[1],\dE})^{[1]}
        &  \text{Case } i=0
        \\
        &\cong
        (\Ver_1\dE)^{[1]}, & \text{Observation}
    \end{align*}
     where the product--internal homs adjunctions hold by \cref{prop: cartesian closed}.
\end{proof}

\begin{Lemma} \label{prop: PV1 double disc fib}
    Given a discrete double fibration $P\colon \dE\to \dC$, the induced double functor \[
    \dHom{\dV[1],P}\colon \dHom{\dV[1],\dE}\to \dHom{\dV[1],\dC} \]
    is also a discrete double fibration.
\end{Lemma}

\begin{proof}
     We have to show that the following square
      is a pullback in $\Cat$
      \[
      \begin{tikzcd}
            {\Ver_1(\dHom{\dV[1],\dE})} \arrow[d, "{\Ver_1 (\dHom{\dV[1],P})}"'] \arrow[r, "t"] & {\Ver_0(\dHom{\dV[1],\dE})} \arrow[d, "{\Ver_0 (\dHom{\dV[1],P})}"] \\
            {\Ver_1(\dHom{\dV[1],\dC})} \arrow[r, "t"']                                  & {\Ver_0(\dHom{\dV[1],\dE})}                                 
        \end{tikzcd}
      \]
      By the natural isomorphism in \cref{lem: Ver commutes with [1]}, this is equivalent to showing that the following square is a pullback in $\Cat$
      \[
      \begin{tikzcd}
            {(\Ver_1\dE)^{[1]}} \arrow[d, "{(\Ver_1 P)^{[1]}}"'] \arrow[r, "{t^{[1]}}"] & {(\Ver_0\dE)^{[1]}} \arrow[d, "{(\Ver_0 P)^{[1]}}"] \\
            {(\Ver_1\dC)^{[1]}} \arrow[r, "{t^{[1]}}"']                                 & {(\Ver_0\dC)^{[1]}},                                
        \end{tikzcd}
      \]
     But this follows directly from the facts that $P$ is a discrete double fibration and that $(-)^{[1]}$ preserves pullbacks, as a right adjoint.
\end{proof}

\begin{Proposition}
    Given a discrete double fibration $P\colon \dE\to \dC$  and a vertical morphism $u\colon x\bulletarrow x'$ in $\dC$, the double category $P^{-1}u$ is a vertical double category. We also denote by $P^{-1}u$ the corresponding category. 
\end{Proposition}

\begin{proof}
    This follows from applying \cref{prop: Px category} to the discrete double fibration $\dHom{\dV[1],P}$ from \cref{prop: PV1 double disc fib}.
\end{proof}

Unpacking definition, we have the following explicit description of the fiber $P^{-1}u$. 

\begin{Remark}\label{descr: fiber at u}
    Given a vertical morphism $u\colon x\bulletarrow x'$ in $\dC$, the category $P^{-1}u$ is the category whose
    \begin{itemize}[leftmargin=1cm]
        \item objects are vertical morphisms $u_-\colon x_-\bulletarrow x'_-$ in $\dE$ such that $Pu_-=u$, 
        \item morphisms $u_-\to \hat{u}_-$ are pairs $(s_-,s'_-)$ of vertical morphisms $s_-$ and $s'_-$ in $\dE$ such that $Ps_-=e_x$ and $Ps'_-=e_{x'}$, and making the following diagram in $\dE$ commutes.
        \begin{equation} \label{eq: comm square in Pu} \begin{tikzcd}
        x_-
        \arrow[d,"u_-"',"\bullet" marking]
        \arrow[r,"s_-","\bullet" marking]
        &
        \hat{x}_-
        \arrow[d,"\hat{u}_-","\bullet" marking]
        \\
        x'_-
        \arrow[r,"s'_-"' {yshift=-2pt},"\bullet" marking]
        &
        \hat{x}'_-.
    \end{tikzcd} \end{equation}
    \end{itemize}
\end{Remark}

Moreover, every square in $\dC$ acts on fibers by taking unique lifts as follows.

\begin{Construction}\label{def: alpha on fibers}
    Given a square $\alpha\colon \edgesquare{u}{f}{f'}{v}$ in $\dC$, there is an induced functor
    \[ \alpha^*\colon P^{-1}v\to P^{-1}u \]
    given by sending
    \begin{itemize}[leftmargin=1cm]
        \item an object $v_-\colon y_-\bulletarrow y'_-$ in $P^{-1}v$ to the object $\alpha^*v_-\colon f^*y_-\bulletarrow f'^*y'_-$ in $P^{-1}u$, i.e., the source of the unique lift of $\alpha$ with target $v_-$,
        \item a morphism $(s_-,s'_-)\colon v_-\to v'_-$ in $P^{-1}v$ to the morphism $(e_f^*s_-,e_{f'}^*s'_-)\colon \alpha^*v_-\to  \alpha^*v'_-$ in~$P^{-1}u$, where $e_f^*s_-$ and $e_{f'}^*s'_-$ are the sources of the unique lifts of the identity squares $e_f$ and $e_{f'}$ with targets $s_-$ and $s'_-$, respectively.
    \end{itemize}
    Note that the assignment on morphisms is well-defined as 
    \[ \alpha^*v'_-\bullet e_f^* s_-= \alpha^*(v'_-\bullet s_-)=\alpha^*(s'_-\bullet v_-) = e_{f'}^*s'_-\bullet \alpha^* v_- \]
    by unicity of lifts and by definition of the morphism $(s_-,s'_-)$ in $P^{-1}v$.  
\end{Construction}

Next, we study the double functor $P_u$, which now amounts simply to an ordinary functor. We show that this is a two-sided discrete fibration. Hence, using the equivalence from \cref{thm: TSdisc vs Prof}, this will yield a profunctor associated with each vertical morphism $u$.

\begin{Proposition}
\label{prop: P_u is TSdiscFib}
    Given a discrete double fibration $P\colon \dE\to \dC$ and a vertical morphism $u\colon x\bulletarrow x'$ in $\dC$, the functor 
    \[ P_u\colon P^{-1}u\to P^{-1}x\times P^{-1}x' \]
    is a two-sided discrete fibration. 
\end{Proposition}

\begin{proof}
    First note that the functor $P_u\colon P^{-1}u\to P^{-1}x\times P^{-1}x'$ is given by sending 
    \begin{itemize}[leftmargin=1cm]
        \item an object $u_-\colon x_-\bulletarrow x'_-$ in $P^{-1}u$ to the object $(x_-,x'_-)$ of $P^{-1}x\times P^{-1}x'$, 
        \item a morphism $(s_-,s'_-)\colon u_-\to \hat{u}_-$ in $P^{-1}u$ to the morphism $(s_-,s'_-)$ of $P^{-1}x\times P^{-1}x'$.
    \end{itemize}
    We write $P_u=(P_u^1,P_u^2)$ for its components.

    We show that $(P_u^1,P_u^2)$ satisfies the conditions of a two-sided discrete fibration from 
    \cref{def: two-sided fib}. To prove (1), let $\hat{u}_-\colon \hat{x}_-\to \hat{x}'_-$ be an object in $P^{-1}u$ and $s_-\colon x_-\bulletarrow \hat{x}_-$ be a morphism in $P^{-1}x$. Then the unique $P_u^1$-lift of $s_-$ with target $\hat{u}_-$ that lies in the fiber of $P_u^2$ at $\hat{x}'_-$ is given by the following commutative square.
    \[
    \begin{tikzcd}
    x_-
    \arrow[r,"\bullet" marking,"s_-"]
    \arrow[d,swap,"\bullet" marking,"\hat{u}_-\bullet s_-"]
    &
    \hat{x}_-
    \arrow[d,"\bullet" marking,"\hat{u}_-"]
    \\
    \hat{x}'_-
    \arrow[r,"\bullet" marking,"e_{\hat{x}'_-}"']
    &
    \hat{x}'_-
    \end{tikzcd}
    \]
    Condition (2) can be shown analogously.
     
     To see (3), let 
    $(s_-,s'_-)\colon u_-\to \hat{u}_-$ be a morphism in $P^{-1}u$ as in \eqref{eq: comm square in Pu}. Then the unique lifts of $s_-$ and $s'_-$ provided by conditions (1) and (2) are the commutative squares
    \[
    \begin{tikzcd}
    x_-
    \arrow[r,"\bullet" marking,"s_-"]
    \arrow[d,swap,"\bullet" marking,"\hat{u}_-\bullet s_-"]
    &
    \hat{x}_-
    \arrow[d,"\bullet" marking,"\hat{u}_-"]
    \\
    \hat{x}'_-
    \arrow[r,"\bullet" marking,"e_{\hat{x}'_-}"']
    &
    \hat{x}'_-
    \end{tikzcd}
    \quad \text{and}\quad
    \begin{tikzcd}
    x_-
    \arrow[r,"\bullet" marking,"e_{x_-}"]
    \arrow[d,swap,"\bullet" marking,"u_-"]
    &
    x_-
    \arrow[d,"\bullet" marking,"s'_-\bullet u_-"]
    \\
    x'_-
    \arrow[r,"\bullet" marking,"s'_-"' {yshift=-2pt}]
    &
    \hat{x}'_-
    \end{tikzcd}
    \]
    respectively. Then, by definition of the morphism $(s_-,s'_-)$ in $P^{-1}u$, we have that $\hat{u}_-\bullet s_-=s'_-\bullet u_-$ and, moreover, the two squares compose to the commutative square \eqref{eq: comm square in Pu} defining $(s_-,s'_-)$. This proves that $P_u$ is a two-sided discrete fibration.
\end{proof}

In particular, as lax double presheaves are normal, we want to be able to associate the identity profunctor, i.e., the hom functor, to a vertical identity. The following result allows us to do so. 

\begin{Lemma} 
\label{Prop: fiber at vertical id}
    Given a discrete double fibration $P\colon \dE\to \dC$ and an object $x$ in $\dC$, the two-sided discrete fibration 
    \[ P_{e_x}\colon P^{-1}e_x\to P^{-1}x\times P^{-1}x \]
    is canonically isomorphic to the identity two-sided discrete fibration
    \[ (s,t)\colon (P^{-1}x)^{[1]}\to P^{-1}x\times P^{-1}x. \]
\end{Lemma}

\begin{proof}
    By definition of $P_{e_x}$ (see \cref{constr: definition of P_u}), to get the desired isomorphism it suffices to show that the following commutative square in $\DblCat$ is a pullback.
    \[
    \begin{tikzcd}
        {\dHom{\dV[1],P^{-1}x}=(P^{-1}x)^{[1]}} \arrow[d] \arrow[r] \arrow[rd,near start, "\lrcorner", phantom] & {\dHom{\dV[1],\dE}} \arrow[d, "{\dHom{\dV[1],P}}"] \\
        {[0]} \arrow[r,swap, "e_x"] & {\dHom{\dV[1],\dC}}  
    \end{tikzcd}
    \]
    However, this follows directly from the definition of the fiber $P^{-1}x$ and the fact that $\dHom{\dV[1],-}$ preserves pullbacks, as a right adjoint.
\end{proof}

\subsection{Representable discrete double fibrations} 

\label{subsec: repdiscfib}

We now introduce our first examples of discrete double fibrations.

\begin{Definition}
    Given a double category $\dC$ and an object $\hat{x}$ in $\dC$, we define the \emph{slice double category} over $\hat{x}$ to be the following pullback in $\DblCat$.
    \[
    \begin{tikzcd}
    \dC/\hat{x} \arrow[r] \arrow[d] \arrow[rd,phantom,near start, "\lrcorner"] & {\dHom{\dH[1],\dC}} \arrow[d, "{(s,t)}"] \\
    \dC \arrow[r, swap,"\id_{\dC}\times\{\hat{x}\}"]             & \dC\times\dC 
    \end{tikzcd}
    \]
\end{Definition}

\begin{Remark} \label{def: double slice}
    Explicitly, given a double category $\dC$ and an object $\hat{x}$, we have that $\dC/\hat{x}$ is the
    double category whose
    \begin{itemize}[leftmargin=1cm]
        \item objects are pairs 
        $(x,g)$ of an object $x$ in $\dC$ and a horizontal morphism $g\colon x\to\hat{x}$ in $\dC$,
        \item horizontal morphisms $(x,g)\to (y,h)$ are horizontal morphisms $f\colon x\to y$ in $\dC$ such that $g=h\circ f$,
        \item vertical morphisms $(x,g)\bulletarrow (x',g')$ are pairs $(u,\eta)$
        of a vertical morphism $u\colon x\bulletarrow x'$ in $\dC$ and a square $\eta\colon\edgesquare{u}{g}{g'}{e_{\hat{x}}}$ in $\dC$,
        \item squares $\alpha\colon \edgesquare{(u,\eta)}{f}{f'}{(v,\theta)}$ are squares 
        $\alpha\colon\edgesquare{u}{f}{f'}{v}$ in $\dC$ such that 
        $\eta=\theta\circ\alpha$.
    \end{itemize}
    Moreover, the canonical projection double functor $\dC/\hat{x}\to \dC$ is given by projecting onto the first component.
\end{Remark}

As claimed before, the canonical projection is a discrete double fibration. A proof can be found in \cite[Proposition 4.17]{internal_GC} by taking $\mathcal{V}=\Cat$.

\begin{Proposition}
\label{ex: C/x->C is discrete double}
    Given a double category $\dC$ and an object $\hat{x}$ in $\dC$, the canonical projection from the slice double category over $\hat{x}$
    \[ \dC/\hat{x}\to\dC \]
    is a discrete double fibration.
\end{Proposition}

\section{Grothendieck construction for lax double presheaves}

In this section, we study the link between lax double presheaves and discrete double fibrations. More precisely, given a double category $\dC$, we prove that there is a pseudo-natural $2$-equivalence---induced by the Grothendieck construction---between the $2$-category of lax double presheaves over $\dC$ and the $2$-category of discrete double fibrations over $\dC$. 

In \cref{subsec: GC1}, we define the Grothendieck construction of a lax double presheaf over a double category $\dC$. We show, using the coherence conditions for normal lax double functors, that this actually defines a discrete double fibration over $\dC$, where the base is an actual strict double category. This may sound surprising at first since only the base double category of a lax double presheaf is strict, but not the presheaf itself. We further show that this construction is pseudo-natural in $\dC$. In \cref{subsec: GC2}, we construct the weak inverse of the Grothendieck construction induced by taking fibers of discrete double fibrations. We show that there are $2$-natural isomorphisms relating the Grothendieck construction and its weak inverse, henceforth giving the desired $2$-equivalence.

\subsection{Grothendieck construction} 
\label{subsec: GC1}

We start by defining the Grothendieck construction of a lax double presheaf. 

\begin{Construction}
\label{def: dgro}
    Given a lax double presheaf $X\colon\dC^{\op}\to\dCat$, its 
    \emph{Grothendieck construction} is the double category $\dgro_\dC X$ whose 
    \begin{itemize}[leftmargin=1cm]
        \item objects are pairs $(x,x_-)$ of objects $x$ in $\dC$ and $x_-$ in $Xx$,
        \item horizontal morphisms 
        $(x,x_-)\to (y,y_-)$ are horizontal morphisms 
        $f\colon x\to y$ in $\dC$ satisfying 
        $x_-=Xf(y_-)$, where we recall that $Xf$ is the data of a functor $Xf\colon Xy\to Xx$,
        \item vertical morphisms 
        $(x,x_-)\bulletarrow (x',x'_-)$ are pairs $(u,u_-)$ of a vertical morphism 
        $u\colon {x\bulletarrow x'}$ in~$\dC$ and an element 
        $u_-$ in $Xu(x_-,x'_-)$, where we recall that $Xu$ is the data of a profunctor $Xu\colon Xx^{\op}\times Xx'\to\Set$,
        \item squares 
        \[
        \begin{tikzcd}
            {(x,x_-)} \arrow[d,"\bullet" marking, "{(u,u_-)}"'] \arrow[r, "{f}"] \arrow[rd, phantom,"{\alpha}"] 
            &
            {(y,y_-)} \arrow[d,"\bullet" marking, "{(v,v_-)}"] \\
            {(x',x'_-)} \arrow[r, "{f'}"'] &[30pt]
            {(y',y'_-)}                  
\end{tikzcd}
        \]
        are squares 
        $\alpha\colon\edgesquare{u}{f}{f'}{v}$ in $\dC$ satisfying
        ${u_-=(X\alpha)_{y_-,y'_-}(v_-)}$, where we recall that $X\alpha$ is the data of a natural transformation 
        $X\alpha\colon Xv\Rightarrow Xu\circ(Xf^{\op}\times Xf')\colon Xy^{\op}\times Xy'\to \Set$.
    \end{itemize}
    Compositions and identities for horizontal morphisms and squares are defined as in $\dC$. The vertical composite of two composable vertical morphisms 
    $(x,x_-)\overset{(u,u_-)}{\bulletarrow}(x',x'_-)\overset{(u',u'_-)}{\bulletarrow}(x'',x''_-)$ is given by the pair $(u'\bullet u,u'_-\bullet u_-)$ of the composite $u'\bullet u$ in $\dC$ and the element
    \[
        u'_-\bullet u_-\coloneqq 
        (\mu_{u,u'})_{x_-,x''_-}([u'_-,u_-]),
     \]
     where $[u'_-,u_-]$ is the element in $(Xu'\bullet Xu)(x_-,x''_-) =\int^{\hat{x}'_-\in X'x'}Xu'(\hat{x}'_-,x''_-)\times Xu(x_-,\hat{x}'_-)$ represented by the element $(u'_-,u_-)$ of $Xu'(x',x'')\times Xu(x,x')$,  and $\mu_{u,u'}$ is the composition comparison natural transformation
    \[ \mu_{u,u'}\colon Xu'\bullet Xu\Rightarrow X(u'\bullet u)\colon Xx^{\op}\times Xx''\to \Set. \]
     The vertical identity at an object $(x,x_-)$ is given by the pair 
     $(e_x,1_{x_-})$, where we recall that $Xe_x=e_{Xx}$ is the hom functor $Xx(-,-)\colon Xx^{\op}\times Xx\to \Set$, as $X$ is normal.
     
     It comes with a canonical projection double functor 
     $\pi_X\colon\dgro_\dC X\to\dC$ given by projecting onto the first component.
\end{Construction}

\begin{Lemma}
    The Grothendieck construction $\dgro_{\dC}X$ is a strict double category. 
\end{Lemma}

\begin{proof}
    We first show that composition of horizontal morphisms and squares is well-defined. Given composable horizontal morphisms $f\colon (x,x_-)\to (y,y_-)$ and $g\colon (y,y_-)\to (z,z_-)$ in $\dgro_{\dC}X$, then the composite $g\circ f\colon (x,x_-)\to (z,z_-)$ is also a horizontal morphism in $\dgro_{\dC}X$ since 
     \[ X(g\circ f)(z_-)=Xf(Xg(z_-))=Xf(y_-)=x_-. \]
    Then, given horizontally composable squares in $\dgro_\dC X$
    \[
        \begin{tikzcd}
            {(x,x_-)} \arrow[d,"\bullet" marking, "{(u,u_-)}"'] \arrow[r, "{f}"] \arrow[rd, phantom,"{\alpha}"] 
            &
            {(y,y_-)} 
            \arrow[d,"\bullet" marking, "{(v,v_-)}"] 
            \arrow[r,"g"]
            \arrow[rd,phantom,"\beta"]
            &
            (z,z_-)
            \arrow[d,"\bullet" marking,"{(w,w_-)}"]
            \\
            {(x',x'_-)} \arrow[r, "{f'}"'] &[30pt]
            {(y',y'_-)}  
            \arrow[r,swap,"g'"]
            &
            (z',z'_-)
\end{tikzcd}
        \]
        their horizontal composite $\beta\circ \alpha$ is also a square in $\dgro_\dC X$ since
        \begin{align*}
            X(\beta\circ\alpha)_{z_-,z'_-}(w_-)
            &=
            (X\alpha)_{Xg(z_-),Xg'(z'_-)}((X\beta)_{z_-,z'_-}(w_-))
            &
            \text{Functorality of }X
            \\
            &=
            (X\alpha)_{y_-,y'_-}((X\beta)_{z_-,z'_-}(w_-))
            &
            \textstyle\text{Horizontal morphisms in }\dgro_\dC X
            \\
            &=
            (X\alpha)_{y_-,y'_-}(v_-)=u_-
            &
            \textstyle\text{Squares in }\dgro_\dC X
        \end{align*}
    Finally, given vertically composable squares in $\dgro_\dC X$
    \[
    \begin{tikzcd}
        {(x,x_-)} \arrow[r, "f"] \arrow[d, "\bullet" marking,"{(u,u_-)}"'] \arrow[rd, phantom,"\alpha"]       & {(y,y_-)} \arrow[d,"\bullet" marking, "{(v,v_-)}"]     \\
        {(x',x'_-)} \arrow[r, "f'"] \arrow[d, "\bullet" marking, "{(u',u'_-)}"'] \arrow[rd, phantom,"\alpha'"] & {(y',y'_-)} \arrow[d,"\bullet" marking, "{(v',v'_-)}"] \\
        {(x'',x''_-)} \arrow[r, "f''"']                                              & {(y'',y''_-)}                       
\end{tikzcd}
    \]
    their vertical composite $\alpha'\bullet \alpha$ is also a square in $\dgro_\dC X$ since
    \begin{align*}
        X(\alpha'\bullet\alpha&)_{y_-,y''_-}(v'_-\bullet v_-) & \\
        &=
        X(\alpha'\bullet\alpha)_{y_-,y''_-}((\mu_{v,v'})_{y_-,y''_-}([v'_-,v_-]))
        &
        \text{Definition of }\bullet
        \\
        &=
        (\mu_{u,u'})_{x_-,x''_-}((X\alpha'\bullet X\alpha)_{y_-,y''_-}([v'_-,v_-]))
        &
        \text{Naturality of } \mu
        \\
        &=
        (\mu_{u,u'})_{x_-,x''_-}([(X\alpha')_{y'_-,y''_-}(v'_-),(X\alpha)_{y_-,y'_-}(v_-)])
        &
        \text{Comp.~of profunctors}
        \\
        &=
        (\mu_{u,u'})_{x_-,x''_-}([u'_-,u_-])
        &
        \textstyle\text{Squares in }\dgro_\dC X
        \\
        &=
        u'_-\bullet u_-.
        &
        \text{Definition of }\bullet
    \end{align*}
    
    Now, it remains to show that composition of vertical morphisms in $\dgro_{\dC}X$ is strictly associative and unital. We start by showing associativity. Given three composable vertical morphisms $(x,x_-)\overset{(u,u_-)}{\bulletarrow}
        (x',x'_-)\overset{(u',u'_-)}{\bulletarrow}
        (x'',x''_-)\overset{(u'',u''_-)}{\bulletarrow}
        (x''',x'''_-)$ in $\dgro_{\dC}X$, we have that
        \begin{align*}
            (u''_-\bullet u'_-)\bullet u_-
            &=
            (\mu_{u,u''\bullet u'})_{x'_-,x'''_-}([u''_-\bullet u'_-,u_-]) 
            &
            \text{Definition of }\bullet
            \\
            &=
            (\mu_{u,u''\bullet u'})_{x'_-,x'''_-}([(\mu_{u',u''})_{x'_-,x'''_-}([u''_-,u'_-]),u_-])
            &
            \text{Definition of }\bullet
            \\
            &=
            (\mu_{u'\bullet u,u''})_{x_-,x'''_-}
            ([u''_-,(\mu_{u,u'})_{x_-,x''_-}([u'_-,u_-]])
            &
            \text{Associativity of }\mu
            \\
            &=
            (\mu_{u'\bullet u,u''})_{x_-,x'''_-}
            ([u''_-,u'_-\bullet u_-])
            &
            \text{Definition of }\bullet
            \\
            &=
            u''_-\bullet(u'_-\bullet u_-)
            &
            \text{Definition of }\bullet
        \end{align*}
        
        Finally, we show unitality. Given a vertical morphism $(u,u_-)\colon (x,x_-)\bulletarrow (x',x'_-)$ in $\dgro_{\dC}X$, we have that
        \[ u_-\bullet 1_{x_-}=(\mu_{u,e_x})_{x_-,x'_-}([u_-,1_{x_-}])=(1_{Xu})_{x_-,x'_-}(u_-)=1_{Xu(x_-,x'_-)}(u_-)=u_- \]
        using that $\mu_{u,e_x}$ is the identity at $Xu$. Similarly, we have $1_{x'_-}\bullet u_-=u_-$. This concludes the proof that $\dgro_\dC X$ is a strict double category. 
\end{proof}

As a first example, we compute the Grothendieck construction of a representable lax double presheaf. 

\begin{Proposition}
    \label{prop: dgro of representable}
    Given a double category $\dC$ and an object $\hat{x}$ in $\dC$, we have
    \[
    \textstyle\dgro_\dC\dC(-,\hat{x})=\dC/\hat{x}.
    \]
\end{Proposition}

\begin{proof}
    Explicitly, the double category $\dgro_\dC\dC(-,\hat{x})$ is such that
    \begin{itemize}[leftmargin=1cm]
        \item an object $(x,x_-)$ in $\dgro_\dC\dC(-,\hat{x})$ is the data of an object $x$ in $\dC$ and an object $x_-$ in $\dC(x,\hat{x})$, i.e., a horizontal morphism $x_-\colon x\to \hat{x}$ in $\dC$, 
         \item a horizontal morphism $f\colon (x,x_-)\to (y,y_-)$ in $\dgro_\dC\dC(-,\hat{x})$ is the data of a horizontal morphism $f\colon x\to y$ in $\dC$ satisfying $x_-=\dC(f,\hat{x})(y_-)=y_-\circ f$, 
         \item a vertical morphism $(u,u_-)\colon (x,x_-)\bulletarrow (x',x_-')$ in $\dgro_\dC\dC(-,\hat{x})$ is the data of a vertical morphism $u\colon x\bulletarrow x'$ in $\dC$ and an element $u_-$ 
 in $\dC(u,\hat{x})(x_-,x'_-)$, i.e., a square $u_-\colon\edgesquare{u}{x_-}{x'_-}{e_{\hat{x}}}$ in~$\dC$, 
         \item a square $\alpha\colon\edgesquare{(u,u_-)}{f}{f'}{(v,v_-)}$ in $\dgro_\dC\dC(-,\hat{x})$ is the data of a square $\alpha\colon\edgesquare{u}{f}{f'}{v}$ in $\dC$ satisfying $u_-=\dC(\alpha,\hat{x})_{y_-,y'_-}(v_-)=v_-\circ\alpha$.
    \end{itemize}
    By comparing with the data of the double category $\dC/\hat{x}$ as described in \cref{def: double slice}, we get the desired result. 
\end{proof}

By \cref{ex: C/x->C is discrete double}, we know that the canonical projection $\dgro_\dC\dC(-,\hat{x})=\dC/\hat{x}\to \dC$ is a discrete double fibration. The following shows that the Grothendieck construction of any lax double presheaf is also a discrete double fibration.

\begin{Proposition}
\label{prop: pi_x is ddisc fib}
    The canonical projection $\pi_X\colon\dgro_\dC X\to\dC$ is a discrete double fibration.
\end{Proposition}

\begin{proof}
    We check the conditions of a discrete double fibration from \cref{remark: dbldiscfib lifts}. 
    \begin{enumerate}[leftmargin=1cm]
        \item Given an object $(y,y_-)$ in $\dgro_\dC X$ and a horizontal morphism
    $f\colon x\to y$ in $\dC$, the unique lift of~$f$ with target $(y,y_-)$ is given by the horizontal morphism 
    $f\colon (x,Xf(y_-))\to (y,y_-)$ in~$\dgro_\dC X$.
        \item Given a vertical morphism 
    $(v,v_-)\colon (y,y_-)\bulletarrow (y',y'_-)$ in $\dgro_\dC X$ and a square
    $\alpha\colon\edgesquare{u}{f}{f'}{v}\colon\nodesquare{x}{x'}{y}{y'}$ in~$\dC$, the unique lift of~$\alpha$ with target $(v,v_-)$ is the square
    \[
    \begin{tikzcd}
        (x,Xf(y_-))
        \arrow[r,"f"]
        \arrow[d,"\bullet" marking,swap,"{(u,(X\alpha)_{y_-,y'_-}(v_-))}"]
        \arrow[rd,phantom,"\alpha"]
        &
        (y,y_-)
        \arrow[d,"\bullet" marking,"{(v,v_-)}"]
        \\
        (x',Xf'(y'_-))
        \arrow[r,"f'"']
        &
        (y',y'_-).
    \end{tikzcd} 
    \]
    \end{enumerate}
    This shows that $\pi_X\colon\dgro_\dC X\to\dC$ is a discrete double fibration. 
\end{proof}

We can extend the Grothendieck construction to a $2$-functor $\dgro_\dC\colon \P{C}\to \dFib(\dC)$. We first describe its assignment on morphisms. 

\begin{Construction}
\label{def: dgro of horizontal transformations}
    Given a horizontal transformation 
    $F\colon X\Rightarrow Y\colon\dC^{\op}\to\dCat$, we define a double functor over $\dC$
    \[\textstyle \dgro_\dC F\colon \dgro_\dC X\to \dgro_\dC Y \] sending
    \begin{itemize}
        \item 
        an object $(x,x_-)$ in $\dgro_\dC X$ to the object $(x,F_xx_-)$ in $\dgro_\dC Y$, where $F_x x_-$ is the image of $x_-$ under the functor $F_x\colon Xx\to Yx$, 
        \item a horizontal morphism 
        $f\colon (x,x_-)\to (y,y_-)$ in $\dgro_\dC X$ to the horizontal morphism in $\dgro_\dC Y$
        \[ f\colon (x,F_xx_-)\to (y,F_yy_-), \]
        where $Yf(F_yy_-)=F_x(Xf(y_-))=F_x x_-$ by naturality of $F_x$, 
        \item 
        a vertical morphism
        $(u,u_-)\colon (x,x_-)\bulletarrow (x',x'_-)$ in $\dgro_\dC X$ to the vertical morphism in $\dgro_\dC Y$ \[ (u,(F_u)_{x_-,x_-'}u_-)\colon(x,F_xx_-)\bulletarrow (x',F_{x'}x'_-),  \]
        where $(F_u)_{x_-,x_-'}u_-$ is the image of $u_-$ under the component \[ (F_u)_{x_-,x_-'}\colon Xu(x_-,x_-')\to Yu(F_xx_-,F_{x'}x_-') \] of the natural transformation $F_u\colon Xu\Rightarrow Yu(F_x^{\op}\times F_{x'})\colon Xx^{\op}\times Xx'\to \Set$. 
        \item a square $\alpha\colon\edgesquare{(u,u_-)}{f}{f'}{(v,v_-)}\colon\nodesquare{(x,x_-)}{(x',x'_-)}{(y,y_-)}{(y',y'_-)}$ in $\dgro_\dC X$
        to the square in $\dgro_\dC Y$
        \[
        \begin{tikzcd}
            {(x,F_xx_-)} \arrow[d,"\bullet" marking, swap,"{(u,(F_u)_{x_-,x'_-}u_-)}"] \arrow[r, "f"] \arrow[rd,phantom, "\alpha"] & {(y,F_yy_-)} \arrow[d,"\bullet" marking,"{(v,(F_v)_{y_-,y'_-}v_-)}"] \\
            {(x',F_{x'}x'_-)} 
            \arrow[r, "f'"']                 & {(y',F_{y'}y'_-)} 
        \end{tikzcd}
        \] 
        where $(Y\alpha)_{F_y y_-,F_{y'}y'_-}((F_v)_{y_-,y_-'}v_-)=(F_u)_{x_-,x'_-}((X\alpha)_{y_-,y'_-}v_-)=(F_u)_{x_-,x'_-}u_-$ by naturality of $F_u$. 
    \end{itemize}
    
    Note that, by construction, the double functor $\dgro_\dC F$ clearly lies over $\dC$.
\end{Construction}

\begin{Lemma}
    The construction $\dgro_\dC F\colon \dgro_\dC X\to \dgro_\dC Y$ is a strict double functor. 
\end{Lemma}

\begin{proof}
    By construction, the double functor $\dgro_\dC F$ clearly preserves compositions and identities of horizontal morphisms and squares. It remains to show that it preserves compositions and identities of vertical morphisms. 
    
    Given composable vertical morphisms 
    $(x,x_-)\overset{(u,u_-)}{\bulletarrow}(x',x'_-)\overset{(u',u'_-)}{\bulletarrow}(x'',x''_-)$ in $\dgro_{\dC}X$, we have:
    \begin{align*}
        \textstyle\dgro_\dC F((u'&,u'_-)\bullet(u,u_-)) &\\
        &=
        (u'\bullet u,(F_{u'\bullet u})_{x_-,x''_-}(u'_-\bullet u_-))
        &
        \textstyle\text{Definition of }\dgro_\dC F
        \\
        &=
        (u'\bullet u,(F_{u'\bullet u})_{x_-,x''_-}(\mu_{u,u'})_{x_-,x''_-}([u'_-,u_-]))
        &
        \textstyle\text{Composition in }\dgro_\dC X
        \\
        &=
        (u'\bullet u,(\mu_{u,u'})_{x_-,x''_-}(F_{u'}\bullet F_u)_{x_-,x''_-}([u'_-,u_-]))
        &
        \text{Compatibility of } F \text{ and } \mu
        \\
        &=
        (u'\bullet u,(\mu_{u,u'})_{x_-,x''_-}([(F_{u'})_{x'_-,x''_-}u'_-,(F_u)_{x_-,x'_-}u_-]))
        &
        \text{Comp.~of profunctors}
        \\
        &=
        (u',(F_{u'})_{x'_-,x''_-}u'_-)\bullet (u,(F_u)_{x_-,x'_-}u_-)
        &
        \textstyle\text{Composition in }\dgro_\dC Y
        \\
        &=
        \textstyle\dgro_\dC F(u',u'_-)\bullet
        \dgro_\dC F(u,u_-).
        &
        \textstyle\text{Definition of }\dgro_\dC F
    \end{align*}
   This shows that $\dgro_\dC F$ preserves vertical composition. Now, given an object $(x,x_-)$ of $\dgro_\dC X$, we have:
    \begin{align*}
        \textstyle\dgro_\dC F(e_{(x,x_-)})
        &=
        \textstyle\dgro_\dC F(e_x,1_{x_-})
        &
        \textstyle\text{Vertical identity in }\dgro_\dC X
        \\
        &=
        (e_x,(F_{e_x})_{x_-,x_-}1_{x_-})
        &
        \textstyle\text{Definition of }\dgro_\dC F
        \\
        &=
        (e_x,1_{F_x(x_-)})
        &
        F_{e_x}=1_{F_{x}}
        \\
        &=
        e_{(x,F_xx_-)}
        &
        \textstyle\text{Vertical identity in }\dgro_\dC Y
        \\
        &=
        e_{\dgro_\dC F(x,x_-)}.
        &
        \textstyle\text{Definition of }\dgro_\dC F
    \end{align*}
    This shows that $\dgro_\dC F$ preserves vertical identities and concludes the proof that $\dgro_\dC F$ is a strict double functor. 
\end{proof}

Now, we describe the assignment of $\dgro_\dC$ on $2$-morphisms.

\begin{Construction}
\label{def: dgro for globular modifications}
    Given a globular modification $A\colon\edgesquare{e_X}{F}{F'}{e_{Y}}\colon\nodesquare{X}{X}{Y}{Y}$ of lax double pre\-sheaves $\dC^{\op}\to \Cat$, we define a vertical transformation over $\dC$
   \[ \textstyle\dgro_{\dC}A\colon\dgro_{\dC}F\Bulletarrow \dgro_{\dC}F'\colon \dgro_{\dC} X\to \dgro_{\dC} Y \]
   whose
   \begin{itemize}[leftmargin=1cm]
       \item component at an object 
    $(x,x_-)$ in $\dgro_\dC X$ is given by the vertical morphism in $\dgro_\dC Y$
    \[ (e_x,(A_x)_{x_-})\colon (x,F_xx_-)\bulletarrow(x,F'_xx_-), \]
    where $(A_x)_{x_-}$ is the component at $x_-$ of the natural transformation $A_x\colon F_x\Rightarrow F'_x\colon Xx\to Yx$,
    \item component at a horizontal morphism 
    $f\colon (x,x_-)\to (y,y_-)$ in $\dgro_\dC X$ is given by the square in~$\dgro_\dC Y$
        \[
        \begin{tikzcd}
            {(x,F_xx_-)} \arrow[d,swap,"\bullet" marking, "{(e_x,(A_x)_{x_-})}"] \arrow[r, "f"] \arrow[rd,phantom, "e_f"] & {(y,F_yy_-)} \arrow[d,"\bullet" marking, "{(e_y,(A_y)_{y_-})}"] \\
            {(x,F'_xx_-)} \arrow[r, "f"']  & {(y,F'_yy_-)}                  
        \end{tikzcd}
        \]
        where $(Ye_f)_{F_yy_-,F'_yy_-}((A_y)_{y_-})=Yf((A_y)_{y_-})=(A_x)_{Xf(y_-)}=(A_x)_{x_-}$ by normality of $Y$ and horizontal compatibility of $A_x$.
   \end{itemize}
   
   Note that, by construction, the vertical transformation $\dgro_\dC A$ clearly lies over $\dC$. 
\end{Construction}

\begin{Lemma}
    The construction $\dgro_{\dC}A\colon\dgro_{\dC}F\Bulletarrow \dgro_{\dC}F'\colon \dgro_{\dC} X\to \dgro_{\dC} Y$ is a vertical transformation.
\end{Lemma}

\begin{proof}
    We first show that the components $(\dgro_\dC A)_{(x,x_-)}$ are natural in objects $(x,x_-)$. Given a vertical morphism
    $(u,u_-)\colon (x,x_-)\bulletarrow (x',x'_-)$ in
    $\dgro_\dC X$, we have to show that the following diagram in $\dgro_\dC Y$ commutes.
    \[
    \begin{tikzcd}
        {\dgro_\dC F(x,x_-)} \arrow[d,"\bullet" marking, "{\dgro_\dC F(u,u_-)}"'] \arrow[r,"\bullet" marking, "{(\dgro_\dC A)_{(x,x_-)}}"] 
        &[30pt]
        {\dgro_{\dC}F'(x,x_-)} \arrow[d,"\bullet" marking, "{\dgro_\dC F'(u,u_-)}"] \\
        {\dgro_\dC F(x',x'_-)} \arrow[r, "\bullet" marking, "{(\dgro_{\dC}A)_{(x',x'_-)}}"']              &[30pt]
        {\dgro_\dC F'(x',x'_-)}        
        \end{tikzcd}
    \]
    This we can do by a short computation:
    \begin{align*}
       \textstyle \dgro_\dC F'(u,&u_-)\circ  \textstyle (\dgro_{\dC}A)_{(x,x_-)} & \\
        &=
        (u,(F'_u)_{x_-,x'_-}(u_-))\bullet
        (e_x,(A_x)_{x_-})
        &
        \textstyle\text{Definition of }\dgro_\dC A\text{, }\dgro_\dC F'
        \\
        &=
        (u\bullet e_x,(\mu_{e_x,u})_{x_-,x'_-}([(F'_u)_{x_-,x'_-}(u_-),(A_x)_{x_-}])
        &
        \textstyle\text{Composition in }\dgro_\dC Y
        \\
        &=
        (e_{x'}\bullet u,
        (\mu_{u,e_{x'}})_{x_-,x'_-}([(A_{x'})_{x'_-},(F_u)_{x_-,x'_-}(u_-)])
        &
        \text{Vertical compatibility of }A
        \\
        &=
        (e_{x'},(A_{x'})_{x'_-})\bullet
        (u,(F_u)_{x_-,x'_-}(u_-))
        &
        \textstyle\text{Composition in }\dgro_\dC Y
        \\
        &=
        \textstyle(\dgro_\dC A)_{(x',x'_-)}\bullet
        \dgro_\dC F(u,u_-)
        &
        \textstyle\text{Definition of }\dgro_\dC A\text{, }\dgro_\dC F
    \end{align*}
    The naturality of the components $(\dgro_\dC A)_{f}$ in horizontal morphisms $f$ is then straightforward from the fact that these components are given by the vertical identity squares $e_f$. 
\end{proof}

Putting everything together, we get the following. 

\begin{Construction} \label{constr: dgroC}
    Given a double category $\dC$, we define a $2$-functor
    \[ \textstyle\dgro_\dC\colon\P{C}\to\dFib(\dC),
    \]
    sending
    \begin{itemize}[leftmargin=1cm]
        \item 
    a double presheaf
    $X\colon\dC^{\op}\to\dCat$ to the discrete double fibration $\pi_X\colon\dgro_\dC X\to\dC$ from \cref{def: dgro},
    \item a horizontal transformation 
    $F\colon X\Rightarrow Y$ to the double functor $\dgro_\dC F\colon \dgro_\dC X\to \dgro_\dC Y$ over $\dC$ from \cref{def: dgro of horizontal transformations},
    \item a globular modification $A\colon\edgesquare{e_X}{F}{F'}{e_{X'}}$ to the vertical transformation $\dgro_{\dC}A\colon\dgro_{\dC}F\Bulletarrow \dgro_{\dC}F'$ over $\dC$ from \cref{def: dgro for globular modifications}.
    \end{itemize}
    It is straightforward to check that this construction is $2$-functorial.
\end{Construction}

Next, we want to show that the construction $\dgro_\dC$ is pseudo-natural in $\dC$. For this, we start by constructing the pseudo-naturality comparison cells. 

\begin{Construction}
\label{constr: dgro_G}
    Given a double functor $G\colon\dC\to\dD$, we construct a $2$-natural isomorphism 
    \[
    \begin{tikzcd}
        \P{D}
        \arrow[r,"(G^{\op})^*"]
        \arrow[swap,d,"\dgro_\dD"]
        &
        \P{C}
        \arrow[d,"\dgro_\dC"]\arrow[ld,"\dgro_G"',"\cong",Rightarrow]
        \\
        \dFib(\dD)
        \arrow[r,swap,"G^*"]
        &
        \dFib(\dC)
    \end{tikzcd}
    \]
    whose component at a lax double presheaf $X\colon \dD^{\op}\to \dCat$ is the invertible double functor over $\dC$
    \[ \textstyle (\dgro_G)_X\colon \dgro_\dC(G^{\op})^*X=\dgro_\dC XG^{\op}\longrightarrow G^*\dgro_\dD X=\dC\times_\dD\dgro_\dD X \]
    given by sending 
    \begin{itemize}[leftmargin=1cm]
        \item an object $(x,x_-)$ in $\dgro_\dC XG^{\op}$, i.e., a pair of objects $x$ in $\dC$ and $x_-$ in $XG(x)$, to the object 
        $(x,(Gx,x_-))$ in $\dC\times_\dD\dgro_\dD X$,
        \item a horizontal morphism $f\colon (x,x_-)\to (y,y_-)$ in $\dgro_\dC XG^{\op}$, i.e., a horizontal morphism $f\colon x\to y$ in $\dC$ such that $XGf(y_-)=x_-$, to the horizontal morphism in $\dC\times_\dD\dgro_\dD X$
        \[ (f,Gf)\colon (x,(Gx,x_-))\to (y,(Gy,y_-)), \]
        \item a vertical morphism $(u,u_-)\colon (x,x_-)\bulletarrow (x',x'_-)$ in $\dgro_\dC XG^{\op}$, i.e., a pair of a vertical morphism $u\colon x \bulletarrow x'$ in $\dC$ and an element $u_-$ in $XGu(x_-,x_-')$, to the vertical morphism in $\dC\times_\dD\dgro_\dD X$
        \[ (u,(Gu,u_-))\colon (x,(Gx,x_-))\bulletarrow (x',(Gx',x'_-)), \]
        \item a square 
        $\alpha\colon\edgesquare{(u,u_-)}{f}{f'}{(v,v_-)}\colon\nodesquare{(x,x_-)}{(x',x'_-)}{(y,y_-)}{(y',y'_-)}$ in $\dgro_\dC XG^{\op}$, i.e., a square $\alpha\colon \edgesquare{u}{f}{f'}{v}$ in $\dC$ such that $(XG\alpha)_{x_-,x_-'}(v_-)=u_-$, to the square in $\dC\times_\dD\dgro_\dD X$
        \[
        \begin{tikzcd}
            {(x,(Gx,x_-))} \arrow[d,"\bullet" marking, "{(u,(Gu,u_-))}"'] \arrow[r, "{(f,Gf)}"] \arrow[rd, phantom,"{(\alpha,G\alpha)}"] & [10pt] {(y,(Gy,y_-))} \arrow[d, "\bullet" marking,"{(v,(Gv,v_-))}"] \\
            {(x',(Gx',x'_-))} \arrow[r, "{(f',Gf')}"']  & {(y',(Gy',y'_-))}.                       
        \end{tikzcd}
        \]
        
    \end{itemize}
    Note that it admits an obvious inverse given by projecting onto the first and last components, so that the components $(\dgro_G)_X$ are invertible.
\end{Construction}

\begin{Lemma}
    The construction $\dgro_G\colon\dgro_\dC\circ(G^{\op})^*\Rightarrow G^*\circ\dgro_\dD$ is a $2$-natural isomorphism. 
\end{Lemma}

\begin{proof}
    It remains to prove that the components $(\dgro_G)_X$ are $2$-natural in $X$. Given a horizontal transformation $F\colon X\Rightarrow Y\colon\dD^{\op}\to\dCat$, we have to show that the following diagram of double functors over $\dC$ commutes.
    \[
    \begin{tikzcd}
        \dgro_\dC XG^{\op}
        \arrow[r,"(\dgro_G)_X"]
        \arrow[swap,d,"\dgro_\dC (G^{\op})^*F=\dgro_\dC F \circ G^{\op}"]
        & [10pt]
        G^* \dgro_\dD X=\dC\times_\dD \dgro_\dD X
        \arrow[d,"{G^*(\dgro_\dD F)=\dC\times_\dD \dgro_\dD F}"]
        \\
        \dgro_\dC YG^{\op}
        \arrow[r,swap,"(\dgro_G)_Y"]
        &
        G^* \dgro_\dD Y=\dC\times_\dD \dgro_\dD Y
    \end{tikzcd}
    \]
    We check that the above diagram commutes on objects, and the proof is similar for horizontal morphisms, vertical morphisms, and squares. To do so, we evaluate both composites at an object $(x,x_-)$ in $\dgro_\dC XG^{\op}$:
        \begin{align*}
           \textstyle(\dgro_G)_Y\circ \dgro_\dC F\circ G^{\op} (x,x_-)
            & =
            \textstyle(\dgro_G)_Y(x,(F \circ G^{\op})_xx_-)
            &
            \textstyle\text{Definition of }\dgro_\dC F\circ G^{\op}
            \\
            &=
            \textstyle(\dgro_G)_Y(x,F_{Gx}x_-)
            & \text{Whiskering}
            \\
            &=
            (x,(Gx,F_{Gx}x_-))
            &
            \textstyle\text{Definition of }(\dgro_G)_Y
            \\
            &=
            \textstyle(\dC\times_\dD\dgro_\dD F)(x,(Gx,x_-))
            &
            \textstyle\text{Definition of } \dC\times_\dD\dgro_\dD F
            \\
            &=
            \textstyle(\dC\times_\dD\dgro_\dD F)\circ(\dgro_G)_X(x,x_-).
            &
            \textstyle\text{Definition of }(\dgro_G)_X
        \end{align*}
        This shows that the desired diagram commutes. 
        
        Next, given a globular modification 
        $A\colon\edgesquare{e_X}{F}{F'}{e_Y}\colon\nodesquare{X}{X}{Y}{Y}$ in $\P{D}$, we have to check that the following diagram of double functors and vertical transformations commutes
        \[
        \begin{tikzcd}
            & \dgro_\dC XG^{\op} \arrow[dd, "\dgro_\dC F\circ G^{\op}"', bend right=55] \arrow[dd, "\dgro_\dC F'\circ G^{\op}", bend left=55] \arrow[rr, "(\dgro_G)_X"] 
            &[30pt]
            & G^*\dgro_\dD X \arrow[dd, "G^*\dgro_\dD F'", bend left=55] \arrow[dd, swap,"G^*\dgro_\dD F", bend right=55] &    
            \\
            {} \arrow[rr, shorten <=30pt, shorten >=60pt,"\dgro_\dC A\circ G^{\op}" {xshift=-15pt},yshift=-5pt,"\bullet" {xshift=-15pt,yshift=-6pt,scale=1.3} marking, Rightarrow] 
            &[30pt]
            & {} \arrow[rr,shorten <=30pt,shorten >=30pt,"\bullet" marking, "G^*\dgro_\dD A", Rightarrow, yshift=-5pt] &  & {} 
            \\
            & \dgro_\dC YG^{\op}\arrow[rr, "(\dgro_G)_Y",swap] 
            &[30pt] 
            & G^*\dgro_\dD Y  &   
        \end{tikzcd}
         \]
        To do so, we first evaluate both whiskerings at an object $(x,x_-)$ of 
        $\dgro_\dC XG^{\op}$:
        \begin{align*}
            \textstyle
            (\dgro_G)_Y (\dgro_\dC A\circ G^{\op})_{(x,x_-)})
            &=
            \textstyle
            (\dgro_G)_Y(e_x,((A\circ G^{\op})_x)_{x_-})
            &
             \textstyle\text{Definition of } \dgro_\dC A\circ G^{\op}
            \\
            &=
            \textstyle
            (\dgro_G)_Y(e_x,(A_{Gx})_{x_-})
            &
            \text{Whiskering}
            \\
            &=
            (e_x,(Ge_x,(A_{Gx})_{x_-}))
            &
             \textstyle\text{Definition of } (\dgro_G)_Y
            \\
            &=
            (e_x,(e_{Gx},(A_{Gx})_{x_-}))
            &
            \text{Functorality of }G
            \\
            &=
            \textstyle
            (G^*\dgro_\dD A)_{(x,(Gx,x_-))}
            &
            \textstyle \text{Definition of }G^*\dgro_\dD A
            \\
            &=
            \textstyle
            (G^*\dgro_\dD A)_{(\dgro_G)_X(x,x_-)}.
            &
            \textstyle \text{Definition of } (\dgro_G)_X
        \end{align*}
        This shows that the components at objects of $\dgro_\dC XG^{\op}$ of the two whiskerings agree. The fact that their components at horizontal morphisms agree is then straightforward from the fact that they are both given by the vertical identity square $(e_f,e_{Gf})$ in $G^*\dgro_\dD Y$. 
\end{proof}

Putting the components $\dgro_\dC$ and $\dgro_G$ together, we get the following. 

\begin{Proposition}
    The Grothendieck construction defines a pseudo-natural transformation of pseudo-functors
    \[ \textstyle\dgro\colon\mathcal{P}\Rightarrow\dFib\colon\DblCat_h^{\coop}\to 2\Cat, \]
    where $\mathcal{P}$ is the $2$-functor from \cref{def: 2-functor P} and $\dFib$ is the pseudo-functor from \cref{def: dFib}.
\end{Proposition}

\begin{proof}
    We first show that the components $\dgro_G$ are functorial in $G$. Given composable double functors 
    $G\colon \dC\to\dD$ and $H\colon \dD\to\dE$, we have to show that the following composite coincides with the $2$-natural transformation $\dgro_{HG}$
    \[ \textstyle\dgro_\dC ((HG)^{\op})^*=\dgro_\dC (G^{\op})^*(H^{\op})^*\xRightarrow{\dgro_G (H^{\op})^*} G^* \dgro_\dD (H^{\op})^* \xRightarrow{G^*\dgro_H} G^*H^*\dgro_\dE \cong (HG)^* \dgro_\dE. \]
   When evaluated at a lax double presheaf $X\colon\dE^{\op}\to\dCat$, this amounts to showing that the following composite coincides with the double functor $(\dgro_{HG})_X$
   \[ \textstyle \dgro_C XH^{\op}G^{\op} \xrightarrow{(\dgro_G)_{XH^{\op}}} G^*\dgro_\dD XH^{\op} \xrightarrow{G^*(\dgro_H)_X} G^* H^* \dgro_\dE X\cong (HG)^*\dgro_\dE X. \] 
   But this is straightforward from the definitions of $\dgro_G$, $\dgro_H$, and $\dgro_{HG}$. Moreover, given a double category $\dC$, a similar computation shows that $\dgro_{\id_\dC}=\id_{\dgro_\dC}$.

    Finally, we show that the components $\dgro_G$ are natural in $G$. Given a horizontal transformation 
    $B\colon G\Rightarrow G'\colon\dC\to\dD$, we have to show that the following diagram of $2$-natural transformations commutes, 
    \[
    \begin{tikzcd}
        \dgro_\dC (G'^{\op})^* \arrow[d, "\dgro_\dC (B^{\op})^*"', Rightarrow] \arrow[r, "\dgro_{G'}", Rightarrow] & G'^*\dgro_\dD \arrow[d, "B^*\dgro_\dD", Rightarrow] \\
        \dgro_\dC (G^{\op})^* \arrow[r, "\dgro_G"', Rightarrow]              & G^*\dgro_\dD.            
\end{tikzcd}
    \]
    where we recall that the left-hand $(B^{\op})^*$ is as in \cref{def: 2-functor P} and the right-hand $B^*$ is as in \cref{constr:Fib(B)}.
    When evaluated at a lax double presheaf $X\colon\dD^{\op}\to\dCat$, this amounts to showing that the following diagram of double functors commutes.
    \[
    \begin{tikzcd}
        \dgro_\dC XG'^{\op} \arrow[d, "\dgro_\dC X\circ B^{\op}"'] \arrow[r, "(\dgro_{G'})_X"] & G'^*\dgro_\dD X \arrow[d, "B^*_{\int\!\!\!\int_\dD\! X}"] \\
        \dgro_\dC X G^{\op} \arrow[r, "(\dgro_G)_X"']              & G^*\dgro_\dD X.            
\end{tikzcd}
    \]
    We check that the above diagram commutes on objects, and the proof is similar for horizontal morphisms, vertical morphisms, and squares. To do so, we evaluate both composites at an object $(x,x_-)$ in $\dgro_\dC XG'^{\op}$:
    \begin{align*}
        \textstyle B^*_{\dgro_\dD X}(\dgro_{G'})_X(x,x_-) & = B^*_{\dgro_\dD X} (x, (G'x,x_-)) & \textstyle\text{Definition of }\dgro_{G'} \\
        & = (x, (B_x)^*(G'x,x_-)) & \textstyle\text{Definition of } B^*
        \\
        & = (x, (Gx, X(B_x)x_-)) & \textstyle\text{Unique lifts in } \dgro_\dD X \\
        &=
        \textstyle(\dgro_G)_X(x,X(B_x)x_-)
        &
         \textstyle\text{Definition of }(\dgro_G)_X
        \\
        &=\textstyle(\dgro_G)_X(x,(X\circ B^{\op})_x x_-) & 
        \text{Whiskering}
        \\
        &=\textstyle(\dgro_G)_X(\dgro_\dC (X\circ B^{\op})(x,x_-)
        &
       \textstyle \text{Definition of }\dgro_\dC X\circ B^{\op}
    \end{align*}
    This shows that the desired diagram commutes. 
\end{proof}

\subsection{Grothendieck \texorpdfstring{$2$}{2}-equivalence} 
\label{subsec: GC2}

We now show that the Grothendieck construction establishes a $2$-equivalence between the $2$-categories $\P{C}$ of lax double presheaves and $\dFib(\dC)$ of discrete double fibrations over a fixed double category $\dC$.

\begin{Theorem}
\label{thm: grothendieck equivalence}
\index{Grothendieck equivalence}
    Given a double category $\dC$, the Grothendieck construction
    \[
    \textstyle\dgro_{\dC}\colon\P{C}\longrightarrow\dFib(\dC)
    \]
    is a $2$-equivalence of $2$-categories, which is pseudo-natural in $\dC$.
\end{Theorem}

In order to show this, we construct a weak inverse $2$-functor $\ddel_\dC\colon \dFib(\dC)\to \P{C}$ of the $2$-functor $\dgro_\dC$ from \cref{constr: dgroC}. We start with its assignment on objects. 

\begin{Construction}
\label{constr: ddel on objects}
Given a discrete double fibration $P\colon\dE\to\dC$, we construct  a lax double presheaf 
\[ \ddel_\dC(P)\colon \dC^{\op}\to \dCat \]
which sends
    \begin{itemize}[leftmargin=1cm]
        \item an object 
        $x$ in $\dC$ to the category $P^{-1}x$ given by the fiber of $P$ at $x$, 
        \item a horizontal morphism $f\colon x\to y$ in $\dC$ to the functor 
        \[ \ddel_\dC(P)f\coloneqq f^*\colon \ddel_\dC(P)y=P^{-1}y\to \ddel_\dC(P)x=P^{-1}x, \]
        from \cref{def: f on fibers},
        \item a vertical morphism $u\colon x\bulletarrow x'$ in $\dC$ to the profunctor
        \[ \ddel_\dC(P)u\colon
            \ddel_\dC(P)x^{\op}
            \times
            \ddel_\dC(P)x'=P^{-1}x^{\op}\times P^{-1}x'
            \to
            \Set\]
        given by the image under the equivalence $\fib\colon \TSdiscFib(P^{-1}x,P^{-1}x')\xrightarrow{\simeq}\Prof(P^{-1}x, P^{-1}x')$ from \cref{thm: TSdisc vs Prof} of the two-sided discrete fibration from 
        \cref{prop: P_u is TSdiscFib}
        \[ P^{-1}u\to P^{-1}x\times P^{-1}x';  \]
        using \cref{Prop: fiber at vertical id}, we impose $\ddel_\dC(P)e_x\coloneqq e_{P^{-1}x}$, for every object $x$ in $\dC$,
        \item a square $\alpha\colon\edgesquare{u}{f}{f'}{v}\colon\nodesquare{x}{x'}{y}{y'}$ in $\dC$ to the natural transformation
        \[
        \ddel_\dC(P)\alpha\colon\ddel_\dC(P)v\Rightarrow\ddel_\dC(P)u\circ(\ddel_\dC(P)f^{\op}\times\ddel_\dC(P)f')
        \]
        which is induced as in \cref{cor: nat.trf. for TS square} by the following commutative square
        \[
        \begin{tikzcd}
            P^{-1}v \arrow[d] \arrow[r, "\alpha^*"]               & [10pt] P^{-1}u \arrow[d]      \\
            P^{-1}y\times P^{-1}y' \arrow[r, "{(f^*,{f'}^*)}"'] & P^{-1}x\times P^{-1}x'
        \end{tikzcd}
        \]
        where $\alpha^*$ is the functor from \cref{def: alpha on fibers}; note that the above diagram indeed commutes by construction of $f^*, f'^*,\alpha^*$,
        \item for composable vertical morphisms $x\overset{u}{\bulletarrow}x'\overset{u'}{\bulletarrow}x''$ in $\dC$, the composition comparison square is given by the natural transformation 
        \[ \mu_{u,u'}\colon \ddel_\dC(P)u'\bullet\ddel_\dC(P)u\Rightarrow 
            \ddel_\dC(P)(u'\bullet u) \]
            given by the image under the equivalence $\fib\colon \TSdiscFib(P^{-1}x,P^{-1}x'')\xrightarrow{\simeq}\Prof(P^{-1}x, P^{-1}x'')$ from \cref{thm: TSdisc vs Prof} of the morphism of two-sided discrete fibrations 
            \[ 
            \begin{tikzcd}
            P^{-1}u\times_{P^{-1}x'} P^{-1}u'_{/\sim}\arrow[rd] \arrow[rr,"{-\bullet-}"] &    & P^{-1}(u'\bullet u) \arrow[ld] \\
                 & P^{-1}x\times P^{-1}x'' & 
            \end{tikzcd}
            \]
            induced by vertical composition in $\dE$.
    \end{itemize}
\end{Construction}

\begin{Lemma}
    The construction $\ddel_\dC(P)\colon \dC^{\op}\to \dCat$ is a lax double presheaf.
\end{Lemma}

\begin{proof}
    First note that $\ddel_\dC(P)$ preserves horizontal compositions and identities for horizontal morphisms and squares by unicity of lifts. 
    
    We first show that the composition comparison natural transformations $\mu_{u,u'}$ are natural in $(u,u')$. Given squares $\alpha\colon\edgesquare{u}{f}{f'}{v}$ and $\alpha'\colon\edgesquare{u'}{f'}{f''}{v'}$ in $\dC$, we have to show that
    the following diagram in $\Prof(P^{-1}x,P^{-1}x'')$ commutes.
     \[
    \begin{tikzcd}
        \ddel_{\dC}(P)v'\bullet\ddel_\dC(P)v \arrow[r, "\ddel_\dC(P)\alpha'\bullet\ddel_\dC(P)\alpha", Rightarrow] \arrow[d, "{\mu_{v,v'}}"', Rightarrow] 
        &[40pt]
        (\ddel_\dC(P)u\bullet\ddel_\dC(P)u)({f^*}^{\op}\times f''^*) \arrow[d, "{\mu_{u,u'}({f^*}^{\op}\times f''^*)}", Rightarrow] \\
        \ddel_\dC(P)(v'\bullet v) \arrow[r, "\ddel_{\dC}(P)(\alpha'\bullet\alpha)"', Rightarrow]                                   &[40pt]
        \ddel_\dC(P)(u'\bullet u)({f^*}^{\op}\times f''^*)                
        \end{tikzcd}
        \]
        Under the natural equivalence $\fib$ from \cref{thm: TSdisc vs Prof,prop: naturality of fib} and using \cref{prop: fib compatible with composition}, this amounts to showing that the following diagram of functors commutes.
        \[
        \begin{tikzcd}
        P^{-1}v\times_{P^{-1}y'} P^{-1}v'_{/\sim}\arrow[d, "{-\bullet-}"'] \arrow[r, "\alpha^*\times \alpha'^*"] & [20pt] P^{-1}u\times_{P^{-1}x'} P^{-1}u'_{/\sim} \arrow[d, "{-\bullet-}"] \\
            P^{-1}(v'\bullet v) \arrow[r, "(\alpha'\bullet\alpha)^*"']                          & P^{-1}(u'\bullet u)                        
        \end{tikzcd}
        \]
        But this follows directly from the unicity of lifts.

        Next, we show that the composition comparison natural transformations $\mu_{u,u'}$ are compatible with vertical composition. Given three composable vertical morphisms $x\overset{u}{\bulletarrow}x'\overset{u'}{\bulletarrow}x''\overset{u''}{\bulletarrow} x'''$ in $\dC$, we have to show that the following diagram in $\TSdiscFib(P^{-1}x,P^{-1}x''')$ commutes.
        \[
    \begin{tikzcd}
        \ddel_{\dC}(P)u''\bullet\ddel_\dC(P)u'\bullet\ddel_\dC(P)u\arrow[r, "\mu_{u',u''}\bullet \ddel_\dC(P)u", Rightarrow] \arrow[d, "{\ddel_{\dC}(P)u''\bullet\mu_{u,u'}}"', Rightarrow] 
        &[40pt]
        \ddel_{\dC}(P)(u''\bullet u')\bullet\ddel_\dC(P)u \arrow[d, "\mu_{u,u''\bullet u'}", Rightarrow] \\
        \ddel_{\dC}(P)u''\bullet\ddel_\dC(P)(u'\bullet u) \arrow[r, "\mu_{u'\bullet u,u''}"', Rightarrow]                                   &[40pt]
        \ddel_\dC(P)(u''\bullet u'\bullet u)({f^*}^{\op}\times f''^*)                
        \end{tikzcd}
        \]
        Under the equivalence $\fib$ from \cref{thm: TSdisc vs Prof}, this amounts to showing that the following diagram in $\TSdiscFib(P^{-1}x,P^{-1}x''')$ commutes.
        \[
        \begin{tikzcd}
        {P^{-1}u\times_{P^{-1}x'} P^{-1}u'\times_{P^{-1}x''} P^{-1}u''_{/\sim}} \arrow[d, "{(-\bullet-)\times 1_{P^{-1}u''}}"'] \arrow[r, "{1_{P^{-1}u}\times (-\bullet-)}"] & [35pt] {P^{-1}u\times_{P^{-1}x'} P^{-1}(u''\bullet u')_{/\sim}} \arrow[d, "{-\bullet-}"] \\
            {P^{-1}(u'\bullet u)\times_{P^{-1}x''} P^{-1}u''_{/\sim}}\arrow[r, "{-\bullet-}"']                          & P^{-1}(u''\bullet u'\bullet u)    
        \end{tikzcd}
        \]
        But this follows directly from the fact that vertical composition in $\dE$ is strictly associative.

        Next, we show that $\ddel_\dC(P)$ preserves vertical identities. By construction, it preserves the vertical identity at an object. Given a horizontal morphism $f\colon x\to y$ in $\dC$, we have to show that the natural transformation $\ddel_\dC(P)e_f$ is the identity at $\ddel_\dC(P)f=f^*$. By definition, we have that $\ddel_\dC(P)e_f$ is the natural transformation induced as in \cref{cor: nat.trf. for TS square} by the below left commutative square. 
        \[
        \begin{tikzcd}
           {P^{-1}e_y} \arrow[r, "{e_f^*}"] \arrow[d, "{P_{e_y}}"']        & {P^{-1}e_x} \arrow[d, "{P_{e_x}}"]   \\
            P^{-1}y\times P^{-1}y \arrow[r, "{(f^*,f^*)}"'] & P^{-1}x\times P^{-1}x
        \end{tikzcd} 
        \quad \quad
        \begin{tikzcd}
            {(P^{-1}y)^{[1]}} \arrow[d, "{(s,t)}"'] \arrow[r, "{(f^*)^{[1]}}"] & {(P^{-1}x)^{[1]}} \arrow[d, "{(s,t)}"] \\
            P^{-1}y\times P^{-1}y \arrow[r, "{(f^*,f^*)}"']                    & P^{-1}x\times P^{-1}x     
        \end{tikzcd}
        \]
        Since the above left commutative square is canonically isomorphic to the above right commutative square by \cref{Prop: fiber at vertical id}, we see that the desired natural transformation is the identity at $f^*$.
        
        Finally, we show that the composition comparison natural transformations $\mu_{u,e_{x'}}$ and $\mu_{u,e_x}$ agree with the identity natural transformation at $\ddel_\dC(P)u$. Given a vertical morphism $u\colon x\bulletarrow x'$ in $\dC$, by \cref{Prop: fiber at vertical id,prop: composition of TS is unital}, we have canonical isomorphisms over $P^{-1}x\times P^{-1}x'$
        \[
        P^{-1}e_x\times_{P^{-1}x} P^{-1}u_{/\sim}\cong (P^{-1}x)^{[1]}\times_{P^{-1}x} P^{-1}u_{/\sim}\cong P^{-1}u
        \]
        and
        \[
        P^{-1}u\times_{P^{-1}x'} {P^{-1}e_{x'}}_{/\sim}\cong
        P^{-1}u\times_{P^{-1}x'}{(P^{-1}x')^{[1]}}_{/\sim}\cong P^{-1}u.
        \]
        Note that these isomorphisms are induced by vertical composition, therefore $\mu_{e_x,u}=1_{\ddel_\dC(P)u}=\mu_{u,e_{x'}}$, concluding the proof.
\end{proof}

We now turn to the assignment of $\ddel_\dC$ on morphisms.

\begin{Construction}
\label{constr: ddel on 1-morphisms}
    Given a morphism of discrete double fibrations
    \[
    \begin{tikzcd}
        \dE\arrow[rr,"F"]\arrow[rd,swap,"P"] &&
        \dF\arrow[ld,"Q"]
        \\
        & \dC &
    \end{tikzcd}
    \]
    we construct a horizontal transformation 
    \[ \ddel_\dC(F)\colon\ddel_\dC(P)\Rightarrow\ddel_\dC(Q) \]
    whose 
    \begin{itemize}[leftmargin=1cm]
        \item component at an object $x$ is given by the unique induced functor between fibers
        \[ \ddel_\dC(F)_x\coloneqq F_x\colon \ddel_\dC(P)x=P^{-1}x\to \ddel_\dC(Q)x=Q^{-1}x \]
        \item component at a vertical morphism $u\colon x\bulletarrow x'$ is given by the natural transformation 
        \[ \ddel_\dC(F)_u\colon \ddel_\dC(P)u\Rightarrow\ddel_\dC(Q)u(\ddel_\dC(F)_x^{\op}\times \ddel_\dC(F)_{x'}) \]
        which is induced as in 
        \cref{cor: nat.trf. for TS square}
         by the following commutative square
        \[
        \begin{tikzcd}
            P^{-1}u \arrow[d] \arrow[r, "{F_u\coloneqq \dHom{\dV[1],F}_u}"]                   & Q^{-1}u \arrow[d]      \\
        P^{-1}x\times P^{-1}x' \arrow[r,swap, "F_x\times F_{x'}"] & Q^{-1}x\times Q^{-1}x'
        \end{tikzcd}
        \]
        where $\dHom{\dV[1],F}_u$ denotes the unique induced functor between fibers.
    \end{itemize}
\end{Construction}

\begin{Lemma}
    The construction $\ddel_\dC(F)\colon\ddel_\dC(P)\Rightarrow\ddel_\dC(Q)$ is a horizontal transformation.
\end{Lemma}

\begin{proof}
    We first show that the components $\ddel_\dC(F)_x$ are natural in $x$. Given a horizontal morphism $f\colon x\to y$ in $\dC$, we have to show that the following diagram of functors commutes.
    \[
    \begin{tikzcd}
        P^{-1}y \arrow[d, "f^*"'] \arrow[r, "F_y"] & Q^{-1}y \arrow[d, "f^*"] \\
        P^{-1}x \arrow[r, "F_x"']                  & Q^{-1}x                 
    \end{tikzcd}
    \]
    But this follows directly from the fact that every morphism of discrete double fibrations over $\dC$ preserves unique lifts.

    Next, we show that the components $\ddel_\dC(F)_u$ are natural in $u$. Given a square $\alpha\colon\edgesquare{u}{f}{f'}{v}\colon\nodesquare{x}{x'}{y}{y'}$ in $\dC$, we have to show that the following diagram in 
    $\Prof(P^{-1}y,P^{-1}y')$ commutes.
    \[
    \begin{tikzcd}
        \ddel_\dC(P)v \arrow[r, "\ddel_\dC(F)_v", Rightarrow] \arrow[d, "\ddel_\dC(P)\alpha"', Rightarrow]         
        &[40pt]
        \ddel_\dC(Q)v(F_y^{\op}\times F_{y'}) \arrow[d, "\ddel_\dC(Q)\alpha(F_y^{\op}\times F_{y'})", Rightarrow] \\
        \ddel_\dC(P)u({f^*}^{\op}\times f'^*) \arrow[r, "\ddel_\dC(F)_u({f^*}^{\op}\times f'^*)"' {yshift=-2pt}, Rightarrow] 
        &[40pt]
        \ddel_\dC(Q)u((F_x{f^*})^{\op}\times F_{x'}f')   
    \end{tikzcd}
    \]
    Under the natural equivalence $\fib$ from \cref{thm: TSdisc vs Prof,prop: naturality of fib}, this amounts to showing that the following diagram of functors commutes.
    \[
    \begin{tikzcd}
        P^{-1}v \arrow[d, "\alpha^*"'] \arrow[r, "F_v"] & Q^{-1}v \arrow[d, "\alpha^*"] \\
        P^{-1}u \arrow[r, "F_u"']                       & Q^{-1}u                      
    \end{tikzcd}
    \]
    But this follows directly from the fact that every morphism of discrete double fibrations over $\dHom{\dV[1],\dC}$ preserves unique lifts.
    
    We now show that the components $\ddel_\dC(F)_u$ are compatible with composition comparison natural transformations. Given composable vertical morphisms $x\overset{u}{\bulletarrow}x'\overset{u'}{\bulletarrow}x''$ in $\dC$, we have to show that the following diagram in 
    $\Prof(P^{-1}x,P^{-1}x'')$ commutes.
    \[
    \begin{tikzcd}
        \ddel_\dC(P)u'\bullet\ddel_\dC(P)u \arrow[r,"\ddel_\dC(F)_{u'}\bullet\ddel_\dC(F)_u", Rightarrow] \arrow[d,"{\mu_{u,u'}}"', Rightarrow] 
        &[40pt]
        (\ddel_\dC(Q)u'\bullet\ddel_\dC(Q)u)(F_x^{\op}\times F_{x''}) \arrow[d, "{\mu_{u,u'}(F_x^{\op}\times F_{x''})}", Rightarrow] \\
        \ddel_\dC(P)(u'\bullet u) \arrow[r, "\ddel_\dC(F)_{u'\bullet u}"', Rightarrow]    &[40pt]
        \ddel_\dC(Q)(u'\bullet u)(F_x^{\op}\times F_{x''})     
    \end{tikzcd}
    \]
    Under the equivalence $\fib$ from \cref{thm: TSdisc vs Prof}, this amounts to showing that the following diagram of functors commutes. 
    \[
    \begin{tikzcd}
        P^{-1}u\times_{P^{-1}x'} P^{-1}u'_{/\sim} \arrow[d, "{-\bullet-}"'] \arrow[r, "F_u\times F_{u'}"] & [10pt] Q^{-1}u\times_{Q^{-1}x'} Q^{-1}u'_{/\sim} \arrow[d, "{-\bullet-}"] \\
        P^{-1}(u'\bullet u) \arrow[r, "F_{u'\bullet u}"']  & Q^{-1}(u'\bullet u) 
    \end{tikzcd}
    \]
    But this follows directly from the fact that $F$ preserves vertical composition. 

    Finally, we show that the components $\ddel_\dC(F)_u$ are compatible with vertical identities. Given an object $x$ in $\dC$, we have to show that the natural transformation 
    $\ddel_\dC(F)_{e_x}$ is the identity at 
    ${\ddel_\dC(F)}_x=F_x$. By definition, we have that $\ddel_\dC(F)_{e_x}$ is the natural transformation induced as in \cref{cor: nat.trf. for TS square} by the below left commutative square. 
    \[
    \begin{tikzcd}
        P^{-1}e_x \arrow[d, "{P_{e_x}}"'] \arrow[r, "{\dHom{\dV[1],F}_{e_x}}"] & Q^{-1}e_x \arrow[d, "{Q_{e_x}}"] \\
        P^{-1}x\times P^{-1}x \arrow[r, swap,"{F_x\times F_x}"]                         & Q^{-1}x\times Q^{-1}x           
    \end{tikzcd}
    \quad \quad
    \begin{tikzcd}
        {(P^{-1}x)^{[1]}} \arrow[d, "{(s,t)}"'] \arrow[r, "{F_x^{[1]}}"] & {(Q^{-1}x)^{[1]}} \arrow[d, "{(s,t)}"] \\
        P^{-1}x\times P^{-1}x \arrow[r, "{F_x\times F_x}"']                     & Q^{-1}x\times Q^{-1}x                   
    \end{tikzcd}
    \]
    Since the above left commutative square is canonically isomorphic to the above right commutative square by \cref{Prop: fiber at vertical id}, we see that the desired natural transformation is the identity at $F_x$, concluding the proof.
\end{proof}

Finally, we turn to the assignment of $\ddel_\dC$ on $2$-morphisms. 

\begin{Construction}   
\label{constr: ddel on 2-morphisms}
    Given a vertical transformation of discrete double fibrations
    \[
    \begin{tikzcd}
        &{}\arrow[dd,shorten <=21pt, shorten >= 20pt,"\bullet" marking,Rightarrow,"A"]&
        \\
        \dE
        \arrow[rr,bend left =20,"F"]
        \arrow[rr,swap,bend right = 20,"F'"]
        \arrow[rd,swap,"P"] &&
        \dF\arrow[ld,"Q"]
        \\
        & \dC &
    \end{tikzcd}
    \]
    we construct a globular modification
    \[
    \begin{tikzcd}
        \ddel_\dC(P)
        \arrow[r,"\ddel_\dC(F)"]
        \arrow[d,swap,"\bullet" marking,equal]
        \arrow[rd,phantom,"\ddel_\dC(A)"]
        & [10pt]
        \ddel_\dC(Q)
        \arrow[d,"\bullet" marking,equal]
        \\
        \ddel_\dC(P)
        \arrow[r,swap,"\ddel_\dC(F')"]
        &
        \ddel_\dC(Q)
    \end{tikzcd}
    \]
    defined as follows. Its component at an object $x$ in $\dC$ is given by the natural transformation 
    \[ 
    \ddel_\dC(A)_x\colon \ddel_\dC(F)_x=F_x\Rightarrow \ddel_\dC(F')_x=F'_x\colon P^{-1}x\to Q^{-1}x
    \] 
    whose component at an object $x_-$ of $P^{-1}x$ is given by the morphism in $Q^{-1}x$ \[ A_{x_-}\colon Fx_-\bulletarrow F'x_-. \]
\end{Construction}

\begin{Lemma}
    The construction $\ddel_\dC(A)\colon \edgesquare{e_{\ddel_\dC(P)}}{\ddel_\dC(F)}{\ddel_\dC(F')}{e_{\ddel_\dC(Q)}}$ is a globular modification.
\end{Lemma}

\begin{proof}
    We first show horizontal compatibility of $\ddel_\dC(A)_x$. Given a horizontal morphism $f\colon x\to y$ in $\dC$, we have to show that the following pasting diagram of functors and natural transformations commutes.
    \[
    \begin{tikzcd}
        & P^{-1}y \arrow[rr, "f^*"] \arrow[dd, "F_y"', bend right = 49] \arrow[dd, "F'_y", bend left = 49] && P^{-1}x \arrow[dd, "F_x"' {yshift=-8pt}, bend right = 49] \arrow[dd, "F'_x" {yshift=-8pt}, bend left = 49] &    
        \\
        {} \arrow[rr,shorten <=30pt,shorten >=30pt, "\ddel_\dC(A)_y", Rightarrow, yshift=-9pt] & & {} \arrow[rr,shorten <=30pt,shorten >=30pt, Rightarrow, yshift=-9pt,"\ddel_\dC(A)_x"] &   & {} 
        \\
        & Q^{-1}y \arrow[rr, "f^*"']& & Q^{-1}x &   
        \end{tikzcd}
    \]
    To see this, we first note the following. Since $A\colon F\Bulletarrow F'$ is a vertical transformation over $\dC$, given an object $y_-$ in $P^{-1}y$, the component of $A$ at the horizontal morphism $P^*f\colon f^*y_-\to y_-$ is a square $A_{P^*f}\colon \edgesquare{A_{f^*y_-}}{FP^*f}{F'P^*f}{A_{y_-}}$ such that $QA_{P^*f}=e_f$. Hence it has to be the unique lift of $e_f$ along $Q$ with target $A_{y_-}$. When evaluating both whiskerings at an object $y_-$ in $P^{-1}y$, we therefore have:
    \begin{align*}
        e_f^*(\ddel_\dC(A)_y)_{y_-}
        &=
        e_f^* A_{y_-}
        &
        \text{Definition of }
        {\ddel_\dC(A)}_y
        \\
        &=
        A_{f^*y_-}
        &
        \text{Unicity of lifts}
        \\
        &=
        (\ddel_\dC(A)_x)_{f^*y_-}.
        &
        \text{Definition of }
        {\ddel_\dC(A)}_x
    \end{align*}
    This shows that the components of the two whiskerings agree, as desired.
    
    Next, we show vertical compatibility of $\ddel_\dC(A)_x$. Given a vertical morphism $u\colon x\bulletarrow x'$ in $\dC$, we have to show that the following diagram 
    in $\Prof(P^{-1},P^{-1}x')$ commutes.
    \[
    \begin{tikzcd}
        \ddel_\dC(P)u \arrow[r, "\ddel_\dC(F)_u", Rightarrow] \arrow[d, "\ddel_\dC(F')_u"', Rightarrow] 
        &[45pt]
        \ddel_\dC(Q)u(F_x^{\op}\times F_{x'}) \arrow[d, "\ddel_\dC(Q)u(F_x^{\op}\times A_{x'})", Rightarrow] \\
        \ddel_\dC(Q)u({F'_x}^{\op}\times F'_{x'}) \arrow[r, "\ddel_\dC(Q)u(A_x^{\op}\times F'_{x'})"' {yshift=-2pt}, Rightarrow] 
        &[45pt]
        \ddel_\dC(Q)u(F_x^{\op}\times F'_{x'})                                                           
    \end{tikzcd}
    \]
    Under the natural equivalence $\fib$ from \cref{thm: TSdisc vs Prof,prop: naturality of fib}, this amounts to showing that the following diagram of functors and natural transformations commutes.

    \begin{equation} \label{vert comp}
    \begin{tikzcd}
        & \ P^{-1}u\ \   \arrow[dd, "F_u"' {xshift=-3pt}, bend right=75] \arrow[dd, "F'_u" {xshift=3pt}, bend left=75] \arrow[rr]
        &[30pt]
        & P^{-1}x\times P^{-1}x' \arrow[dd, "{F_x\times F_{x'}}"' {yshift=-9pt}, bend right=70] \arrow[dd, "{F'_x\times F'_{x'}}" {yshift=-9pt}, bend left=70] &    
        \\
        {} \arrow[rr, "{(\ddel_\dC(A)_x,\ddel_\dC(A)_{x'})}" {xshift=-14pt},shorten <=25pt, shorten >=55pt, Rightarrow,yshift=-9pt] 
        &&[30pt]
        {} \arrow[rr, "{\ddel_\dC(A)_x\times\ddel_\dC(A)_{x'}}",shorten <=45pt, shorten >=45pt, Rightarrow,yshift=-9pt] &  & {} 
        \\
        & \ Q^{-1}u\ \   \arrow[rr]
        &[30pt]
        & Q^{-1}x\times Q^{-1}x' &   
    \end{tikzcd}
    \end{equation}
    Here, the left-hand transformation is defined as follows: using the description of the morphisms in $Q^{-1}u$ from \cref{descr: fiber at u}, there is a natural transformation $(\ddel_\dC(A)_x,\ddel_\dC(A)_{x'})\colon F_u\Rightarrow F'_u$ whose component at an object $u_-\colon x_-\bulletarrow x'_-$ in $P^{-1}u$ is given by the commutative diagram in $\dF$
    \[ \begin{tikzcd}
        Fx_-
        \arrow[d,"Fu_-"',"\bullet" marking]
        \arrow[r,"A_{x_-}","\bullet" marking]
        &
        F'x_-
        \arrow[d,"F'u_-","\bullet" marking]
        \\
        Fx'_-
        \arrow[r,"A_{x'_-}"',"\bullet" marking]
        &
        F'x'_-.
    \end{tikzcd} \]
    Note that this gives a well-defined morphism of $Q^{-1}u$ since the above diagram of vertical morphisms commutes by naturality of $A_{x_-}$ in $x_-$, and $QA_{x_-}=e_x$ and $QA_{x'_-}=e_{x'}$ by definition of the vertical transformation $A$ living over $\dC$. The fact that the diagram \eqref{vert comp} commutes is then straightforward from the fact that the horizontal arrows are picking source and target as described in \cref{prop: P_u is TSdiscFib}. 
\end{proof}

Putting everything together, we get the following. 

\begin{Construction}
\label{constr: ddel}
    Given a double category $\dC$, we define a $2$-functor 
    \[ \ddel_\dC\colon \dFib(\dC)\to \P{C} \]
    which sends
    \begin{itemize}[leftmargin=1cm]
        \item a discrete double fibration $P\colon \dE\to \dC$ to the lax double presheaf $\ddel_\dC(P)\colon \dC^{\op}\to \dCat$ from \cref{constr: ddel on objects}, 
        \item a double functor $F\colon \dE\to \dF$ over $\dC$ to the horizontal transformation $\ddel_\dC(F)\colon \ddel_\dC(P)\Rightarrow \ddel_\dC(Q)$ from \cref{constr: ddel on 1-morphisms}, 
        \item a vertical transformation $A\colon F\Bulletarrow F'$ over $\dC$ to the globular modification
        \[ \ddel_\dC(A)\colon \edgesquare{e_{\ddel_\dC(P)}}{\ddel_\dC(F)}{\ddel_\dC(F')}{e_{\ddel_\dC(Q)}} \]
        from \cref{constr: ddel on 2-morphisms}.
    \end{itemize}
    It is straightforward to check that this construction is $2$-functorial.
\end{Construction}

We are now ready to show that $\ddel_\dC$ is a weak inverse of $\dgro_\dC$. We start by constructing a $2$-natural isomorphism $\epsilon\colon \ddel_\dC\dgro_\dC\cong \id_{\P{C}}$. For this, we first describe the composite $\ddel_\dC\dgro_\dC$ on objects, morphisms, and $2$-morphisms explicitly.

\begin{Lemma}
\label{lem: ddel dgro(X) explicit}
    Given a lax double presheaf $X\colon\dC^{\op}\to\dCat$, the lax double presheaf 
    \[ \ddel_\dC(\pi_X)\colon \dC^{\op}\to \dCat \]
    associated with the Grothendieck construction $\pi_X\colon \dgro_\dC X\to \dC$ admits the following description.
    \begin{itemize}[leftmargin=1cm]
        \item Given an object $x$ in $\dC$, the category $\ddel_\dC(\pi_X)x=\pi_X^{-1}x$ is given by $\{x\}\times Xx$. 
        \item Given a horizontal morphism $f\colon x\to y$ in $\dC$, the functor $\ddel_\dC(\pi_X)f\colon \pi_X^{-1} y\to \pi_X^{-1} x$ is given by 
        \[ [y\mapsto x]\times Xf\colon \{y\}\times Xy\to \{x\}\times Xx. \]
        \item Given a vertical morphism 
        $u\colon x\bulletarrow x'$ in $\dC$, the profunctor $\ddel_\dC(\pi_X)u\colon \pi_X^{-1}x^{\op}\times\pi_X^{-1}x'\to \Set$ is given by 
        \[ \{u\}\times Xu\colon \{x\}\times Xx^{\op}\times \{x'\}\times Xx'\cong Xx^{\op}\times Xx'\to\Set. \]
        \item Given a square $\alpha\colon\edgesquare{u}{f}{f'}{v}\colon\nodesquare{x}{x'}{y}{y'}$ in $\dC$, the natural transformation 
        \[ \textstyle\ddel_\dC(\pi_X)\alpha \colon \ddel_\dC(\pi_X)v \Rightarrow \ddel_\dC(\pi_X)u(\ddel_\dC(\pi_X)f^{\op}\times \ddel_\dC(\pi_X)f') \]
        is given by
        \[ [v\mapsto u]\times X\alpha\colon\{v\}\times Xv\Rightarrow \{u\}\times Xu(Xf^{\op}\times Xf'). \]
        \item Given composable vertical morphisms $x\overset{u}{\bulletarrow} x'\overset{u'}{\bulletarrow} x''$ in $\dC$, the compositor comparison natural transformation $\ddel_\dC(\pi_X)u'\bullet \ddel_\dC(\pi_X)u\Rightarrow\ddel_\dC(\pi_X)(u'\bullet u)$ is given by 
        \[ [(u,u')\mapsto u'\bullet u]\times \mu_{u,u'}\colon (\{u\}\times Xu)\bullet (\{u'\}\times Xu')\cong \{(u,u')\}\times (Xu\bullet Xu')\to \{u'\bullet u\}\times X(u'\bullet u)\]
        where $\mu_{u,u'}$ denotes the composition comparison natural transformation of $X$.
    \end{itemize}
\end{Lemma}

\begin{proof}
    By unpacking \cref{def: dgro,constr: ddel on objects}, we observe the following. 

    Given an object $x$ in $\dC$, then $\pi_X^{-1}x$ is the category whose 
        \begin{itemize}[leftmargin=1cm]
            \item objects are pairs $(x,x_-)$ with $x_-$ an object in $Xx$, 
            \item morphisms $(x,x_-)\to (x,x'_-)$ are pairs $(e_x,s_-)$ with $s_-$ an element of $Xe_x(x_-,x_-')=Xx(x_-,x_-')$, i.e., a morphism $s_-\colon x_-\to x_-'$ in $Xx$.
        \end{itemize}
        Hence, we have that $\pi_X^{-1}x=\{x\}\times Xx$, as desired. 
     
     Next, given a horizontal morphism $f\colon x\to y$ in $\dC$, the functor $\ddel_{\dC}(\pi_X)f$ is given by assigning the source of the unique lift  of $f$ or $e_f$ along the discrete double fibration $\pi_X\colon \dgro_\dC X\to \dC$. Hence, by the proof of 
        \cref{prop: pi_x is ddisc fib}, we get the desired description, namely $\ddel_{\dC}(\pi_X)f=[y\mapsto x]\times Xf$.

    Now, given a vertical morphism $u\colon x\bulletarrow x'$ in $\dC$, the profunctor $\ddel_\dC(\pi_X)u$ is the image under the equivalence $\fib$ from
        \cref{thm: TSdisc vs Prof} of the two-sided discrete fibrations 
        \[ \pi_X^{-1}u\to \pi_X^{-1}x\times \pi_X^{-1}x'. \]
        Here $\pi_X^{-1}u$ is the category whose
        \begin{itemize}[leftmargin=1cm]
            \item objects are pairs $(u,u_-)$ with $u_-$ an element of $Xu(x_-,x_-')$, for some object $(x_-,x_-')$ in $Xx\times Xx'$,
            \item morphisms $(u,u_-)\to (u,\hat{u}_-)$ with $u_-\in Xu(x_-,x_-')$ and $\hat{u}_-\in Xu(\hat{x}_-,\hat{x}_-')$ are tuples $(u,s_-,s_-')$ with $(x_-\xrightarrow{s_-} \hat{x}_-, x_-'\xrightarrow{s_-'} \hat{x}_-')$ a morphism in $Xx\times Xx'$ making the following square in $\dgro_\dC X$ commutes. 
            \[ \begin{tikzcd}
        (x,x_-)
        \arrow[d,"{(u,s_-)}"',"\bullet" marking]
        \arrow[r,"{(e_x,s_-)}","\bullet" marking]
        &
        (x,\hat{x}_-)
        \arrow[d,"{(u,\hat{u}_-)}","\bullet" marking]
        \\
        (x',x'_-)
        \arrow[r,"{(e_{x'},s'_-)}"',"\bullet" marking]
        &
        (x',\hat{x}'_-).
    \end{tikzcd} \]
        \end{itemize} 
        Hence, by applying the functor $\fib$ from \cref{constr: functor fib}, we see that the profunctor $\ddel_\dC(\pi_X)u$ assigns to an object $(x_-,x'_-)$ in $\pi_X^{-1}x\times \pi_X^{-1}x'$ the set $\{u\}\times Xu(x_-,x_-')$. It remains to show that the action on morphisms of $\ddel_\dC(\pi_X)u$ and $\{u\}\times Xu$ agree. 
        
        By the proof of 
        \cref{prop: P_u is TSdiscFib}, the action of a morphism $(x_-\xrightarrow{s_-} \hat{x}_-, x'_-\xrightarrow{s_-'} \hat{x}_-')$ in $\pi_X^{-1}x\times \pi_X^{-1}x'$ on fibers is given by taking an element $(u,u_-)$ in $\{u\}\times Xu(\hat{x}_-,x'_-)$ to the composite $(e_x,s')\bullet (u,u_-) \bullet (e_{x'},s_-)=(u,s'\bullet u_-\bullet s_-)$ of $\dgro_\dC X$. However, we have: 
        \begin{align*}
            s_-'\bullet u_-\bullet s_- &= (\mu_{u, e_{x'}})_{\hat{x}_-,\hat{x}'_-} ([s_-',u_-\bullet s_-]) & \text{Definition of } \bullet \\
            & = (\mu_{u, e_{x'}})_{x_-,\hat{x}'_-} ([s_-',(\mu_{e_x,u})_{x_-,x'_-}([u_-,s_-])]) & \text{Definition of } \bullet \\
            & = (\mu_{u, e_{x'}})_{x_-,\hat{x}'_-} ([s_-',Xu(s_-,x'_-)(u_-)]) & \mu_{e_x,u}=1_{Xu} \\
            & = Xu(x_-,s_-')(Xu(s_-,x_-')(u_-)) & \mu_{u,e_{x'}}=1_{Xu} \\
            & = Xu(s_-,s_-')(u_-). & \text{Functorality of } Xu
        \end{align*}
        Hence, this shows that $\ddel_\dC(\pi_X)u=\{u\}\times Xu$, as desired.
        
        Next, given a square 
        $\alpha\colon\edgesquare{u}{f}{f'}{v}\colon\nodesquare{x}{x'}{y}{y'}$ in $\dC$, the natural transformation 
        $\ddel_\dC(\pi_X)\alpha$ is induced as in \cref{cor: nat.trf. for TS square} by the following commutative square
        \[
        \begin{tikzcd}
            \pi_X^{-1}v \arrow[r, "\alpha^*"] \arrow[d]                & \pi^{-1}_Xu \arrow[d]       \\
            \pi^{-1}_Xy\times\pi^{-1}_Xy' \arrow[r, "f^*\times f'^*"'] & \pi^{-1}_Xx\times\pi^{-1}x'
        \end{tikzcd}
        \]
        where $\alpha^*$ is given by assigning the source of the unique lift of $\alpha$ along the discrete double fibration $\pi_X\colon\dgro_\dC X\to \dC$. Hence, by the proof of 
        \cref{prop: pi_x is ddisc fib}, we get the desired description, namely $\ddel_{\dC}(\pi_X)\alpha=[u\mapsto v]\times X\alpha$.
    
        Finally, given composable vertical morphisms 
        $x\overset{u}{\bulletarrow}x'\overset{u'}{\bulletarrow}x''$ in $\dC$, the composition comparison 
        natural transformation of $\ddel_\dC(\pi_X)$ is the image under the equivalence $\fib$ from
        \cref{thm: TSdisc vs Prof} of the morphism of two-sided discrete fibrations
        \[
        \begin{tikzcd}
            \pi^{-1}_Xu\times_{\pi^{-1}_Xx'}\pi^{-1}_Xu'_{/\sim} \arrow[rd] \arrow[rr, "{-\bullet-}"] & & \pi^{-1}_X(u'\bullet u) \arrow[ld] \\
            & \pi^{-1}_Xx\times\pi^{-1}_Xx''. & 
        \end{tikzcd}
        \]
        induced by vertical composition in $\dgro_\dC X$. However, by definition, the composite of two vertical morphisms $(u,u_-)$ and $(u',u'_-)$ with $u_-\in Xu(x_-,x'_-)$ and $u'_-\in Xu(x'_-,x''_-)$ in $\dgro_\dC X$ is given by the pair $(u'\bullet u,u'_-\bullet u)$ with $u'_-\bullet u=(\mu_{u,u'})_{x_-,x''_-}([u'_-,u_-])$. Hence, we get the desired description, namely the composition comparison is given by $[(u,u')\mapsto u'\bullet u]\times \mu_{u,u'}$.
\end{proof}

\begin{Lemma}
\label{lem: ddel dgro(F) explicit}
Given a horizontal transformation $F\colon X\Rightarrow Y$ between lax double presheaves ${X,Y\colon \dC^{\op}\to \dCat}$, the horizontal transformation \[ \textstyle \ddel_\dC\dgro_\dC F\colon \ddel_\dC(\pi_X)\Rightarrow \ddel_\dC(\pi_Y) \]
admits the following description.
\begin{itemize}[leftmargin=1cm]
    \item Given an object $x$ in $\dC$, the functor $(\ddel_\dC\dgro_\dC F)_x\colon \ddel_\dC(\pi_X)x\to \ddel_\dC(\pi_Y)x$ is given by 
    \[ \{x\}\times F_x\colon \{x\}\times Xx\to \{x\}\times Yx. \]
    \item Given a vertical morphisms $u\colon x\bulletarrow x'$ in $\dC$, the natural transformation 
    \[ \textstyle(\ddel_\dC\dgro_\dC F)_u\colon \ddel_\dC(\pi_X)u\Rightarrow \ddel_\dC(\pi_Y)u((\ddel_\dC\dgro_\dC F)_x^{\op}\times (\ddel_\dC\dgro_\dC F)_{x'}) \]
    is given by 
    \[ \{u\}\times F_u\colon \{u\}\times Xu\Rightarrow \{u\}\times Yu(F_x^{\op}\times F_{x'}). \]
\end{itemize}
\end{Lemma}

\begin{proof}
    Using \cref{lem: ddel dgro(X) explicit} and unpacking \cref{def: dgro of horizontal transformations,constr: ddel on 1-morphisms}, we observe the following.

    Given an object $x$ in $\dC$, the functor 
        $(\ddel_\dC\dgro_\dC F)_x$ is the unique induced functor between fibers 
        \[ \textstyle(\dgro_\dC F)_x=\{x\}\times F_x\colon \pi_X^{-1}(x)=\{x\}\times Xx\to \pi_Y^{-1}x=\{x\}\times Yx. \]
        Hence, we get the desired description, namely $(\ddel_\dC\dgro_\dC F)_x=\{x\}\times F_x$.
        
        Given a vertical morphism $u\colon x\bulletarrow x'$ in $\dC$, the natural transformation $(\ddel_\dC\dgro_\dC F)_u$ is induced as in \cref{cor: nat.trf. for TS square} by the following commutative square
        \[
        \begin{tikzcd}
        \pi^{-1}_Xu \arrow[d,swap,"{(s_X,t_X)}"] \arrow[r, "(\dgro_\dC F)_u"]          & [10pt] \pi^{-1}_Yu \arrow[d,"{(s_Y,t_Y)}"]         \\
        \pi^{-1}_Xx\times\pi^{-1}_Xx' \arrow[r, "F_x\times F_{x'}"'] & \pi^{-1}_Yx\times\pi^{-1}_Yx',
        \end{tikzcd}
        \]
        where $(\dgro_\dC F)_u$ is the unique induced functor between fibers \[ \textstyle\dHom{\dV[1],\dgro_\dC F}_u=\{u\}\times F_u=\colon \pi_X^{-1}u=\{u\}\times Xu\to \pi_Y^{-1}u=\{u\}\times Yu. \]
        Hence, we get the desired description, namely $(\ddel_\dC\dgro_\dC F)_u=\{u\}\times F_u$.
\end{proof}

\begin{Lemma}
\label{lem: ddel dgro(A) explicit}
Given a globular modification $A\colon \edgesquare{e_X}{F}{F'}{e_Y}$ of lax double presheaves $\dC^{\op}\to \dCat$, the globular modification \[ \textstyle \ddel_\dC\dgro_\dC A\colon \edgesquare{e_{\ddel_\dC(\pi_X)}}{\ddel_\dC\dgro_\dC F}{\ddel_\dC\dgro_\dC F'}{e_{\ddel_\dC(\pi_Y)}} \]
admits the following description. Given an object $x$ in $\dC$, the natural transformation \[ \textstyle(\ddel_\dC\dgro_\dC A)_x\colon (\ddel_\dC\dgro_\dC F)_x\Rightarrow (\ddel_\dC\dgro_\dC F')_x \]
is given by 
    \[ \{x\}\times A_x\colon \{x\}\times Fx\Rightarrow \{x\}\times F'x. \]
\end{Lemma}

\begin{proof}
    Using \cref{lem: ddel dgro(F) explicit} and unpacking \cref{def: dgro for globular modifications,constr: ddel on 2-morphisms}, we observe the following. 
    
     Given an object $x$ in $\dC$, the natural transformation 
    \[
    \textstyle(\ddel_\dC\dgro_\dC A)_x\colon(\ddel_\dC\dgro_\dC F)_x=\{x\}\times F_x\Rightarrow(\ddel_\dC\dgro_\dC F')=\{x\}\times F'_x \]
    is given at an object $(x,x_-)$ in $\pi_X^{-1}x=\{x\}\times Xx$ by the vertical morphism in $\dgro_\dC Y$
    \[ \textstyle(\dgro_\dC A)_{(x,x_-)}=(e_x,(A_x)_{x_-})\colon (x,F_xx_-)\bulletarrow (x,F'_xx_-). \]
    Hence, we get the desired description, namely $(\ddel_\dC\dgro_\dC A)_x=\{x\}\times A_x$. 
\end{proof}

Using the explicit description of the composite $\ddel_\dC\dgro_\dC$, we can now construct the desired $2$-natural isomorphism. 

\begin{Proposition}
\label{prop: epsilon: ddel dgro -> 1}
    There is a $2$-natural isomorphism 
    $\epsilon\colon\ddel_\dC\dgro_\dC\Rightarrow\id_{\P{C}}$.
\end{Proposition}

\begin{proof}
    Given a lax double presheaf $X\colon \dC^{\op}\to \dCat$, we define the component of $\epsilon\colon\ddel_\dC\dgro_\dC\Rightarrow\id_{\P{C}}$ at $X$ to be the invertible horizontal transformation
    \[
    \textstyle\epsilon_X\colon \ddel_\dC\dgro_\dC X\Rightarrow X
    \]
    whose components at an object $x$ in $\dC$ and at a vertical morphism $u\colon x\bulletarrow x'$ in $\dC$ are given by the canonical isomorphisms 
    \[ \{x\}\times Xx \cong Xx \quad \text{and} \quad \{u\}\times Xu\cong Xu. \]
    Using the description from 
    \cref{lem: ddel dgro(X) explicit}, it is straightforward to check that this gives a well-defined horizontal transformation. 

    Moreover, it follows easily from the descriptions in 
    \cref{lem: ddel dgro(F) explicit,lem: ddel dgro(A) explicit} that the components $\epsilon_X$ are $2$-natural in $X$. 
\end{proof}

Next, we want to construct a $2$-natural isomorphism $\eta\colon \id_{\dFib(\dC)}\cong \dgro_\dC\ddel_\dC$. As before, we first describe the composite $\dgro_\dC\ddel_\dC$ on objects, morphisms, and $2$-morphisms. 

\begin{Lemma}
\label{lem: dgro ddel(P)}
Given a discrete double fibration $P\colon \dE\to \dC$, the discrete double fibration \[ \textstyle\pi_{\ddel_\dC(P)}\colon \dgro_\dC\ddel_\dC(P)\to \dC \]
admits the following description. We have that $\dgro_\dC\ddel_\dC(P)$ is the double category whose
\begin{itemize}[leftmargin=1cm]
    \item objects are pairs $(x,x_-)$ of objects $x$ in $\dC$ and $x_-$ in $\dE$ such that $Px_-=x$, 
    \item horizontal morphisms $(x,x_-)\to (y,y_-)$ are horizontal morphisms $f\colon x\to y$ in~$\dC$ such that $f^*y_-=x_-$, i.e., $x_-$ is the source of the unique lift $P^*f\colon x_-\to y_-$ in $\dE$ of $f$, 
    \item vertical morphisms $(x,x_-)\bulletarrow (x',x'_-)$ are pairs $(u,u_-)$ of vertical morphisms $u\colon x\bulletarrow x'$ in~$\dC$ and $u_-\colon x_-\bulletarrow x'_-$ in $\dE$ such that $Pu_-=u$, 
    \item squares $\edgesquare{(u,u_-)}{f}{f'}{(v,v_-)}$ are squares $\alpha\colon\edgesquare{u}{f}{f'}{v}$ in $\dC$ such that $\alpha^*v_-=u_-$, i.e., $u_-$ is the source of the unique lift $P^*\alpha\colon\edgesquare{u_-}{f_-}{f'_-}{v_-}$ in $\dE$ of $\alpha$,  
\end{itemize}
and that $\pi_{\ddel_\dC(P)}\colon \dgro_\dC\ddel_\dC(P)\to \dC$ is given by projecting onto the first variable.
\end{Lemma}

\begin{proof}
    By unpacking \cref{constr: ddel on objects,def: dgro}, we observe the following.
    
    We have that $\dgro_\dC\ddel_\dC(P)$ is the double category whose
    \begin{itemize}[leftmargin=1cm]
        \item objects are pairs 
        $(x,x_-)$ of objects $x$ in $\dC$ and $x_-$ in $\ddel_\dC(P)x=P^{-1}x$, i.e., $x_-$ is an object in $\dE$ such that 
        $Px_-=x$,
        \item horizontal morphisms 
        $(x,x_-)\to (y,y_-)$ are horizontal morphisms $f\colon x\to y$ in $\dC$ such that 
        $x_-=\ddel_\dC(P)f(y_-)=f^*y_-$,
        \item vertical morphisms
        $(x,x_-)\bulletarrow (x',x'_-)$ are pairs $(u,u_-)$ of a vertical morphism $u\colon x\bulletarrow x'$ in $\dC$ and an element
        $u_-$ in $\ddel_\dC(P)u(x_-,x'_-)$, i.e., $u_-$ is an element in the fiber of the two-sided discrete fibration $P^{-1}u\to P^{-1}x\times P^{-1}x'$ at $(x_-,x'_-)$, namely a vertical morphism 
        $u_-\colon x_-\bulletarrow x'_-$ in $\dE$ such that $Pu_-=u$,
        \item squares 
        $\edgesquare{(u,u_-)}{f}{f'}{(v,v_-)}\colon\nodesquare{(x,x_-)}{(x',x'_-)}{(y,y_-)}{(y',y'_-)}$ are squares 
        $\alpha\colon\edgesquare{u}{f}{f'}{v}\colon\nodesquare{x}{x'}{y}{y'}$ in $\dC$ such that
        $u_-=\ddel_\dC(P)\alpha_{y_-,y'_-}(v_-)=\alpha^*v_-$.
    \end{itemize}
    Moreover, the double functor $\pi_{\ddel_\dC(P)}\colon \dgro_\dC\ddel_\dC(P)\to \dC$ is given by projecting onto the first variable. This gives the desired description.
\end{proof}

\begin{Lemma}
\label{lem: dgro ddel(F)}
    Given a double functor $F\colon \dE\to \dF$ between discrete double fibration $P\colon \dE\to \dC$ and $Q\colon \dF\to \dC$, the double functor 
    \[ \textstyle\dgro_\dC\ddel_\dC(F)\colon \dgro_\dC\ddel_\dC(P)\to \dgro_\dC\ddel_\dC(Q) \]
    admits the following description. It sends 
    \begin{itemize}[leftmargin=1cm]
        \item an object $(x,x_-)$ in $\dgro_\dC\ddel_\dC(P)$ to the object $(x,Fx_-)$ in $\dgro_\dC\ddel_\dC(Q)$, 
        \item a horizontal morphism $f\colon (x,x_-)\to (y,y_-)$ in $\dgro_\dC\ddel_\dC(P)$ to the horizontal morphism in $\dgro_\dC\ddel_\dC(Q)$
        \[ f\colon (x,Fx_-)\to (y,Fy_-), \]
        \item a vertical morphism $(u,u_-)\colon (x,x_-)\bulletarrow (x',x'_-)$ in $\dgro_\dC\ddel_\dC(P)$ to the vertical morphism in $\dgro_\dC\ddel_\dC(Q)$ \[ (u,Fu_-)\colon (x,Fx_-)\bulletarrow (x',Fx'_-), \]
        \item a square $\alpha\colon \edgesquare{(u,u_-)}{f}{f'}{(v,v_-)}$ in $\dgro_\dC\ddel_\dC(P)$ to the square in $\dgro_\dC\ddel_\dC(Q)$
        \[ \alpha\colon \edgesquare{(u,Fu_-)}{f}{f'}{(v,Fv_-)}. \]
    \end{itemize}
\end{Lemma}

\begin{proof}
    This is straightforward from unpacking \cref{constr: ddel on 1-morphisms,def: dgro of horizontal transformations}.
\end{proof}

\begin{Lemma}
\label{lem: dgro ddel(A)}
Given a vertical transformation $A\colon F\Bulletarrow F'\colon \dE\to \dF$ between morphisms of discrete double fibrations 
$P\colon \dE\to \dC$ and $Q\colon \dF\to \dC$, the vertical transformation \[ \textstyle\dgro_\dC\ddel_\dC(A)\colon \dgro_\dC\ddel_\dC(F)\Bulletarrow\dgro_\dC\ddel_\dC(F') \]
admits the following description. 
\begin{itemize}[leftmargin=1cm]
    \item Its component at an object $(x,x_-)$ in $\dgro_\dC\ddel_\dC(P)$ is the vertical morphism in $\dgro_\dC\ddel_\dC(Q)$
    \[ (e_x, A_{x_-})\colon (x,Fx_-)\bulletarrow (x,F'x_-), \]
    \item Its component at a horizontal morphism $f\colon (x,x_-)\to (y,y_-)$ in $\dgro_\dC\ddel_\dC(P)$ is the square $e_f\colon\edgesquare{(e_x,A_{x_-})}{f}{f}{(e_y,A_{y_-})}$ in $\dgro_\dC\ddel_\dC(Q)$.
\end{itemize}
\end{Lemma}

\begin{proof}
    This is straightforward from unpacking \cref{constr: ddel on 2-morphisms,def: dgro for globular modifications}.
\end{proof}

We are now ready to construct the desired $2$-natural isomorphism.

\begin{Construction}
\label{constr: eta: 1->dgro ddel}
    We construct a $2$-natural transformation $\eta\colon \id_{\dFib(\dC)}\Rightarrow \dgro_\dC \ddel_\dC$ whose component at a discrete double fibration $P\colon \dE\to \dC$ is given by the functor over $\dC$
    \[ \textstyle\eta_P\colon \dE\to \dgro_\dC \ddel_\dC(P) \]
    sending, using the description of $\dgro_\dC \ddel_\dC(P)$ from 
    \cref{lem: dgro ddel(P)},
    \begin{itemize}[leftmargin=1cm]
        \item an object $x_-$ in $\dE$ to the object $(Px_-,x_-)$ in $\dgro_\dC \ddel_\dC(P)$, 
        \item a horizontal morphism $f_-\colon x_-\to y_-$ in $\dE$ to the horizontal morphism in $\dgro_\dC \ddel_\dC(P)$
        \[ Pf_-\colon (Px_-,x_-)\to (Py_-,y_-), \]
        where $(Pf_-)^*(y_-)=x_-$ by unicity of the lift $f_-$ of $Pf_-$ with target~$y_-$,
        \item a vertical morphism $u_-\colon x_-\bulletarrow x'_-$ in $\dE$ to the vertical morphism in $\dgro_\dC \ddel_\dC(P)$ \[ (Pu_-,u_-)\colon (Px_-,x_-)\to (Px'_-,x'_-), \]
        \item a square $\alpha_-\colon \edgesquare{u_-}{f_-}{f'_-}{v_-}$ in $\dE$ to the square in  $\dgro_\dC \ddel_\dC(P)$
        \[ P\alpha_-\colon \edgesquare{(Pu_-,u_-)}{Pf_-}{Pf'_-}{(Pv_-,v_-)}, \]
        where $(P\alpha_-)^*(v_-)=u_-$ by unicity of the lift $\alpha_-$ of $P\alpha_-$ with target $v_-$. 
    \end{itemize}
\end{Construction}

\begin{Proposition}
\label{prop: eta is a 2-nat iso}
    The construction $\eta\colon\id_{\dFib}\Rightarrow\dgro_\dC\circ\ddel_\dC$ is a $2$-natural isomorphism.
\end{Proposition}

\begin{proof}
    First note that the double functor $\eta_P\colon \dE\to \dgro_\dC \ddel_\dC(P)$ admits as an inverse the double functor $\dgro_\dC \ddel_\dC(P)\to \dE$ given by projecting onto the second component on objects and vertical morphisms, and by picking the unique lift of horizontal morphisms and squares.

    Moreover, it follows easily from the descriptions in 
    \cref{lem: dgro ddel(F),lem: dgro ddel(A)} that the components $\eta_P$ are $2$-natural in $P$.
\end{proof}

\begin{proof}[Proof of \cref{thm: grothendieck equivalence}]
    The data $(\dgro_\dC,\ddel_\dC,\eta,\epsilon)$ from \cref{constr: dgroC,constr: ddel,prop: eta is a 2-nat iso,prop: epsilon: ddel dgro -> 1} provides the desired $2$-equivalence.
\end{proof}

As a consequence of \cref{lem: PC for 2-categories,thm: grothendieck equivalence}, when taking $\dC=\dH\CC$ with $\CC$ a $2$-category, we retrieve the Grothendieck $2$-equivalence from \cite[Theorem 5.1]{internal_GC} in the case of $\mathcal{V}=\Cat$.  

\begin{Corollary} \label{GC for 2-categories}
    Given a $2$-category $\CC$, the Grothendieck construction
    \[
    \textstyle\dgro_{\dH\CC}\colon [\CC^{\op},\Cat]\longrightarrow\dFib(\dH\CC)
    \]
    is a $2$-equivalence of $2$-categories, which is pseudo-natural in $\CC$.
\end{Corollary}

\section{Representation theorem for lax double presheaves}

In this last section, we turn to a representation theorem for double categories. While in the presheaf world a lax double functor is said to be represented by an object if it is isomorphic to a representable lax double functor at the same object, in the fibrational world we can reformulate this representation condition in terms of \emph{double terminal objects}. In \cref{subsec: double terminal}, we first introduce double terminal objects and show that a discrete double fibration is represented, i.e., it is isomorphic to a double slice, if and only if it has a double terminal object. Then, in \cref{subsec: representation}, we use this result to show that a lax double presheaf is represented by an object if and only if its Grothendieck construction has a double terminal object. This gives a nice criterion to test representability of a given lax double presheaf, and formulate universal properties for double categories.

\subsection{Double terminal objects}
\label{subsec: double terminal}

Let us fix a double category $\dC$. We start by recalling the definition of a double terminal object. 

\begin{Definition}
   An object $\hat{x}$ in $\dC$ is 
    \emph{double terminal} if the canonical projection $\dC/\hat{x}\to \dC$ is an isomorphism of double categories. 
\end{Definition}

\begin{Remark}
    \label{lem: double terminal objects}
    Unpacking the double isomorphism, an object~$\hat{x}$ in $\dC$ is double terminal if and only if the following conditions hold:
    \begin{enumerate}[leftmargin=1cm]
        \item for every object $x$ in $\dC$, there is a unique horizontal morphism $t_x\colon x\to \hat{x}$ in $\dC$,
        \item for every vertical morphism $u\colon x\bulletarrow x'$ in $\dC$, there is a unique square 
        $\tau_u\colon\edgesquare{u}{t_x}{t_{x'}}{e_{\hat{x}}}$ in $\dC$.
    \end{enumerate}
\end{Remark}

\begin{Notation}
     In what follows, for an object $\hat{x}$ in $\dC$ and a double functor $X\colon \dC^{\op}\to \dCat$, we abuse notation and denote the discrete double fibrations $\dC/\hat{x}\to \dC$ and $\pi_X\colon \dgro_\dC X\to \dC$ simply by their sources $\dC/\hat{x}$ and $\dgro_\dC X$.
\end{Notation}

We can translate the Yoneda lemma in the fibrational setting as follows. 

\begin{Theorem}
    Given a discrete double fibration 
    $P\colon\dE\to\dC$ and an object $\hat{x}$ in $\dC$, there is an isomorphism of categories
    \[ \Psi\colon \dFib(\dC)(\dC/{\hat{x}},P)\xrightarrow{\cong} P^{-1}\hat{x}
    \]
    which is $2$-natural in $\hat{x}$ in $\HH\dC$ and in $P$ in $\dFib(\dC)$.
\end{Theorem}

\begin{proof}
    We have the following $2$-natural isomorphisms 
    \begin{align*}
    \dFib(\dC)(\dC/{\hat{x}},P) &= \textstyle\dFib(\dC)(\dgro_\dC \dC(-,\hat{x}),P) & \text{\cref{prop: dgro of representable}} \\
    & \cong \P{C}(\dC(-,\hat{x}), \ddel_\dC(P)) & \textstyle \dgro_\dC\dashv \ddel_\dC \\
    & \cong \ddel_\dC(P)\hat{x}=P^{-1}\hat{x}, & \text{\cref{thm:Yoneda}}
    \end{align*}
    as desired.
\end{proof}

In particular, we are interested in the inverse of $\Psi$
\[ \Phi\colon P^{-1}\hat{x}\xrightarrow{\cong}\dFib(\dC)(\dC/{\hat{x}},P).
    \]
    We describe its action on objects explicitly. 

\begin{Lemma}
\label{prop: Phi(x) explicitly}
    Given a discrete double fibration 
    $P\colon\dE\to\dC$, an object $\hat{x}$ in $\dC$ and an object $\hat{x}_-$ in $P^{-1}\hat{x}$, then the double functor over $\dC$ 
    \[ \Phi(\hat{x}_-)\colon \dC/\hat{x}\to \dE \]
    can be described as follows. It sends 
    \begin{itemize}[leftmargin=1cm]
        \item an object $(x,g)$ in $\dC/\hat{x}$ to the object 
        $g^*\hat{x}_-$ in $\dE$, i.e., the source of the unique lift of the horizontal morphism $g\colon x\to \hat{x}$ with target $\hat{x}_-$,
        \item a horizontal morphism 
        $f\colon (x,g)\to (y,h)$ in $\dC/\hat{x}$ to the horizontal morphism in $\dE$ 
        \[ P^*f\colon g^*\hat{x}_-=f^*h^*\hat{x}_-\to h^*\hat{x}_-, \]
        i.e., the unique lift of the horizontal morphism $f\colon x\to y$ with target $h^*\hat{x}_-$,
        \item a vertical morphism
        $(u,\eta)\colon (x,g)\bulletarrow (x',g')$ in $\dC/\hat{x}$ to the vertical morphism in $\dE$ \[ \eta^*e_{\hat{x}_-}\colon g^*\hat{x}_-\bulletarrow g'^*\hat{x}_-, \]
        i.e., the source of the unique lift of the square $\eta\colon \edgesquare{u}{g}{g'}{e_{\hat{x}}}$ with target $e_{\hat{x}_-}$,  
        \item a square 
        $\alpha\colon\edgesquare{(u,\eta)}{f}{f'}{(v,\theta)}$ in $\dC/\hat{x}$ to the square in $\dE$
        \[ P^*\alpha\colon\edgesquare{ \eta^*e_{\hat{x}_-}=\alpha^*\theta^*e_{\hat{x}_-}}{P^*f}{P^*f'}{\theta^*e_{\hat{x}_-}},  \]
        i.e., the unique lift of the square $\alpha\colon\edgesquare{u}{f}{f'}{v}$ with target $\theta^*e_{\hat{x}_-}$.
    \end{itemize}
\end{Lemma}

\begin{proof}
    By construction, the functor $\Phi$ is given by the composite
    \[
    \begin{tikzcd}
        P^{-1}\hat{x}
        \arrow[r,"{\Phi_{\hat{x},\ddel_\dC(P)}}"]
        &[20pt]
        \P{C}(\dC(-,\hat{x}),\ddel_\dC(P))
        \arrow[r,"\dgro_\dC"]
        &
        \dFib(\dC)(\dC/\hat{x},\dgro_\dC\ddel_\dC(P))
        \arrow[r,"{(\eta_P^{-1})_*}"]
        &
        \dFib(\dC)(\dC/\hat{x},P).
    \end{tikzcd}
    \]
    and therefore $\Phi(\hat{x}_-)=\eta_P^{-1}(\dgro_\dC(\Phi_{\hat{x},\ddel_\dC(P)}(\hat{x}_-)))$.
    By \cref{constr:onobjects}, the horizontal transformation 
    \[
    \Phi_{\hat{x},\ddel_\dC(P)}(\hat{x}_-)\colon\dC(-,\hat{x})\Rightarrow\ddel_\dC(P)
    \]
    is defined as follows. 
    \begin{itemize}[leftmargin=1cm]
        \item Its component at an object $x$ in $\dC$ is the functor $\dC(x,\hat{x})\to\ddel_\dC(P)x=P^{-1}x$ which sends
        \begin{itemize}[leftmargin=1cm]
            \item a horizontal morphism $g\colon x\to\hat{x}$ in $\dC$ to the object 
            $\ddel_\dC(P)g(\hat{x}_-)=g^*\hat{x}_-$ in $P^{-1}x$, i.e., the source of the unique lift of $g$ with target $\hat{x}_-$,
            \item a globular square 
            $\eta\colon\edgesquare{e_x}{g}{g'}{e_{\hat{x}}}$ to the vertical morphism 
            $(\ddel_\dC(P)\eta)_{\hat{x}_-}=
            \eta^*e_{\hat{x}_-}$ in $P^{-1}x$, i.e., the source of the unique lift of $\eta$ with target $e_{\hat{x}_-}$. 
        \end{itemize}
        \item Its component at a vertical morphism 
        $u\colon x\bulletarrow x'$ in $\dC$ is the natural transformation 
        \[
        \Phi_{\hat{x},\ddel_\dC(P)}(\hat{x}_-)_u\colon
        \dC(u,\hat{x})\Rightarrow
        \ddel_\dC(P)u\circ(\Phi_{\hat{x},\ddel_\dC(P)}(\hat{x}_-)_x^{\op}\times
        \Phi_{\hat{x},\ddel_\dC(P)}(\hat{x}_-)_{x'}),
        \]
        whose component at an object 
        $(x\xrightarrow{g} \hat{x},x'\xrightarrow{g}\hat{x})$ of $\dC(x,\hat{x})\times\dC(x',\hat{x})$ is given by the map
        \[
        \dC(u,\hat{x})(g,g')\to
        \ddel_\dC(P)u(g^*\hat{x}_-,g'^*\hat{x}_-)=\{g^*\hat{x}_-\overset{u_-}{\bulletarrow}g'^*\hat{x}_-\mid Pu_-=u\},
        \]
        sending a square 
        $\eta\colon\edgesquare{u}{g}{g'}{e_{\hat{x}}}$ in $\dC$ to the vertical morphism
        $(\ddel_\dC(P)\eta)_{\hat{x}_-,\hat{x}_-}(1_{\hat{x}_-})=\eta^*e_{\hat{x}_-}$, i.e., the source of the unique lift of $\eta$ with target $e_{\hat{x}_-}$.
    \end{itemize}
    Next, using the description from \cref{def: dgro of horizontal transformations}, the double functor 
    \[ \textstyle \dgro_\dC \Phi_{\hat{x},\ddel_\dC(P)}(\hat{x}_-)\colon \dgro_\dC\dC(-,\hat{x})=\dC/\hat{x}\to\dgro_\dC\ddel_\dC(P)\]
    is given by sending
    \begin{itemize}[leftmargin=1cm]
        \item an object $(x,g)$ of $\dC/\hat{x}$ to the object 
        $(x,g^*\hat{x}_-)$ in $\dgro_\dC\ddel_\dC(P)$,
        \item a horizontal morphism 
        $f$ in $\dC/\hat{x}$ to the corresponding horizontal morphism $f$ in 
        $\dgro_\dC\ddel_\dC(P)$, 
        \item a vertical morphism
        $(u,\eta)\colon (x,g)\bulletarrow (x',g')$ in $\dC/\hat{x}$ to the vertical morphism in 
        $\dgro_\dC\ddel_\dC(P)$
        \[ (u,\eta^*e_{\hat{x}_-})\colon (x,g^*\hat{x}_-)\bulletarrow (x',g'^*\hat{x}_-) , \]
        \item a square $\alpha$ in $\dC/\hat{x}$ to the corresponding square $\alpha$ in 
        $\dgro_\dC\ddel_\dC(P)$.
    \end{itemize}
    Finally, applying the inverse $\eta^{-1}_P\colon \dgro_\dC\ddel_\dC(P)\to \dE$ of the unit as described in \cref{prop: eta is a 2-nat iso}, we can observe that the double functor
    $\Phi(\hat{x}_-)=\eta_P^{-1}(\dgro_\dC \Phi_{\hat{x},\ddel_\dC(P)}(\hat{x}_-))\colon \dC/\hat{x}\to\dE$ is 
   as described in the statement. 
\end{proof}

Finally, we can prove that double terminal objects detect representability. For this, we first prove the following result. 

\begin{Lemma}
\label{lem: P induces iso on slices}
    Given a discrete double fibration $P\colon\dE\to\dC$ and an object $\hat{x}_-$ in $\dE$ with $\hat{x}\coloneqq P(\hat{x}_-)$, then the induced double functor between slices 
    \[ P_{/\hat{x}_-}\colon\dE/\hat{x}_-\to\dC/\hat{x} \]
    is an isomorphism of double categories. 
\end{Lemma}

\begin{proof}
    By definition, the double functor $P_{/\hat{x}_-}$ fits into a commutative square in $\DblCat$
    \[
    \begin{tikzcd}
        \dE/\hat{x}_- \arrow[r, "P_{/\hat{x}_-}"] \arrow[d] & \dC/\hat{x}   \arrow[d] \\
        \dE\arrow[r, "P"']                               & \dC               
    \end{tikzcd}
    \]
    where the vertical canonical projections are discrete double fibrations by \cref{prop: pi_x is ddisc fib}. Hence $P_{/\hat{x}_-}$ is a morphism of the discrete double fibrations 
   over $\dC$. Therefore, to show that it is an isomorphism, by \cref{lem: iso of ddcis fib can be checked on Ver}, it suffices to show that the induced functor \[ \Ver_0(P_{/\hat{x}_-})\colon\Ver_0(\dE/\hat{x}_-)\to\Ver_0(\dC/\hat{x}) \]
   is an isomorphism of categories. 

    We construct an inverse of $\Ver_0(P_{/\hat{x}_-})$
    \[
    L\colon\Ver_0(\dC/\hat{x})\to\Ver_0(\dE/\hat{x}_-)
    \]
    as follows. It is the functor sending
    \begin{itemize}[leftmargin=1cm]
        \item an object $(x,g)$ of $\dC/\hat{x}$ to the object 
        $(g^*\hat{x}_-,P^*g)$ of $\dE/\hat{x}_-$, where $P^*g\colon g^*\hat{x}_-\to \hat{x}_-$ denotes the unique lift of the horizontal morphism $g\colon x\to \hat{x}$ in $\dC$,  
        \item a vertical morphism 
        $(u,\eta)\colon (x,g)\bulletarrow (x',g')$ in $\dC/\hat{x}$ to the vertical morphism in 
        $\dE/\hat{x}_-$
        \[ (\eta^*e_{\hat{x}_-},P^*\eta)\colon 
        (g^*\hat{x}_-,P^*g)\bulletarrow (g'^*\hat{x}_-,P^*g'), \]
        where $P^*\eta\colon \edgesquare{\eta^*e_{\hat{x}_-}}{P^*g}{P^*g'}{e_{\hat{x}_-}}$ denotes the unique lift of the square $\eta\colon \edgesquare{u}{g}{g'}{e_{\hat{x}}}$.
    \end{itemize}
    Clearly, the functor $L$ defines an inverse of $\Ver_0(P_{/\hat{x}_-})$ by unicity of lifts.
\end{proof}

\begin{Theorem}
\label{thm: Phi detects terminal objects}
    Given a discrete double fibration 
    $P\colon\dE\to\dC$ and an object $\hat{x}_-$ in $\dE$ with $\hat{x}\coloneqq P(\hat{x}_-)$, then the object $\hat{x}_-$ is double terminal in $\dE$ if and only if the double functor 
    \[ \Phi(\hat{x}_-)\colon \dC/\hat{x}\to \dE \]
    is an isomorphism of double categories.
\end{Theorem}

\begin{proof}
    We first show that the following triangle of double functors commutes.
    \[
    \begin{tikzcd}
        \dE/\hat{x}_- \arrow[r, ""] \arrow[d, "P_{/\hat{x}_-}"'] & \dE \\
        \dC/\hat{x} \arrow[ru, "\Phi(\hat{x}_-)"']                  &    
    \end{tikzcd}
    \]
    Unpacking the description of $\Phi(\hat{x}_-)$ from \cref{prop: Phi(x) explicitly} and using unicity of lifts, we note that the composition $\Phi(\hat{x}_-)\circ P_{/\hat{x}_-}$ sends
    \begin{itemize}[leftmargin=1cm]
        \item an object $(x_-,x_-\xrightarrow{g_-}\hat{x}_-)$ in 
        $\dE/\hat{x}_-$ to the object in $\dE$
        \[
        \Phi(\hat{x}_-)(P_{/\hat{x}_-}(x_-,g_-))
        =
        \Phi(\hat{x}_-)(Px_-,Pg_-)=(Pg_-)^*\hat{x}_-=x_-, \]
        \item a horizontal morphism 
        $f_-\colon (x_-,g_-)\to (y_-,h_-)$ in $\dE/\hat{x}_-$ to the horizontal morphism in $\dE$
        \[ \Phi(\hat{x}_-)(P_{/\hat{x}_-}(f_-))
            =
            \Phi(\hat{x}_-)(Pf_-)
            =
            P^*(Pf_-)=f_-, \]
        \item a vertical morphism 
        $(u_-,\eta_-)\colon (x_-,g_-)\bulletarrow (x'_-,g'_-)$ in $\dE/\hat{x}_-$ to the vertical morphism in $\dE$
        \[\Phi(\hat{x}_-)(P_{/\hat{x}_-}(u_-,\eta_-))
            =
            \Phi(\hat{x}_-)(Pu_-,P\eta_-)
            =
            (P\eta_-)^*e_{\hat{x}_-}
            =
            u_-,\]
        \item a square 
        $\alpha_-\colon\edgesquare{(u_-,\eta_-)}{f_-}{f'_-}{(v_-,\theta_-)}$ to the square in $\dE$
        \[ \Phi(\hat{x}_-)(P_{/\hat{x}_-}(\alpha_-))
            =
            \Phi(\hat{x}_-)(P\alpha_-)
            =
            P^*(P\alpha_-)
            =
            \alpha_-. \]
    \end{itemize}
    Hence the composite $\Phi(\hat{x}_-)\circ P_{/\hat{x}_-}$ coincides with the canonical projection $\dE/\hat{x}\to \dE$, as desired. 

    Now, by \cref{lem: P induces iso on slices}, the double functor $P_{/\hat{x}_-}$ is an isomorphism of double categories. Therefore, by $2$-out-of-$3$ for isomorphisms, the double functor $\Phi(\hat{x}_-)\colon \dC/\hat{x}\to \dE$ is an isomorphism if and only if the canonical projection $\dE/\hat{x}_-\to \dE$ is an isomorphism. However, by definition, this means that the object $\hat{x}_-$ is double terminal in $\dE$.
\end{proof}

\subsection{Representation theorem}
\label{subsec: representation}

Finally, we state our representation theorem. For this, we first introduce the notion of a represented lax double presheaf. 

\begin{Definition}
\label{def: repr. double presheaf}
    A lax double presheaf $X\colon\dC^{\op}\to\dCat$ is \emph{represented by an object $\hat{x}$ in $\dC$} if there is an object $\hat{x}_-$ in $X\hat{x}$ such that the induced horizontal transformation 
    \[ \Phi_{\hat{x},X}(\hat{x}_-)\colon\dC(-,\hat{x})\Rightarrow X \]
    from \cref{constr:onobjects} is invertible. In this case, we also say that $X$ is \emph{represented by $(\hat{x},\hat{x}_-)$}. 
\end{Definition}

\begin{Theorem}
\index{representability theorem}
\label{thm: representability theorem}
    Given a lax double presheaf $X\colon\dC^{\op}\to\dCat$ and objects $\hat{x}$ in $\dC$ and $\hat{x}_-$ in~$X\hat{x}$, then $X$ is represented by $(\hat{x},\hat{x}_-)$ if and only if the object $(\hat{x},\hat{x}_-)$ is double terminal in the Grothendieck construction $\dgro_\dC X$.
\end{Theorem}

\begin{proof}
    By definition, the lax double presheaf $X$ is represented by $(\hat{x},\hat{x}_-)$ if and only if the induced horizontal transformation 
    \[ \Phi_{\hat{x},X}(\hat{x}_-)\colon\dC(-,\hat{x})\Rightarrow X\]
    is invertible. By applying the $2$-equivalence $\dgro_\dC\colon \P{C}\to \dFib(\dC)$ from \cref{thm: grothendieck equivalence}, this is the case if and only if the induced morphism of discrete double fibrations over $\dC$
    \[ \textstyle\dgro_\dC\Phi_{\hat{x},X}(\hat{x}_-)\colon\dgro_\dC\dC(-,\hat{x})=\dC/\hat{x}\to\dgro_\dC X \]
    is an isomorphism. Noticing that $\Phi(\hat{x},\hat{x}_-)=\dgro_\dC\Phi_{\hat{x},X}(\hat{x}_-)$ using \cref{prop: Phi(x) explicitly}, it follows from \cref{thm: Phi detects terminal objects} that this is equivalent to the object $(\hat{x},\hat{x}_-)$ being double terminal in $\dgro_\dC X$, as desired.  
\end{proof}

As a consequence, we get that the double slice always has a double terminal object given by the identity.  

\begin{Corollary}
\label{Cor: dbl terminal object of C/x}
    The object $(\hat{x},1_{\hat{x}})$ is double terminal in $\dC/\hat{x}$. 
\end{Corollary}

\begin{proof}
    As $\Phi_{\hat{x},\dC(-,\hat{x})}(1_{\hat{x}})$ is the identity horizontal transformation $\dC(-,\hat{x})\Rightarrow\dC(-,\hat{x})$ by functorality of $\Phi_{\hat{x},\dC(-,\hat{x})}$, the lax double presheaf $\dC(-,\hat{x})$ is represented by $(\hat{x},1_{\hat{x}})$.
    Then, by \cref{thm: representability theorem}, the object $(\hat{x},1_{\hat{x}})$ is double terminal in the double category
    $\dgro_\dC\dC(-,\hat{x})=\dC/\hat{x}$, where the last equality holds by \cref{prop: dgro of representable}.
\end{proof}

Finally, as a consequence of \cref{lem: PC for 2-categories,thm: representability theorem}, when taking $\dC=\dH\CC$ with $\CC$ a $2$-category, we retrieve the representation theorem for $2$-presheaves from 
\cite[Theorem 6.12]{internal_GC} in the case of $\mathcal{V}=\Cat$. This is a stricter version of \cite[Theorem 6.8]{clingmanmoser2022limits}.

\begin{Corollary} \label{Representation for 2-categories}
    Given a $2$-presheaf $X\colon\CC^{\op}\to\Cat$ and objects $\hat{x}$ in $\CC$ and $\hat{x}_-$ in~$X\hat{x}$, then $X$ is represented by $(\hat{x},\hat{x}_-)$ if and only if the object $(\hat{x},\hat{x}_-)$ is double terminal in the Grothendieck construction $\dgro_{\dH\CC} X$.
\end{Corollary}

\bibliographystyle{alpha}
\bibliography{literature}

\begin{thebibliography}{CLPS22}

\bibitem[B{\'e}n73]{Benabou}
Jean B{\'e}nabou.
\newblock Les distributeurs.
\newblock {\em Inst. de Math. Pure et Appliqu\'ee}, 33, 1973.

\bibitem[Buc14]{Buckley}
Mitchell Buckley.
\newblock Fibred 2-categories and bicategories.
\newblock {\em J. Pure Appl. Algebra}, 218(6):1034--1074, 2014.

\bibitem[CLPS22]{Double_Fibrations}
G.~S.~H. Cruttwell, M.~J. Lambert, D.~A. Pronk, and M.~Szyld.
\newblock Double fibrations.
\newblock {\em Theory Appl. Categ.}, 38:Paper No. 35, 1326--1394, 2022.

\bibitem[cM22a]{2Limits}
tslil clingman and Lyne Moser.
\newblock 2-limits and 2-terminal objects are too different.
\newblock {\em Appl. Categ. Structures}, 30(6):1283--1304, 2022.

\bibitem[cM22b]{clingmanmoser2022limits}
tslil clingman and Lyne Moser.
\newblock Bi-initial objects and bi-representations are not so different.
\newblock {\em Cah. Topol. G\'{e}om. Diff\'{e}r. Cat\'{e}g.}, 63(3):259--330,
  2022.

\bibitem[CS10]{cruttwellshulman}
G.~S.~H. Cruttwell and Michael~A. Shulman.
\newblock A unified framework for generalized multicategories.
\newblock {\em Theory Appl. Categ.}, 24:Paper No. 21, 580--655, 2010.

\bibitem[Fem23]{bifunctor}
Bojana Femi\'{c}.
\newblock Bifunctor theorem and strictification tensor product for double
  categories with lax double functors.
\newblock {\em Theory Appl. Categ.}, 39:Paper No. 29, 824--873, 2023.

\bibitem[FGK12]{fioregambinokock}
Thomas~M. Fiore, Nicola Gambino, and Joachim Kock.
\newblock Double adjunctions and free monads.
\newblock {\em Cah. Topol. G\'{e}om. Diff\'{e}r. Cat\'{e}g.}, 53(4):242--306,
  2012.

\bibitem[FPP08]{Model_Structure}
Thomas~M. Fiore, Simona Paoli, and Dorette Pronk.
\newblock Model structures on the category of small double categories.
\newblock {\em Algebr. Geom. Topol.}, 8(4):1855--1959, 2008.

\bibitem[GHL22]{Gagna_Harpaz_Lanari}
Andrea Gagna, Yonatan Harpaz, and Edoardo Lanari.
\newblock Bilimits are bifinal objects.
\newblock {\em J. Pure Appl. Algebra}, 226(12):Paper No. 107137, 36, 2022.

\bibitem[GP19]{persistentflexible}
Marco Grandis and Robert Par\'{e}.
\newblock Persistent double limits and flexible weighted limits.
\newblock available at \url{https://www.mscs.dal.ca/~pare/DblPrs2.pdf}, 2019.

\bibitem[Gra19]{Grandis2019}
Marco Grandis.
\newblock {\em Higher Dimensional Categories: From Double To Multiple
  Categories -}.
\newblock World Scientific, Singapur, 2019.

\bibitem[Joh02]{elephant}
Peter~T. Johnstone.
\newblock {\em Sketches of an elephant: a topos theory compendium. {V}ol. 1},
  volume~43 of {\em Oxford Logic Guides}.
\newblock The Clarendon Press, Oxford University Press, New York, 2002.

\bibitem[JY21]{2Dim_Categories}
Niles Johnson and Donald Yau.
\newblock {\em 2-dimensional categories}.
\newblock Oxford University Press, Oxford, 2021.

\bibitem[Lam21]{discrete_double_fibrations}
Michael Lambert.
\newblock Discrete double fibrations.
\newblock {\em Theory Appl. Categ.}, 37:671--708, 2021.

\bibitem[LR20]{CatFibrations}
Fosco Loregian and Emily Riehl.
\newblock Categorical notions of fibration.
\newblock {\em Expo. Math.}, 38(4):496--514, 2020.

\bibitem[MSV23]{internal_GC}
Lyne Moser, Maru Sarazola, and Paula Verdugo.
\newblock Internal {G}rothendieck construction for enriched categories.
\newblock \href{https://arxiv.org/abs/2308.14455}{arXiv:2308.14455}, 2023.

\bibitem[Par11]{Yoneda_double_cats}
Robert Par\'{e}.
\newblock Yoneda theory for double categories.
\newblock {\em Theory Appl. Categ.}, 25:No. 17, 436--489, 2011.

\bibitem[Rie14]{riehl2014categoricalhomotopy}
Emily Riehl.
\newblock {\em Categorical homotopy theory}, volume~24 of {\em New Mathematical
  Monographs}.
\newblock Cambridge University Press, Cambridge, 2014.

\bibitem[Rie17]{Context}
Emily Riehl.
\newblock {\em Category Theory in Context -}.
\newblock Courier Dover Publications, Mineola, New York, 2017.

\bibitem[RV22]{Elements}
Emily Riehl and Dominic Verity.
\newblock {\em Elements of {$\infty$}-category theory}, volume 194 of {\em
  Cambridge Studies in Advanced Mathematics}.
\newblock Cambridge University Press, Cambridge, 2022.

\end{thebibliography}
\end{document}